\documentclass[11 pt]{report}
\usepackage{amsmath} 
\usepackage{amssymb} 
\usepackage{makeidx} 
\usepackage{boxedminipage}
\usepackage{amssymb, amsfonts, amsthm, setspace, indentfirst, enumerate}
\theoremstyle{definition}
\usepackage{fancyhdr}
\usepackage{csquotes}
\usepackage{tikz}
\usetikzlibrary{calc}
\newcommand\HRule{\rule{\textwidth}{1pt}}
\usepackage{graphicx}
\usepackage{multicol}
\usepackage{epstopdf}
\graphicspath{{G:\Back up\My Folders\My Work\Thesis\My thesis\complied form}}
\usepackage{hyperref}
\newtheorem{theo}{Theorem}[section]
\newtheorem*{theo A}{Theorem}
\newtheorem*{theo 2.A}{Theorem 2.A}
\newtheorem*{theo 3.A}{Theorem 3.A}
\newtheorem*{theo 4.A}{Theorem 4.A}
\newtheorem*{theo 5.A}{Theorem 5.A}
\newtheorem*{theo 6.A}{Theorem 6.A}
\newtheorem*{theo 7.A}{Theorem 7.A}
\newtheorem*{theo 8.A}{Theorem 8.A}
\newtheorem*{theo 9.A}{Theorem 9.A}
\newtheorem*{theo 2.B}{Theorem 2.B}
\newtheorem*{theo 3.B}{Theorem 3.B}
\newtheorem*{theo 4.B}{Theorem 4.B}
\newtheorem*{theo 5.B}{Theorem 5.B}
\newtheorem*{theo 6.B}{Theorem 6.B}
\newtheorem*{theo 7.B}{Theorem 7.B}
\newtheorem*{theo 8.B}{Theorem 8.B}
\newtheorem*{theo 9.B}{Theorem 9.B}
\newtheorem*{theo 2.C}{Theorem 2.C}
\newtheorem*{theo 3.C}{Theorem 3.C}
\newtheorem*{theo 4.C}{Theorem 4.C}
\newtheorem*{theo 5.C}{Theorem 5.C}
\newtheorem*{theo 6.C}{Theorem 6.C}
\newtheorem*{theo 7.C}{Theorem 7.C}
\newtheorem*{theo 8.C}{Theorem 8.C}
\newtheorem*{theo 9.C}{Theorem 9.C}
\newtheorem*{theo 2.D}{Theorem 2.D}
\newtheorem*{theo 3.D}{Theorem 3.D}
\newtheorem*{theo 4.D}{Theorem 4.D}
\newtheorem*{theo 5.D}{Theorem 5.D}
\newtheorem*{theo 6.D}{Theorem 6.D}
\newtheorem*{theo 7.D}{Theorem 7.D}
\newtheorem*{theo 8.D}{Theorem 8.D}
\newtheorem*{theo 9.D}{Theorem 9.D}
\newtheorem*{theo 2.E}{Theorem 2.E}
\newtheorem*{theo 3.E}{Theorem 3.E}
\newtheorem*{theo 4.E}{Theorem 4.E}
\newtheorem*{theo 5.E}{Theorem 5.E}
\newtheorem*{theo 6.E}{Theorem 6.E}
\newtheorem*{theo 7.E}{Theorem 7.E}
\newtheorem*{theo 8.E}{Theorem 8.E}
\newtheorem*{theo 9.E}{Theorem 9.E}
\newtheorem*{theo 2.F}{Theorem 2.F}
\newtheorem*{theo 3.F}{Theorem 3.F}
\newtheorem*{theo 4.F}{Theorem 4.F}
\newtheorem*{theo 5.F}{Theorem 5.F}
\newtheorem*{theo 6.F}{Theorem 6.F}
\newtheorem*{theo 7.F}{Theorem 7.F}
\newtheorem*{theo 8.F}{Theorem 8.F}
\newtheorem*{theo 9.F}{Theorem 9.F}
\newtheorem*{theo 2.G}{Theorem 2.G}
\newtheorem*{theo 3.G}{Theorem 3.G}
\newtheorem*{theo 4.G}{Theorem 4.G}
\newtheorem*{theo 5.G}{Theorem 5.G}
\newtheorem*{theo 6.G}{Theorem 6.G}
\newtheorem*{theo 7.G}{Theorem 7.G}
\newtheorem*{theo 8.G}{Theorem 8.G}
\newtheorem*{theo 9.G}{Theorem 9.G}
\newtheorem*{theo 2.H}{Theorem 2.H}
\newtheorem*{theo 3.H}{Theorem 3.H}
\newtheorem*{theo 4.H}{Theorem 4.H}
\newtheorem*{theo 5.H}{Theorem 5.H}
\newtheorem*{theo 6.H}{Theorem 6.H}
\newtheorem*{theo 7.H}{Theorem 7.H}
\newtheorem*{theo 8.H}{Theorem 8.H}
\newtheorem*{theo 9.H}{Theorem 9.H}
\newtheorem*{theo 2.I}{Theorem 2.I}
\newtheorem*{theo 2.J}{Theorem 2.J}
\newtheorem{Theorem }{Theorem }
\newtheorem*{theo C}{Theorem C}
\newtheorem*{theo D}{Theorem D}
\newtheorem*{theo E}{Theorem E}
\newtheorem*{theo F}{Theorem F}
\newtheorem*{theo G}{Theorem G}
\newtheorem*{theo H}{Theorem H}
\newtheorem{defi}{Definition}[section]
\newtheorem*{defi A}{Definition}
\newtheorem*{exm A}{Example}
\newtheorem{lem}{Lemma}[section]
\newtheorem{cor}{Corollary}[section]
\newtheorem{rem}{Remark}[section]
\newtheorem{exm}{Example}[section]
\newtheorem{ques}{Question}[section]
\newcommand{\ol}{\overline}
\newcommand{\be}{\begin{equation}}
\newcommand{\ee}{\end{equation}}
\newcommand{\bes}{\begin{equation*}}
\newcommand{\ees}{\end{equation*}}
\newcommand{\bea}{\begin{eqnarray}}
\newcommand{\eea}{\end{eqnarray}}
\newcommand{\eeas}{\end{eqnarray*}}
\newcommand{\beas}{\begin{eqnarray*}}
\newcommand{\lra}{\longrightarrow}

\numberwithin{equation}{section}
\begin{document}
\renewcommand{\familydefault}{\sfdefault}
\sffamily
\fancyhead{}
\chead{\leftmark}
\pagestyle{fancy}
\begin{titlepage}

\begin{tikzpicture}[remember picture, overlay]
\end{tikzpicture}
\begin{center}
\begin{minipage}{0.45\textwidth}
\begin{flushleft}
{\bfseries \textsc{\Large{University of Kalyani, India}}}
\vspace{.35 cm}
\\
Faculty of Science
\\
Department of Mathematics
\end{flushleft}
\end{minipage}
\begin{minipage}{0.45\textwidth}
\begin{flushleft}
\end{flushleft}
\end{minipage}
\begin{minipage}{0.45\textwidth}
\begin{flushright}
\end{flushright}
\end{minipage}
\end{center}
$~~~~~~~~~~~~~~~~~~~~~$
\\
\vspace{4 cm}
\begin{center}
{\Large{\textsc{\textbf{SOME ASPECTS OF UNIQUENESS THEORY OF}}}}
\\
\vspace{.15 cm}
{\Large{\textsc{\textbf{ENTIRE AND MEROMORPHIC}}}}
\\
\vspace{.15 cm}
{\Large{\textsc{\textbf{ FUNCTIONS}}}}
\end{center}
$~~~~~~~~~~~~~~~~~~~~~$
\vspace{0.25 cm}
\begin{center}
{\bfseries \textsc{\Large{Bikash Chakraborty, M.Sc.}}}
\end{center}
$~~~~~~~~~~~~~~~~~~~~~~~~~~~~~~~$
\\
\vspace{4 cm}
\begin{center}
\begin{minipage}{0.45\textwidth}
\begin{flushleft}
\end{flushleft}
\end{minipage}
\begin{minipage}{0.45\textwidth}
\begin{flushright}
A thesis submitted for the degree of Doctor of Philosophy in Science (Mathematics)\\
\vspace{0.35 cm}
Supervisor:\\Prof. Abhijit Banerjee
\end{flushright}
\end{minipage}
\end{center}
$~~~~~~~~~~~~~~~~~~~~~~~~~~~~~~~$
\\
\vspace{2 cm}
\begin{center}
{\large{Summer, 2017}}
\end{center}
\end{titlepage}
\newpage
\begin{titlepage}

\begin{tikzpicture}[remember picture, overlay]
\end{tikzpicture}

\begin{center}
\vspace{6cm}
\HRule \\[0.4cm]
{  \bfseries \textsc{\large{To\\\vspace{.3cm} my parents\\
\vspace{.5cm}
Sri Bidyut Chakraborty and Smt. Pramila Chakraborty\\
\vspace{.5cm}
for their great interest in my education} }}\\[0.4cm]
\HRule \\[3cm]
\end{center}
\end{titlepage}
\newpage
\begin{titlepage}

\begin{tikzpicture}[remember picture, overlay]
  \draw[line width = 4pt] ($(current page.north west) + (1in,-1in)$) rectangle ($(current page.south east) + (-1in,1in)$);
\end{tikzpicture}
\vspace{6cm}
{\bfseries \textsc{\LARGE{\begin{figure}[!htb]
\centering
\includegraphics[scale=.2]{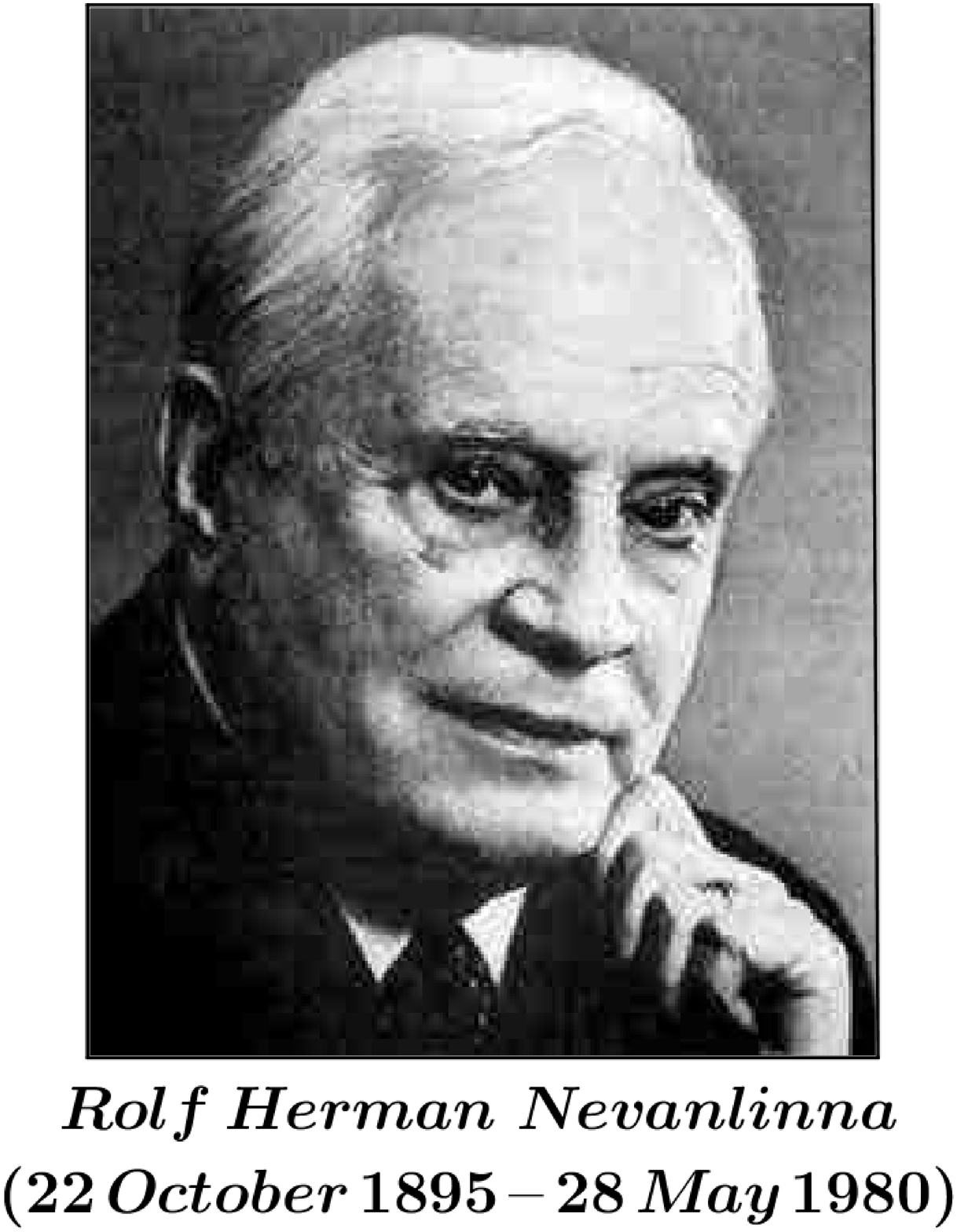}
\end{figure}
\\
\enquote{Rolf Nevanlinna, creator of the modern theory of meromorphic functions} $-$ Hayman, W. K.}}}\\[0.4cm]
\end{titlepage}
\newpage
\begin{titlepage}

\begin{figure}[!htb]
\begin{center}
\includegraphics[scale=.6]{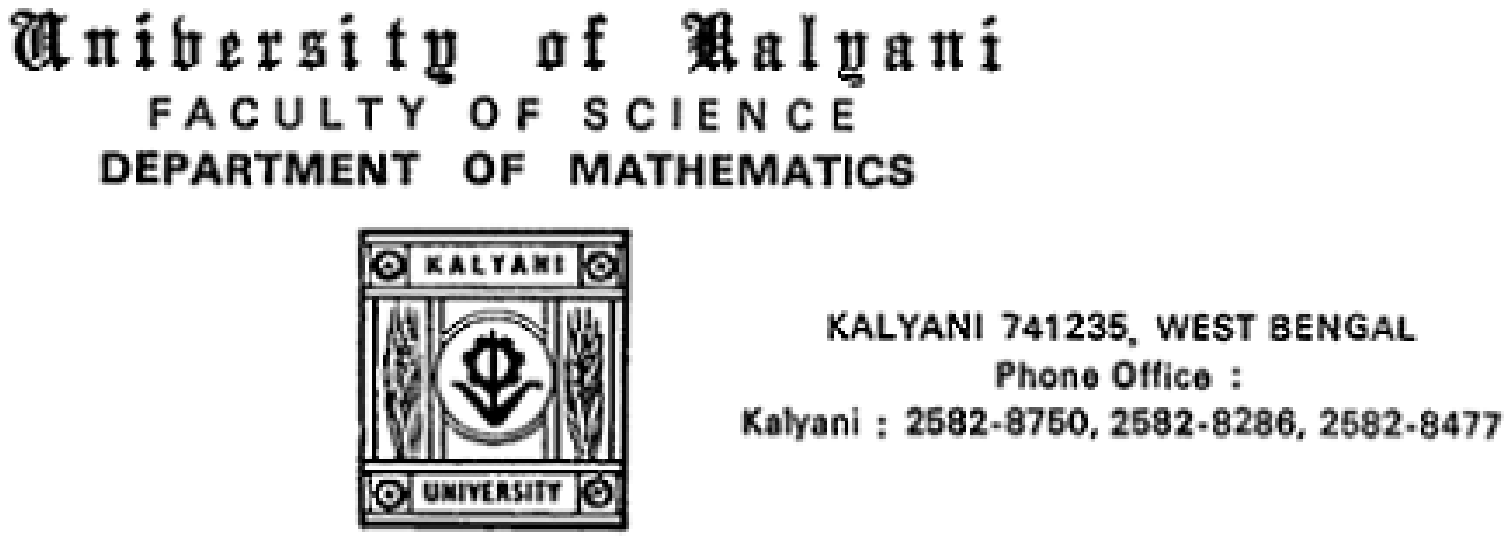}
\end{center}
\end{figure}
\begin{center}
\HRule
\vspace{2cm}
\large{\textbf{\underline{Certificate from Supervisor}}}
\end{center}
\vspace{1cm}
~~~This is to certify that the research works done in the dissertation entitled \enquote{\textsc{Some Aspects Of Uniqueness Theory Of Entire And Meromorphic Functions}}, have been carried out at Department of Mathematics, University of Kalyani, India-741235 by Bikash Chakraborty, M.Sc. under my supervision.\par
Mr. Bikash Chakraborty has joined in the Department on $05^{\text{th}}$ February, $2014$ as a DST-INSPIRE Fellow and followed the rules and regulations as laid down by University of Kalyani for the fulfillment of requirements for the degree of Doctor of Philosophy in Science.\par
As far as I know, no part of his thesis was submitted elsewhere for any degree, diploma etc. In my opinion the thesis submitted by Mr. Bikash Chakraborty is worthy of consideration for the Ph.D. degree.
\vspace{2 cm}
\begin{center}
\begin{minipage}{0.45\textwidth}
\begin{flushleft} \large
\textbf{Date:}\\
\textbf{Department of Mathematics}\\
\textbf{University of Kalyani}
\end{flushleft}
\end{minipage}
\begin{minipage}{0.45\textwidth}
\begin{flushright} \large
\textbf{(Abhijit Banerjee)}
\end{flushright}
\end{minipage}
\end{center}
\end{titlepage}
\newpage
\begin{titlepage}
$~~~~~~~~~~~~~$
\begin{center}
\vspace{2cm}
\large{\textbf{\huge{University of Kalyani}}}
\end{center}
\begin{figure}[!htb]
\begin{center}
\includegraphics[scale=.3]{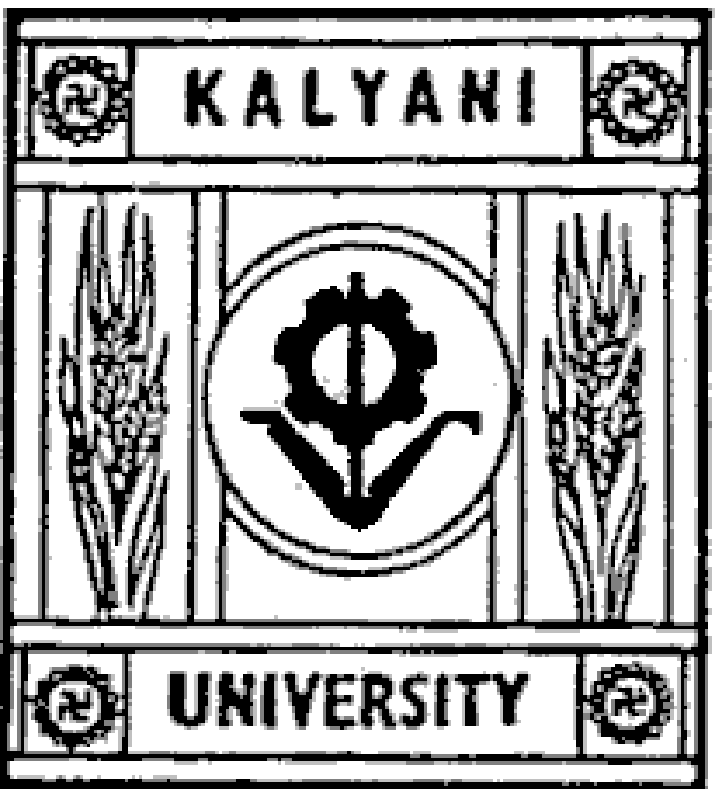}
\end{center}
\end{figure}
\begin{center}
\large{\textbf{\underline{Declaration for non-commitment of plagiarism}}}
\end{center}
\vspace{1cm}
~~~I, Bikash Chakraborty, a student of Ph.D. of the Department of Mathematics, have submitted a thesis in partial fulfillment of the requirements to obtain the above noted degree. I declare that I have not committed plagiarism in any form or violated copyright while writing the thesis, and have acknowledged the sources and/or the credit of other authors wherever applicable. If subsequently it is found that I have committed plagiarism or violated copyright, then the University authority has full right to cancel/ reject/ revoke my thesis / degree.
\vspace{2 cm}
\begin{center}
\begin{minipage}{0.45\textwidth}
\begin{flushleft} \large
\textbf{Date:}\\
\textbf{Department of Mathematics}\\
\textbf{University of Kalyani}
\end{flushleft}
\end{minipage}
\begin{minipage}{0.45\textwidth}
\begin{flushright} \large
\textbf{(Bikash Chakraborty)}
\end{flushright}
\end{minipage}

\end{center}
\end{titlepage}
\newpage
\begin{center}
\large{\textbf{\huge{Acknowledgement~~~~~~~~~~~~~~~~~~~~~~~~~~~~~~~~}}}
\end{center}
\thispagestyle{empty}
\vspace{1cm}
~~~This is my pleasure and sincere duty to acknowledge my indebtedness to all who assisted, guided and contributed their time and support to make this thesis possible.\par
I would like to express my deep gratitude to my supervisor, Professor Abhijit Banerjee, Department of Mathematics, University of Kalyani, India-741235, for his
constant help and guidance throughout my time as his student. I thank him also for all the encouragement and inspiration that he has provided, and for the tremendous generosity with which he has shared his ideas.\par
This research work would not have been possible without the financial support of Department of Science and Technology, Govt.of India.  I am thankful to Department of Science and Technology, Govt. of India for the financial support (DST/INSPIRE Fellowship/2014/IF140903) towards my Ph.D. programme.\par
I thank all the referees of various Journals in which some of the contents of my thesis have been published from time to time. I would also like to thank the authors of various papers and books which I have been consulted during the preparation of the thesis.\par
I am also grateful to all the Teachers, Research Scholars and Non-teaching staffs of Department of Mathematics, University of Kalyani, India for their hearty co-operation.\par
I shall always be grateful to all of my School, College and University days Teachers, specially Dipankar Das, Sanjib Mitra, Somnath Bandyopadhyay, Amalendu Ghosh, Krishnendu Dutta, Abhijit Banerjee, Dibyendu De, Indrajit Lahiri who have taught me about the fascinating world of the mathematics.\par
My sincere respect and gratefulness are also accorded to Swami Kamalasthananda Maharaj and Swami Vedanuragananda Maharaj of  Ramakrishna Mission Vivekananda Centenary College, Rahara, India-700118.\par
I always remember all my beloved friends from childhood to university level who has been superb all through the stages of my study.\par
I am grateful to my close friend and younger brother Mr. Sagar Chakraborty for his inspiration and loving support during my research work.\par
Finally, my enormous debt of gratitude goes to my loving parents Sri Bidyut Chakraborty and Smt. Pramila Chakraborty,  without whose motivation, support, endless patience and encouragement this thesis would never have been possible.\par
Lastly, I declare that this thesis represents my own work based on the suggestions of my supervisor
Professor Abhijit Banerjee. All the works are done under the supervision of him during the
period $2014-2017$ for the degree of Doctor of Philosophy at the University of Kalyani, India. The work submitted has not been previously included in any thesis, dissertation or report submitted to any institution for a degree, diploma or other qualification as per my knowledge.\par
I alone remain responsible for the contents of the thesis, including any possible errors and shortcomings.
\vspace{1 cm}
\begin{center}
\begin{minipage}{0.45\textwidth}
\begin{flushleft} \large
\textbf{Summer, 2017}
\end{flushleft}
\end{minipage}
\begin{minipage}{0.45\textwidth}
\begin{flushright} \large
\textbf{Bikash Chakraborty}
\end{flushright}
\end{minipage}

\end{center}
\newpage
\pagenumbering{roman}
\tableofcontents \vspace{.4in}
\begin{itemize}
\item \textbf{Bibliography} \hfill \textbf{96} \newline
\end{itemize}
\newpage
\chapter*{\textbf{Preface}}
\addcontentsline{toc}{chapter}
{Preface}
\fancyhead[r]{Preface}
\fancyhead[c]{}
\fancyhead[l]{}
The subject of this thesis is the uniqueness theory of meromorphic functions and it is devoted to problems concerning Br\"{u}ck Conjecture, Set sharing and related topics. The tool, we used in our discussions is classical value distribution theory of meromorphic functions, which is one of the milestones of complex analysis during the last century.\par
Throughout the thesis, we shall denote by $\mathbb{C}$ the set of all complex number and $\mathbb{\ol C}=\mathbb{C}\cup\{\infty\}$. Also we denote by $\mathbb{N}$ and $\mathbb{Z}$ the set of all natural numbers and the set of all integers respectively.\par
In our thesis, by \emph{Theorem $i.j.k.$}, we mean that $k$-th theorem in $j$-th section of the $i$-th chapter. Similar conventions are followed in definitions, examples, remarks etc. cases also.\par
This dissertation contains nine chapters. Also at the end of the thesis, we added one section containing the list of publications and bibliography.\par
Chapter 1 deals with the Nevanlinna theory (\cite{bd1, 2, bd2, bd3, bd4, 5, 1234}). It contains some basic definitions, notations, estimations which we used vividly throughout our journey.\par
Chapters 2, 3 and 4 are devoted to the subject on sharing values between meromorphic functions and their derivatives and more general differential expressions as well. In 1996, in connection to the value sharing of an entire functions together with its derivative, R. Br\"{u}ck (\cite{br3}) proposed a conjecture. Based on his conjecture, there were several generalizations and extensions, e.g. (\cite{zl1}, \cite{zl2a}, \cite{zl2}, \cite{br8}, \cite{zl5a}, \cite{br8c}, \cite{br8a}, \cite{br8b}, \cite{br10}, \cite{br11}, \cite{br12}, \cite{br14}, \cite{br15}, \cite{br16}). In chapters 2, 3 and 4, we studied Br\"{u}ck conjecture in more general settings in different aspects.\par
Chapters 5 and 6 introduce the idea of unique range sets, set sharing problems (\cite{am11}, \cite{am1}, \cite{ami1}, \cite{am1e}, \cite{am3}, \cite{am2}, \cite{am5}, \cite{am15}) in different aspects and its relevant concepts like uniqueness polynomial (\cite{am12}, \cite{am6.1}), strong uniqueness polynomials (\cite{am1a}), critically injective polynomials (\cite{am2}, \cite{am2a}) etc. We investigated some new class of strong uniqueness polynomials satisfying Fujimoto's condition and generating unique range sets. We have also established some sufficient conditions for uniqueness polynomials to be strong uniqueness polynomials.\par
Chapter 7 is also concerned with unique range sets corresponding to the derivatives of two meromorphic functions. Here we have also introduced some new unique range sets.\par
The shared value problems relative to a meromorphic function $f$ and its higher ordered derivative have been widely studied as subtopic of the uniqueness theory in chapters two, three and four. Changing that flavor, in Chapters 8 and 9, we consider uniqueness problem of a meromorphic function with its higher ordered derivative under aegis of set sharing to connect the Br\"{u}ck conjecture with Gross' problem.\par
Almost all the results have already been published in the form of research papers in different journals of international repute.
\newpage
\pagenumbering{arabic}
\chapter{Basic Nevanlinna Theory}
\fancyhead[l]{Chapter 1}
\fancyhead[r]{Basic Nevanlinna Theory}
\fancyhead[c]{}
\section{Introduction}
\par
The value distribution theory of  meromorphic functions is one of the milestones of complex analysis during the last century. This theory was greatly developed by the finish Mathematician Rolf Nevanlinna (\cite{5}) during the 1920's. In both its scope and its power, his approach greatly surpasses previous results, and in his honor the field is now also known as \textbf{Nevanlinna  theory}.\par
The aim of this chapter is to describe selected parts of the Nevanlinna Theory. Proofs will not be reproduced here, rather we shall indicate
where they may be found in the literature. The definitive reference for this section is Hayman's monograph (\cite{2}). For a short survey on Nevanlinna  theory one can also go thorough the article (\cite{1}).\par
Let $\mathbb{C}$ be the complex plane. A function $f$ is said to be \emph{analytic} in a domain $D\subset\mathbb{C}$ if $f'(z)$ exists finitely at every point of
$D$. The function $f$ is said to be analytic at a point $z=z_{0}$ if there exists a neighborhood of the point $z_{0}$ in which f is analytic.\par
A point $z=z_{0}\in \mathbb{C}$ is called a \emph{regular point} for a complex-valued function $f$ if $f$ is analytic at $z_{0}$. A point $z_{0}$ is called a \emph{singular point} or a \emph{singularity} of $f$, if $f$ is not analytic at $z_{0}$ but every neighborhood of $z_{0}$ contains atleast one point at which $f$ is analytic.\par
A singular point $z_{0}$ of $f$ is called an \emph{isolated singularity} of $f$ if $f$ is defined and analytic in some deleted neighborhood of $z_{0}$. Otherwise, it is called \emph{non-isolated singularity} of $f$.\par
Let $z = \alpha$ be an isolated singularity of an analytic function $f$. Then in some deleted neighborhood
of $z =\alpha$, $f$ can be expanded in the following form:
$$f(z)=\sum_{n=0}^{\infty}a_{n}(z-\alpha)^{n}+\sum_{n=1}^{\infty}b_{n}(z-\alpha)^{-n}$$
This series expansion of $f$ in negative and non-negative powers of $(z-\alpha)$ is called \emph{Laurent series} expansion of $f$. If, in particular, $f$ is analytic at $z =\alpha$,  then $b_{n} = 0$ for $n = 1,2,3,\ldots$ and the corresponding series expansion is called \emph{Taylor series} expansion of $f$.\par
In Laurent series expansion, the part $\sum_{n=0}^{\infty}a_{n}(z-\alpha)^{n}$ is called the \emph{regular part} and the part $\sum_{n=1}^{\infty}b_{n}(z-\alpha)^{-n}$ is called the \emph{principal part}. For the principal part following three possibilities are to be considered:
\begin{enumerate}
\item [i)] If all the coefficients $b_{n}$ are zero, then we call that $z =\alpha$ is a \emph{removable singularity} of $f$ because we can make $f$ regular at $z =\alpha$  by suitably defining its value at $z = \alpha$.
\item [ii)] If the principal part of $f$ at $\alpha$ contains a finite number of terms only, then
$f$ is said to have a pole at $z = \alpha$. If $b_{m}$ $(m\geq1)$ is the last non-vanishing term in
the principal part, then $\alpha$ is called a \emph{pole} of $f$ with order $m$.
\item [iii)]
If the principal part of $f$ at $\alpha$ contains infinitely many non-zero terms, then
$f$ is said to have an \emph{essential singularity} at $z = \alpha$.
\end{enumerate}
If $f$ is analytic at $z =\alpha$ and $a_{0}=a_{1}=\ldots=a_{k-1}=0$ and $a_{k}\neq0$ in the Taylor series expansion of $f$ around $\alpha$, then the expansion takes the form $f(z)=\sum_{n=k}^{\infty}a_{n}(z-\alpha)^{n}$. In this case $z = \alpha$  is called a \emph{zero} of $f$ with multiplicity $k$.\par
Also for some finite complex number $a$, a zero of $f-a$  with multiplicity $k$ is called an \emph{a-point} of $f$ with multiplicity $k$. Sometimes, a pole of $f$ with multiplicity $k$ is called as an $\infty$-point of $f$ with multiplicity $k$.\par
A function $f:\mathbb{C}\rightarrow\mathbb{C}$ is called an \emph{entire function} if it is analytic on the entire complex plane. Also, $f$ is called \emph{meromorphic} on $\mathbb{C}$ if $f$ is analytic on $\mathbb{C}$ except possibly at isolated singularities, each of which is a pole.\par
A meromorphic function in the complex plane may have infinite number of poles or zeros or $a$-points but there is only a finite number of them in any finite domain. Otherwise, there would exist at least one limit point of poles or zeros or $a$-points in the finite plane and this would be an \emph{essential singularity}. Thus two cases may be distinguished:
\begin{enumerate}
\item [i)] The point at infinity is a \emph{regular point or a pole}. So the function has a finite number of poles or zeros. In this case the function is called  \emph{rational meromorphic function}.
\item [ii)] The point at infinity is an \emph{essential singularity}. Thus the function has an infinite number of poles or zeros or $a$-points which accumulate at the point at infinity. In this case the function is called \emph{transcendental meromorphic function}.
\end{enumerate}
Now we are the position to state  \emph{Poisson-Jensen's Formula} (\cite{1234}) for a meromorphic function $f$ in a disc $|z|\leq R~(0<R<\infty)$, which is the gateway to the Nevanlinna Theory.
\begin{theo}\label{pjt}
Suppose that $f(z)$ is meromorphic in $|z|\leq R~(0<R<\infty)$ and $a_{\mu}~(\mu=1,2,\ldots, m)$ and $b_{\nu}~(\nu=1,2,\ldots,n)$ are the zeros and poles of $f(z)$  inside the disc $|z|\leq R$ respectively. If $f(z)$ is analytic elsewhere inside and on the boundary of the disc and  $z=re^{i\theta}$ $(0\leq r<R)$ is a point in $|z|\leq R$ and $f(z)\not=0,\infty$, then \beas \log|f(re^{i\theta})|&=&\frac{1}{2\pi}\int\limits_0^{2\pi}\frac{R^2-r^2}{R^2-2rR\cos(\theta-\phi)+r^2}\log|f(Re^{i\phi})|\,d\phi\\
&+& \sum_{\nu=1}^{n}q_\nu \log\bigg|\frac{R^2-\overline b_\nu re^{i\theta}}{R(re^{i\theta}-b_\nu)}\bigg|-\sum_{\mu=1}^{m}p_\mu \log\bigg|\frac{R^2-\overline a_\mu re^{i\theta}}{R(re^{i\theta}-a_\mu)}\bigg|,\eeas
 where $p_\mu$ is the order of $a_\mu$ and $q_\nu$ is the order of $b_\nu$.
\end{theo}
$\bullet$ Here if we agree to count the poles or zeros according to their multiplicity, then we can drop $p_{\mu}$ and $q_{\nu}$ form the above expression.\par
$\bullet$ This formula is deduced under the assumption that neither the poles $b_{\nu}$ nor the zeros $a_{\mu}$ are to be found on $|z|= R$. The theorem continues to hold if we permit zeros and poles on $|z|=R$.
\begin{cor}[\emph{Poisson's formula}]
Suppose that $f(z)$ has no zeros and poles in $|z|\leq R$. If $z=re^{i\theta}$ $(0\leq r<R)$, then
$$\log|f(re^{i\theta})|=\frac{1}{2\pi}\int\limits_0^{2\pi}\frac{R^2-r^2}{R^2-2rR\cos(\theta-\phi)+r^2}\log|f(Re^{i\phi})|\,d\phi.$$
\end{cor}
\begin{theo}[\emph{Jensen's Theorem}]\label{j}
Suppose that $f(z)$ is meromorphic in $|z|\leq R~(0<R<\infty)$ and that $a_{\mu}~(\mu=1,2,\ldots, m)$ and $b_{\nu}~(\nu=1,2,\ldots,n)$ are the zeros and poles of $f(z)$  inside the disc $|z|\leq R$ respectively. If $f(z)$ is analytic elsewhere inside and on the boundary of the disc and $f(0)\neq 0, \infty$, then $$\log|f(0)|=\frac{1}{2\pi}\int\limits_0^{2\pi}\log|f(Re^{i\phi})|\,d\phi+\sum\limits_{\mu =1}^{m}\log\frac{|a_\mu|}{R}-\sum\limits_{\nu =1}^{n}\log\frac{|b_\nu|}{R}.$$
\end{theo}
\begin{rem}\label{jt}
For the case that $f(z)$ has zero of order $\lambda$ or pole of order $-\lambda$ at $z=0$, if the Laurent expansion of $f(z)$ at the origin be
$$f(z)= c_{\lambda}z^{\lambda} + c_{\lambda+1}z^{\lambda+1}+\ldots,~(c_{\lambda}\not=0),$$ then \emph{Jensen's formula} takes the form
$$ \log \mid c{_{\lambda}} \mid +\lambda \log R= \frac {1}{2\pi}\int_{0}^{2\pi} \log \mid f(Re^{i \phi})\mid d \phi  + \sum _{\mu = 1 }^{m} \log \mid \frac {a_{\mu } }{R} \mid - \sum _{\nu = 1 }^{n}\log \mid \frac {b_{\nu }}{R}\mid.$$
\end{rem}
This is the general form of the Jensen's formula. The theory of meromorphic functions depends largely on this formula.
\section{The First Fundamental Theorem}
\par
Rolf Nevanlinna developed a systematic study of the value distribution theory by means of his First and Second Fundamental Theorems. To this end, we shall explain some preliminaries of Nevanlinna theory.
\par
Let $f(z)$ be a non-constant meromorphic function defined in the complex plane. We denote
by $\emph{n(r,a; f)}$ the number of $a$-points of $f(z)$ in the disc $|z|<r$  for $a\in\mathbb{C}\cup\{\infty\}$, where an $a$-point is
counted according to its multiplicity. We put
$$N(r,a;f)=\int_{0}^{r}\frac{n(t, a; f)-n(0, a; f)}{t}dt+n(0, a; f) \log r,$$
where $n(0,a;f)$ denotes the multiplicity of $a$-points of $f(z)$ at the origin. Clearly $n(0,a;f)=0$ when $f(0)\not=a$. If $a = \infty$, we put $N(r,\infty; f) = N(r,f)$\par The function $N(r,a;f)$ is a real valued continuous function in $r$, known as \emph{integrated counting function} of the $a$-points of $f(z)$.
\begin{defi}
For $x\geq0$, we define
$$ \log^{+}x = \max\{\log x,0\}= \bigg\{ \begin{array}{c}
\log x,~~~x \geq1\\
0,~0\leq x<1.
\end{array}
$$
\text{It is obvious that} $~~\log x=\log^{+}x-\log^{+}{\frac{1}{x}}$ \text{for all} $x>0$.
\end{defi}
\begin{defi}
We define $$m(r,\infty;f)=m(r,f)=\frac{1}{2\pi}\int\limits_{0}^{2\pi}\log^{+}|f(re^{i\theta})|d\theta,$$ which we call the \emph{proximity function} of the meromorphic function $f(z)$.
\end{defi}
If $a\in\mathbb{C}$, then we set $m(r,a; f) = m(r,\frac{1}{f-a})$. The quantity
$$m(r,a; f)=\frac{1}{2\pi}\int\limits_{0}^{2\pi}\log^{+}\frac{1}{|f(re^{i\theta})-a|}d\theta$$
measures the mean deviation of the values of $f(z)$ from the value $a$ as $z$ is varying over the circle $|z| = r$. Actually we see that when the values of $f$ are relatively far away from the value $a$ for $z$ on $|z| = r$, then $m(r, a; f)$ is small. On the other hand, if the values of $f$ are relatively close to the value $a$ for $z$ on $|z| = r$, then $m(r, a; f)$ is large.\par
Also the quantity $N(r, a; f)$ is large or small according to $f(z) - a = 0$ has relatively many or relatively few roots inside the disc $|z| \leq r$.
\begin{defi}\label{gadu1}
The function $T(r, f) := m(r, f) + N(r, f)$ is called the \emph{Nevanlinna's characteristic function} of $f$, which plays a cardinal role in Nevanlinna Theory.
\end{defi}
The following inequalities are the basic properties of characteristic functions:
\begin{theo}
If $f_{\nu}(z)~(\nu=1,2,\ldots,p)$ are meromorphic functions in $\mid z\mid \leq R~(0<R<\infty)$, then for $0<r<R$, we have
\begin{enumerate}
\item [i)] $N\left(r,\prod\limits_{\nu=1}^{p}f_\nu\right)\leq\sum\limits_{\nu=1}^{p}N\left(r,f_\nu\right)$ and $N\left(r,\sum\limits_{\nu=1}^{p}f_\nu\right)\leq\sum\limits_{\nu=1}^{p}N\left(r,f_\nu\right)$,
\item [ii)] $m\left(r,\prod\limits_{\nu=1}^{p}f_\nu\right)\leq \sum\limits_{\nu=1}^{p}m\left(r,f_\nu\right)$ and  $m\left(r,\sum\limits_{\nu=1}^{p}f_\nu\right)\leq\sum\limits_{\nu=1}^{p}m\left(r,f_\nu\right)+\log p$,
\item [iii)] $T\left(r,\prod\limits_{\nu=1}^{p}f_\nu\right)\leq\sum\limits_{\nu=1}^{p}T\left(r,f_\nu\right)$ and  $T\left(r,\sum\limits_{\nu=1}^{p}f_\nu\right)\leq\sum\limits_{\nu=1}^{p}T\left(r,f_\nu\right)+\log p$.
\end{enumerate}
\end{theo}
We see that in view of Definition (\ref{gadu1}) and Remark \ref{jt}, one can write general \emph{Jensen's formula} in terms of Nevanlinna characteristic functions and this form is known as \emph{Jensen-Nevanlinna formula}.
\begin{theo}\label{cth1}
Let $f(z)$ be meromorphic in $\mid z\mid \leq R~(0<R<\infty)$. Then for $0<r<R$, we have
$$ T\left(r,\frac{1}{f}\right)=T(r,f)+O(1),$$
where $O(1)$ is a bounded term depending on $f$ and $a$ but not on  $r$.
\end{theo}
\begin{theo}[\emph{Nevanlinna's First Fundamental Theorem}]
Let $f(z)$ be a non-constant meromorphic function defined in $|z|< R~(0<R \leq \infty)$ and let $a\in \mathbb{C} \cup\{\infty\}$ be any complex number. Then for $0<r<R$
$$T\left(r,\frac{1}{f-a}\right)=T(r,f)+O(1),$$
where $O(1)$ is a bounded quantity depending on $f$ and $a$ but not on  $r$.
\end{theo}
\begin{rem}
The First Fundamental Theorem provides an upper bound to the number of zeros of the functional equation $f(z) = a$ for $a\in\mathbb{C}\cup\{\infty\}$.
\end{rem}
In order to study some properties of characteristic functions, we state an alternative expression of $T(r,f)$:
\begin{theo}
If $f$ is meromorphic in $|z|<R$, then $$T(r,f)=\frac{1}{2\pi}\int\limits_{0}^{2\pi}N(r,e^{i\theta})d\theta+\log^{+}|f(0)|,~0<r<R.$$
This expression of $T(r,f)$ by an integral is known as \emph{Henri Cartan's Identity}.
\end{theo}
\begin{cor}
$\frac{1}{2\pi}\int\limits_{0}^{2\pi}m(r,e^{i\theta})d\theta\leq \log2.$
\end{cor}
\begin{defi}
A function $y=f(x)$ is called \emph{convex downward} if for any two points $(x_1,f(x_1))$ and $(x_2,f(x_2))$ on the curve $y=f(x)$, the chord joining these two points lies above the arc of the curve between the two points i.e, for any $x\in[x_1,x_2]$ if $f'(x)$ is non-decreasing or, $f(x_1)(x_1-x_2)+f(x)(x_2-x_1)+f(x_2)(x_1-x)<0$.
\end{defi}
\begin{theo}For any complex number $a\in\mathbb{C}\cup\{\infty\}$,
\begin{enumerate}
\item  [i)] $N(r,a;f)$ is a non-decreasing function of $r$ and convex function of $\log r$.
\item [ii)] Also $T(r,f)$ is a non-decreasing function of $r$ and convex function of $\log r$.
\end{enumerate}
\end{theo}
But the function $m(r,a;f)$ is not necessarily increasing and convex, for example
\begin{exm}\label{homeoex1}
We consider $$f(z)=\frac{z}{1-z^{2}}.$$
Then $m(r,f)=0$ for $r\leq \frac{1}{2}$ and $r\geq2$ but on the other hand $m(r,f)>0$ for $r=1$.\end{exm}
Below we are giving some examples, some of which will be needful in this sequel.
\begin{exm}\label{homeoex2}
If $f(z)=\frac{P(z)}{Q(z)}$, where $P(z)$ and $Q(z)$ are polynomials having no common factors of degree $p$ and $q$  respectively and $d=\max\{p,q\}$, then
$$T(r,f)=d\log r+O(1)~\text{as}~r \to \infty.$$
\end{exm}
\begin{exm}\label{homeoex3}
If $g(z)=\frac{af(z)+b}{cf(z)+d}$, where $ad-bc\neq0$, then $T(r,g)=T(r,f)+O(1).$
\end{exm}
\begin{exm}\label{homeoex4}
If $f(z)=e^{z}$, then $T(r,f)=\frac{r}{\pi}~\text{as}~r \to \infty.$
\end{exm}
\begin{exm}\label{homeoex5}
If $f(z)=e^{P(z)}$, where $P(z)$ is polynomial of degree $p$, then
$$T(r,f)\sim \frac{|a_{p}|r^{p}}{\pi}~\text{as}~r \to \infty,$$
 where $a_{p}$ is the coefficients of $z^{p}$ in $P(z)$.
\end{exm}
\begin{exm}\label{homeoex6}
If $f(z)=e^{e^{z}}$, then $T(r,f)\thicksim\frac{e^{r}}{\sqrt{(2\pi^{3}r)}}~\text{as}~r \to \infty.$
\end{exm}
The following theorem is the characterization of a meromorphic function, whether it is transcendental or rational.
\begin{theo}
If $fz)$ is a transcendental meromorphic function in the complex plane, then
$$\lim_{r \to \infty}\frac{T(r,f)}{\log r} =\infty.$$
\end{theo}
\begin{cor}\label{1.1}
If $fz)$ is a non-constant meromorphic function in the complex plane, then $f(z)$ is a rational function if and only if
$$\liminf_{r \to \infty}\frac{T(r,f)}{\log r} <\infty.$$
\end{cor}
Now we briefly discuss about \emph{Order} and \emph{Type} of a meromorphic function.
\begin{defi}
Let $S(r)$ be a real valued and non-negative increasing function for $0<r_{0}<r<\infty$. The \emph{order $\rho$} and the \emph{lower order $\lambda$} of the function $S(r)$ are defined as $$\rho=\limsup\limits_{r\rightarrow\infty}\frac{\log^{+} S(r)}{\log r}~~and~~\lambda=\liminf\limits_{r\rightarrow\infty}\frac{\log^{+} S(r)}{\log r}.$$
\end{defi}
From the definition it is clear that the order and the lower order of a function always satisfies the relation $ 0\leq\lambda\leq\rho\leq\infty.$
\begin{defi}
If $\rho=\infty$, then $S(r)$ is said to be of \emph{infinite order}.\vspace{.3 cm}\\
If $0<\rho<\infty$, we set $c=\limsup\limits_{r\rightarrow\infty}\frac{S(r)}{r^\rho}$ and consider the following possibilities:
\begin{enumerate}
\item [i)] $S(r)$ has \emph{maximal type} if $c=\infty$,
\item [ii)] $S(r)$ has \emph{mean type} if $0<c<\infty$,
\item [iii)] $S(r)$ has \emph{minimal type} if $c=0$,
\item [iv)] $S(r)$ has \emph{convergence class} if $\int_{r_{_{0}}}^{\infty}\frac{S(t)}{t^{\rho+1}}dt$ converges.
\end{enumerate}
\end{defi}
\begin{rem}
If $S(r)$ is of order $\rho~(0<\rho<\infty)$, then for each $\varepsilon(>0)$
\begin{enumerate}
\item [i)] $S(r)<r^{\rho+\varepsilon}$ for all sufficiently large r and
\item [ii)] $S(r)>r^{\rho-\varepsilon}$ for a sequence of values of r tending to $\infty$.
\end{enumerate}
\end{rem}
\begin{rem}
If $S(r)$ is of lower order $\lambda~ (0<\lambda<\infty)$, then for each $\varepsilon(>0)$
\begin{enumerate}
\item [i)] $S(r)>r^{\lambda-\varepsilon}$ for all sufficiently large r and
\item [ii)] $S(r)<r^{\lambda+\varepsilon}$ for a sequence of values of $r \to \infty$.
\end{enumerate}
\end{rem}
\begin{theo}
Let $S(r)$ be of order $\rho~(0<\rho<\infty)$. Then the integral $\int_{r_{_{0}}}^{\infty}\frac{S(t)}{t^{k+1}}dt$ converges if $k>\rho$ and diverges if $k<\rho$.
\end{theo}
\begin{theo}
If $S(r)$ is of order $\rho~(0<\rho<\infty)$ and has convergence class, then $S(r)$ has minimal type of order $\rho$.
\end{theo}
The following example shows that the converse of this theorem is not true.
\begin{exm}
Assume $S(r)=r^\lambda(\log r)^\mu$.\par
Then $\lambda$ is the order of $S(r)$. If we choose $0<\lambda<\infty$ and $-1<\mu< 0$, then $S(r)$ is of minimal type but does not belong to the convergence class.
\end{exm}
Let $f(z)$ be a non-constant entire function. Then to characterize the growth of $f(z)$ and the distribution of its zeros, we introduce a special growth scale called \emph{maximum modulus function} as follows: $$M(r)=M(r,f)=\max\limits_{\mid z \mid =r } \mid f(z)\mid.$$
\begin{theo}
If $f(z)$ is non-constant and regular in $|z|\leq R$, then
$$T(r,f)\leq \log^{+}M(r,f)\leq\frac{R+r}{R-r}T(R,f),$$
where $0\leq r < R.$
\end{theo}
From the above inequality, we have the following Corollary:
\begin{cor}\label{homeo}
If $f(z)$ is an non-constant entire function, then the functions $S_1(r)=\log^{+}M(r,f)$ and $S_{2}(r)=T(r,f)$ have the same order.\par
Furthermore, if the value of the common order be $\rho ~(0<\rho<\infty)$, then $S_{1}(r)$ and $S_{2}(r)$ are together of minimal type, mean type, maximal type or of convergence class.
\end{cor}
Next we define the \emph{order} and \emph{hyper order} of a meromorphic function.
\begin{defi}\label{homeo1}
Let $f(z)$ be a non-constant meromorphic function in the open complex plane. The function $f(z)$ is said to have order  $\rho$, maximal, mean or minimal type or convergence class if the characteristic function $T(r,f)$ has this property.
\end{defi}
\begin{defi}
The hyper order $\rho _{_{2}}(f)$ of a non-constant meromorphic function $f(z)$ is defined by
$$\rho _{_{2}}(f)=\limsup\limits_{r\lra\infty}\frac{ \log \log T(r,f)}{\log r}.$$
\end{defi}
It may be noted that for an integral function $f(z)$, we can take $\log^{+}M(r,f)$ in place of $T(r,f)$ in Definition \ref{homeo1} in view of the Corollary \ref{homeo}.
\begin{theo}
If $\rho_{_{f}}$ and $\rho_{_{g}}$ be the orders of the meromorphic functions $f(z)$ and $g(z)$ respectively, then
\begin{enumerate}
\item [i)] order of $(f\pm g)\leq \max\{\rho_{_{f}}$, $\rho_{_{g}}\}$,
\item [ii)] order of $(fg)\leq \max\{\rho_{_{f}}$, $\rho_{_{g}}\}$,
\item [iii)] order of $\left(\frac{f}{g}\right)\leq \max\{\rho_{_{f}}$, $\rho_{_{g}}\}$.
\end{enumerate}
Also equality holds in each case if $\rho_{_{f}}\not=\rho_{_{g}}$.\\
\end{theo}
\begin{exm}
For a rational function $f(z)$, in view of Example \ref{homeoex2}, the order is\\ $\rho=\limsup\limits_{r\rightarrow\infty}\frac{\log(O(\log r))}{\log r}=0$.
\end{exm}
\begin{exm}
If $f(z)=e^{z}$, then in view of Example \ref{homeoex4}, the order is\\ $\rho=\limsup\limits_{r\rightarrow\infty}\frac{\log(\frac{r}{\pi})}{\log r}=1$.
\end{exm}
\begin{exm} If $f(z)=e^{e^{z}}$, then in view of Example \ref{homeoex6}, the order is $\rho=\infty$.
\end{exm}
\section{The Second Fundamental Theorem via Logarithmic Derivatives}
\par
Now we are going to discuss to Nevanlinna's Second Main Theorem. In order to state Second Main Theorem, we introduce a important Lemma, which is the main result of this section, known as the \emph{Lemma on the Logarithmic Derivative}.\par
This Lemma says that if $f$ is a meromorphic function, then the integral means of $\log^{+}|\frac{f'}{f}|$ over large circles cannot approach infinity quickly compared with the rate at which $T(r,f)$ tends to infinity.\par
For example, if $f(z)=e^{z^{2}}$, then $\frac{f'}{f}=2z$, so  $\frac{f'}{f} \to \infty$ as $z \to \infty$. However, we see here that the rate at which $\log\frac{f'}{f}$ approaches infinity is very slow by comparison to $T(r,f)$.\par
The Lemma on the Logarithmic Derivative was first proved by R. Nevanlinna and was the basis for his first proof of the Second Main Theorem.
\begin{theo}\label{kus3} (\cite{1234})
Suppose that $f(z)$ is a non-constant meromorphic function in $|z| < R$ and $a_{j}~(j=1,2,\ldots,q)$ are $q$ distinct \emph{finite complex numbers}. Then
$$m\left(r,\sum\limits_{j=1}^{q}\frac{1}{f-a_{j}}\right)=\sum\limits_{j=1}^{q}m\left(r,\frac{1}{f-a_{j}}\right) + O(1),$$
holds for $0<r<R$.
\end{theo}
\begin{theo} (\cite{1234})\label{meat}
Suppose that $f(z)$ is a non-constant meromorphic function in the complex plane and $f(0)\neq0,\infty$. Then
$$m\left(r,\frac{f'}{f}\right)< 4\log ^{+} T(R,f)+3\log ^{+}\frac{1}{R-r}+4\log ^{+}R+2\log^{+}\frac{1}{r}+4\log ^{+}\log ^{+}\frac{1}{|f(0)|}+10$$
holds for $0<r<R<\infty$.
\end{theo}
\begin{rem} If $f(0)=0$ or $\infty$, then we can write $f(z)=z^{\lambda}g(z)$ for some $\lambda \in\mathbb{Z}$ and some meromorphic function $g(z)$ such that $g(0)\not=0$ or $\infty$.\par
Thus $\frac{f'(z)}{f(z)}=\frac{\lambda}{z}+\frac{g'(z)}{g(z)}$. So we can apply Theorem \ref{meat} to $m\left(r,\frac{g'}{g}\right)$ and hence get similar type inequality for $m\left(r,\frac{f'}{f}\right)$.
\end{rem}
\begin{theo} (\cite{1234})
Suppose that $f(z)$ is a non-constant meromorphic function in the complex plane.
\begin{enumerate}
\item [i)] If the order of $f(z)$ is \emph{finite}, then $m\left(r,\frac{f'}{f}\right)=O(\log r)~\text{as}~r \to \infty,$
\item [ii)] If the order of $f(z)$ is \emph{infinite}, then $m\left(r,\frac{f'}{f}\right)=O(\log( rT(r,f)))~\text{as}~ r \to \infty$ and $r \not\in E_{0},$
\end{enumerate}
where $E_{0}$ is a set whose linear measure is not greater than $2$.
\end{theo}
\begin{defi}
Suppose that $f(z)$ is a non-constant meromorphic function in the complex plane. A quantity $\Delta$ is said to be $S(r,f)$ if $\frac{\Delta}{T(r,f)} \to 0 $ as $r \to \infty$ and $r \not\in E_{0}$ where $E_{0}$ is a set whose linear measure is not greater than $2$.
\end{defi}
\begin{theo} (\cite{1234})
Suppose that $f(z)$ is a non-constant meromorphic function in the plane.
\begin{enumerate}
\item [i)] If the order of $f(z)$ is \emph{finite}, then $O(\log r)=S(r,f)~\text{as}~r \to \infty,$
\item [ii)] If the order of $f(z)$ is \emph{infinite}, then $O\left(\log( rT(r,f))\right)=S(r,f)~\text{as}~ r \to \infty$ and $r \not\in E_{0},$
\end{enumerate}
where $E_{0}$ is a set whose linear measure is not greater than $2$.
\end{theo}
\begin{theo}[\emph{Lemma of Logarithmic Derivative}] (\cite{1234})
Suppose that $f(z)$ is a non-constant meromorphic function in whole complex plane. Then\\
$$m\bigg(r,\frac{f'}{f}\bigg)=S(r,f)~\text{as}~ r \to \infty~\text{and}~r \not\in E_{0},$$
where $E_{0}$ is a set whose linear measure is not greater than $2$.
\end{theo}
\begin{cor}\label{ghgh}
Suppose that $f(z)$ is a non-constant meromorphic function in whole complex plane and $l$ is natural number. Then\\
$$m\bigg(r,\frac{f^{(l)}}{f}\bigg)=S(r,f)~\text{as}~ r \to \infty~\text{and}~r \not\in E_{0},$$
where $E_{0}$ is a set whose linear measure is not greater than $2$.
\end{cor}
Now we are in the position to state one of the most valuable result in value distribution theory, namely \emph{Second Fundamental Theorem of Nevanlinna}.
\begin{theo}\label{kus1}
Suppose that $f(z)$ is a non-constant meromorphic function in $|z|< R$ and $a_{j}~(j=1,2,\ldots,q)$ are $q~(q\geq2)$ \emph{distinct finite} complex numbers.
Then for $0<r<R$, we have
$$m(r,f)+\sum\limits_{j=1}^{q}m\left(r,\frac{1}{f-a_j}\right)\leq2T(r,f)-N_1(r)+\Delta(r,f),$$
where
\bea\label{kus2} N_1(r)=N\left(r,0;f\right)+2N(r,\infty;f)-N(r,\infty;f^{'}),\eea
and
$$\Delta(r,f)=m\left(r,\frac{f'}{f}\right)+m\left(r,\sum\limits_{j=1}^{q}\frac{f'}{f-a_j}\right)+q\log^+\frac{3q}{\delta}+\log 2.$$
\end{theo}
In view of Theorems \ref{kus3}, \ref{kus1} and Lemma of Logarithmic derivative, the following corollary is immediate:
\begin{cor}\label{kus5}
Suppose that $f(z)$ is a non-constant meromorphic function in $|z|< R$ and $a_{j}~(j=1,2,\ldots,q)$ are $q~(q\geq2)$ \emph{distinct finite} complex numbers.
Then for $0<r<R$, we have
$$m(r,f)+\sum\limits_{j=1}^{q}m\left(r,\frac{1}{f-a_j}\right)\leq2T(r,f)-N_1(r)+S(r,f),$$
outside a set $E_{0}$ of $r$ which has linear measure is at most $2$ and the quantity $N_1(r)$ is same as equation (\ref{kus2}).
\end{cor}
To describe Second Main Theorem in more general settings, we need to introduce some notations.\par
Let $f$ be a non-constant meromorphic function in $|z|<R~(\leq\infty)$ and $a\in\mathbb{C}\cup\{\infty\}$. For $0<r<R$, we denote by $\ol{n}(r,a; f)$ the number of distinct roots of $f(z)=a$ in $|z|\leq r$, multiple roots being counted once. Correspondingly we define
$$\ol{N}(r,a;f)=\int_{0}^{r}\frac{\ol{n}(t, a; f)-\ol{n}(0,a; f)}{t}dt+\ol{n}(0, a; f) \log  r,$$
where $ \ol{n}(0,a; f)=0$ if $f(0)\neq a$ and $ \ol{n}(0,a; f)=1$ if $f(0)=a$.\\
\begin{lem}\label{kus6}
Suppose that $f(z)$ is a non-constant meromorphic function in the complex plane. Let $a_j~(j=1,2,\ldots,q)$ be distinct complex numbers. Then
$$\sum\limits_{j=1}^{q}N\left(r,\frac{1}{f-a_j}\right)-N_{1}(r) \leq \sum\limits_{j=1}^{q}\ol{N}\left(r,\frac{1}{f-a_j}\right)$$
\end{lem}
Thus in view of Lemma \ref{kus6}, Corollary \ref{kus5} can be written as follows:
\begin{theo}\label{kus7}
Suppose that $f(z)$ is a non-constant meromorphic function in $|z|< R$. Let $a_{1},a_{2},\ldots,a_{q}$ be $q~(q\geq2)$ distinct \emph{finite} complex numbers. Then for $0<r<R$,
\bea\label{csmt} (q-1)T(r,f) &\leq& \ol{N}(r,\infty;f)+\sum\limits_{j=1}^{q}\ol{N}\left(r,\frac{1}{f-a_j}\right)+S(r,f),\eea
holds $\forall r\in (0,\infty)$, except possibly outside of a set whose linear measure is atmost $2$.
\end{theo}
From this point, Theorem \ref{kus7} will be treated as the \emph{Second Fundamental Theorem}.
\begin{rem}
The Second Fundamental Theorem gives the lower bounds for the number of zeros of the equation $f(z) = a$.
\end{rem}
\begin{defi}\label{small}
Let $f(z)$ and $a(z)$ be two meromorphic functions in the complex plane. If $T(r,a)=S(r,f)$, then $a(z)$ is called a \emph{small function} with respect to $f(z)$.
\end{defi}
\begin{exm}
i) Rational functions are small function with respect to any transcendental\par
~~~~~~~~~~~~~~~~~~~~meromorphic function.\par
~~~~~~~~~~~~~~~~ii) $e^{z}$ is small function with respect to $e^{e^{z}}$.
\end{exm}
\begin{theo}[\emph{Milloux's Theorem}]
Suppose $f(z)$ is a non-constant meromorphic function in the complex plane and $l$ is a positive integer. If $\psi(z)=\sum\limits_{j=1}^{l}a_j(z)f^{(j)}(z)$, where $a_{1},a_{2},\ldots,a_{l}$ are small functions of $f(z)$, then
\begin{enumerate}
\item [i)] $m(r,\frac{\psi}{f})=S(r,f)$ and
\item [ii)] $T(r,\psi) \leq T(r,f)+l\;\ol{N}(r,\infty;f)+S(r,f)$.
\end{enumerate}
\end{theo}
Next we state the \emph{Second Fundamental Theorem for small functions}.
\begin{theo}\label{kus9}
Suppose that $f(z)$ is a non-constant meromorphic function in the complex plane and $a_{1}(z), a_{2}(z),a_{3}(z)$ are three distinct small functions of $f(z)$. Then
$$T(r,f)<\sum\limits_{j=1}^{3}\overline{N}(r,a_{j};f)+S(r,f),$$
for any positive $r$ excluding some set $E$ with finite linear measure.
\end{theo}
In 2001, Li-Zhang (\cite{18012017}) proved the following theorem:
\begin{theo}
Suppose that $f(z)$ is a transcendental meromorphic function in the complex plane and $a_{1}(z), a_{2}(z),\ldots,a_{q}(z)$ are $q$ distinct small functions of $f(z)$. Then
$$(q-1)T(r,f)<N(r,\infty;f)+\sum\limits_{j=1}^{q}N(r,a_{j};f)+S(r,f),$$
for any positive $r$ excluding some set $E_{0}$ with finite linear measure.
\end{theo}
Now we state a particular case of K. Yamanoi's (\cite{6}) result:
\begin{theo}\label{kus8}
Suppose that $f(z)$ is a meromorphic function in the complex plane and $a_{1}(z), a_{2}(z),\ldots,a_{q}(z)$ are $q$ mutually distinct meromorphic functions with $f\not\equiv a_{i}$. Then for every $\varepsilon> 0$, there exists a positive constant $C(\varepsilon)$ such that
$$(q-2-\varepsilon)T(r,f)< \sum\limits_{j=1}^{q}\ol{N}(r,a_{j};f)+C(\varepsilon)\sum\limits_{j=1}^{q}T(r,a_{j})+S(r,f),$$
for any positive $r$ excluding some set $E_{0}$ with finite linear measure.
\end{theo}
The following theorem is an immediate consequence of Theorem \ref{kus8} which is the generalization of Theorem \ref{kus9}.
\begin{theo}
Suppose that $f(z)$ is a meromorphic function in the complex plane and $a_{1}(z), a_{2}(z),\ldots,a_{q}(z)$ are $q$ mutually distinct small functions with respect to $f$. Then for every $\varepsilon> 0$, there exists a positive constant $C(\varepsilon)$ such that
$$(q-2-\varepsilon)T(r,f)< \sum\limits_{j=1}^{q}\ol{N}(r,a_{j};f)+S(r,f),$$
for any positive $r$ excluding some set $E_{0}$ with finite linear measure.
\end{theo}
\section{Some Applications}
The value distribution theory has several applications in many branches of Mathematics. But as our thesis is on uniqueness theory of Entire and Meromorphic functions, we shall restrict the discussion within Function Theory.
\subsection*{Defect Relation}
In this section, first we state a weaker reformulation of the Second Fundamental Theorem of Nevanlinna and then after state some nice consequences of it. For this, we introduce some notations.
\begin{defi}
Let $f(z)$ be a non-constant meromorphic function defined in $|z|<R$. For $0<r<R\leq\infty$ and $a \in \mathbb{C}\cup\{\infty\}$, we define
\beas && \delta(a,f)=1-\limsup\limits_{r\rightarrow R} \frac{N(r,a;f)}{T(r,f)},\\
&& \Theta(a,f)=1-\limsup\limits_{r\rightarrow\ R} \frac{\overline{N}(r,a;f)}{T(r,f)},\\
&& \theta(a,f)=\liminf\limits_{r\rightarrow R} \frac{N(r,a;f)-\overline{N}(r,a;f)}{T(r,f)}.\eeas
The terms $\delta(a,f),~\Theta(a,f)$ and $\theta(a,f)$ are known as the \emph{deficiency}, \emph{ramification index} and \emph{index of multiplicity} of the value $a$ respectively for the function $f$.
\end{defi}
$\bullet$ It is clear that $0 \leq \delta(a,f)\leq\Theta(a,f) \leq 1.$\par
$\bullet$ For given any $\varepsilon>0$ and for sufficiently large values of $r$,
$$ N(r,a;f)-\overline{N}(r,a;f)>\{\theta(a,f)-\varepsilon\}T(r,f)~\text{and}~N(r,a;f)<\{1-\delta(a,f)+\varepsilon\}T(r,f).$$
and hence $$\overline{N}(r,a;f)<\{1-\delta(a,f)-\theta(a,f)+2\varepsilon\}T(r,f);~~\text{i.e.,}~~\delta(a,f)+\theta(a,f)\leq \Theta(a,f).$$
So $0\leq \min\{\delta(a,f),\theta(a,f)\} \leq  \max\{\delta(a,f),\theta(a,f)\} \leq \delta(a,f)+\theta(a,f) \leq \Theta(a,f) \leq 1.$
\begin{theo}(\cite{1234}) [\emph{Weaker reformulation of the Second Fundamental Theorem}]\label{kus10}
Suppose that $f(z)$ is a non-constant meromorphic function in the complex plane. Then the  set $T = \{a : \Theta(a, f ) > 0\}$ is countable and on summing over all such values $a$, we have
\bea \sum\limits_{a \in T}\{\delta(a,f)+\theta(a,f)\} &\leq& \sum\limits_{a\in T}\Theta(a,f)\leq2.\eea
\end{theo}
Next, we shall look at some important theorems of the analysis of entire or meromorphic functions. To begin, consider the Fundamental Theorem of Algebra (FTA), if $f$ is a complex polynomial of degree $n > 0$, then $f-a$ has $n$ zeros in $\mathbb{C}$ for every $a \in \mathbb{C}$. But the direct generalization of the FTA is not true. Picard's theorem adds that entire functions can take on values infinitely many times while omitting one value.
\begin{defi}
Let $f(z)$ be a meromorphic function in the complex plane. A finite value $a$ is called a \emph{Picard exceptional value} of $f(z)$ if $f(z)-a$ has no zero.
\end{defi}
\begin{theo}[\emph{Picard's Little Theorem}]
Any non-constant entire function has at most one Picard exceptional value in the complex plane. For example, $e^z$ omits the value $0$.
\end{theo}
\begin{rem}
In fact, the Second Fundamental Theorem generalizes \emph{Picard's Theorem}.
\end{rem}
\begin{theo}
Any non-constant meromorphic function  has at most two Picard exceptional values in the complex plane. For instance, $\tan z$ omits $\pm i$.
\end{theo}
The Fermat's Last Theorem tells us that for any integer $n\geq 3$, there exist no non-zero integer triples $(x, y, z)$ such that $x^n + y^n = z^n$. In this section, we will look at a generalized version of this theorem for functions.
\begin{theo}
Let $f$ and $g$ be non-constant entire (resp. meromorphic) functions such that $f^n+g^n=1$ for all $z\in \mathbb{C}$, and let $n$ be a non-negative integer. Then $n\leq 2$ (resp. 3).
\end{theo}
\subsection*{Functions Sharing Values}
The most fascinating result of Nevanlinna in the uniqueness theory is \emph{five value theorem}. Like the other beautiful results in complex analysis, there is no corresponding version of this theorem in real case also. To state this theorem, we introduce some notations.
\begin{defi}\label{daug} (\cite{padic})
Let $f$ and $g$ be two non-constant meromorphic functions and $a\in\mathbb{C}\cup\{\infty\}$. We define
$$E_{f}(a)=\{(z,p) \in \mathbb{C}\times\mathbb{N}~ |~ f(z)=a ~with~ multiplicity~ p\},$$
$$\ol{E}_{f}(a)=\{(z,1) \in \mathbb{C}\times\mathbb{N}~ |~ f(z)=a~with~ multiplicity~ p\}.$$
If $E_{f}(a)=E_{g}(a)$ \big(resp. $\ol{E}_{f}(a)=\ol{E}_{g}(a)$\big), then it is said that $f$ and $g$ share the value $a$ \emph{counting multiplicities or in short CM} \big(resp. \emph{ignoring multiplicities or in brief IM}\big).
\end{defi}
\begin{theo}[\emph{Nevanlinna's five-value theorem}]\label{chap1}
Suppose that  $f(z)$ and $g(z)$ are two non-constant meromorphic functions and $a_{1},a_{2},\ldots,a_{5}$ are five distinct values in the extended complex plane. If $f$ and $g$ share $a_j$ ($j=1,2,\ldots,5$) IM, then $f\equiv g$.
\end{theo}
\begin{rem} The number five is sharp in the above theorem. for example, we note that $f(z)=e^z$ and  $g(z)=e^{-z}$ share the four values $0,1,-1,\infty$ IM but $f\not\equiv g$.
\end{rem}
\begin{defi}
Let $m$ be a positive integer or infinity and $a\in\mathbb{C}\cup\{\infty\}$. By $E_{m)}(a;f)$, we mean the set of all $a$-points of $f$ with multiplicities not exceeding $m$, where an $a$-point is counted according to its multiplicity. Also by $\ol E_{m)}(a;f)$, we mean the set of distinct $a$-points of $f(z)$ with multiplicities not greater than $m$.\par
If for some $a\in\mathbb{C}\cup\{\infty\}$, $E_{m)}(a;f)=E_{m)}(a;g)$ \big(resp. $\ol E_{m)}(a;f)=\ol E_{m)}(a;g)$\big) holds for $m=\infty$, then we see that $f$ and $g$ share the value $a$ CM (resp. IM).
\end{defi}
\begin{theo} (Theorem 3.12, \cite{1234})\label{4117}
Let $f(z)$ be a non-constant meromorphic function. Then $f(z)$  can be uniquely determined by $q$ ($=5+[\frac{2}{k}]$) sets $\ol E_{k)}(a_{j},f)$ ($j=1,2,\ldots,q$), where $a_{j}$ ($j=1,2,\ldots,q$) are $q$ distinct complex numbers and $[\frac{2}{k}]$ denotes the largest integer less than or equal to $\frac{2}{k}$.
\end{theo}
\begin{cor} If, in Theorem \ref{4117}, we choose $q=7$, $6$ or $5$, then the obvious choice of $k$  is respectively $1$, $2$ or $3$.
\end{cor}
\subsection*{Uniqueness Involving Derivatives}
It is natural to ask that what will happen in a special case when $g=f'$ in Theorem \ref{chap1}? Based on this curiosity, the subject matter on sharing values between entire functions and their derivatives was developed which was first studied by Rubel-Yang (\cite{br10a}).\par
Analogous to the Nevanlinna's five value theorem, in early 1977, Rubel-Yang (\cite{br10a}) proved that if any non-constant entire function $f$ and its first derivative $f'$ share two distinct finite values $a$ and $b$ CM, then $f=f'$.\par
Two years later, Mues-Steinmetz (\cite{br10}) proved that actually in the result of Rubel-Yang, one does not even need the multiplicities. Later it was shown that, in general, the results of Rubel-Yang or  Mues-Steinmetz are false, if $f$ and $f'$ share only one value. Thus one may ask that what conclusion can be made, if $f$ and $f'$ share only one value, and if an appropriate restriction on the growth of $f$ is assumed. In this direction, in 1996, a famous conjecture was proposed by R. Br\"{u}ck (\cite{br3}). Since then the conjecture and its analogous results have been investigated by many researchers which we shall discuss in the concerned chapters. Now we recall the following definition.
\begin{defi} (\cite{2}) Let $n_{0j},n_{1j},\ldots,n_{kj}$ be non-negative integers. The expression $M_{j}[f]=(f)^{n_{0j}}(f^{(1)})^{n_{1j}}\ldots(f^{(k)})^{n_{kj}}$ is called a \emph{differential monomial} generated by $f$ of degree $d(M_{j})=\sum_{i=0}^{k}n_{ij}$ and weight $\Gamma_{M_{j}}=\sum_{i=0}^{k}(i+1)n_{ij}$.\par
The sum $P[f]=\sum_{j=1}^{t}b_{j}M_{j}[f]$ is called a \emph{differential polynomial} generated by $f$ of degree $\ol{d}(P)=max\{d(M_{j}):1\leq j\leq t\}$
and weight $\Gamma_{P}=max\{\Gamma_{M_{j}}:1\leq j\leq t\}$, where $T(r,b_{j})=S(r,f)$ for $j=1,2,\ldots,t$.\par
The numbers $\underline{d}(P)=min\{d(M_{j}):1\leq j\leq t\}$ and $k$ \big( the highest order of the derivative of $f$ in $P[f]$ \big) are called respectively the lower degree and order of $P[f]$.\par
$P[f]$ is said to be homogeneous if $\ol{d}(P)$=$\underline{d}(P)$. Also $P[f]$ is called a Linear Differential Polynomial generated by $f$ if $\ol {d} (P)=1$.\par
We also denote by $\mu=Q=max\; \{\Gamma _{M_{j}}-d(M_{j}): 1\leq j\leq t\}=max\; \{ n_{1j}+2n_{2j}+\ldots+kn_{kj}: 1\leq j\leq t\}$.
\end{defi}
Since the generalizations of derivative is the differential polynomials, the afterward research on Br\"{u}ck conjecture and its generalization, one setting among the sharing functions has been restricted to only various powers of $f$ not involving any other variants such as derivatives of $f$, where as the generalization have been made on the second setting, for example, Li-Yang (\cite{br8c}), Zhang (\cite{br15}). \par Thus it will be interesting to study the relation between a power of a meromorphic function with its differential polynomial when they share some values or more generally small functions taking in background the conjecture of Br\"{u}ck. Concerning the above discussions, chapter two, three and four have been organized.
\subsection*{Functions Sharing Sets}
In course of studying the factorization of meromorphic functions, in 1977, F. Gross (\cite{am4}) initiated the uniqueness theory under more general setup
by introducing the concept of a unique range set by considering pre-images of sets of distinct elements (counting multiplicities). He also asked that does there exist a finite set $S$ such that for two entire functions $f$ and $g$, $f^{-1}(S) = g^{-1}(S)$ implies $f\equiv g$?
\begin{defi}
Let $f$ be a non-constant meromorphic function and $S\subseteq\mathbb{C}\cup\{\infty\}$. We define $E_{f}(S)=\cup_{a\in S}\{z~:~f(z)=a\}$, where a zero of $f-a$ with multiplicity $m$ counts $m$-times. Similarly we can define $\overline{E}_{f}(S)$ where a zero of $f-a$ is counted ignoring multiplicity.
\end{defi}
\begin{defi}
A set $S\subset \mathbb{C}\cup\{\infty\}$ is called a unique range set for meromorphic (resp. entire) functions, if for any two non-constant meromorphic (resp. entire) functions $f$ and $g$, the condition $E_{f}(S)=E_{g}(S)$ implies $f\equiv g$. In short, we call the set $S$ as URSM (resp. URSE). Similarly, we can define the unique range set ignoring multiplicity.\end{defi}
After the introduction of the novel idea of unique range sets researchers were getting more involved to find new unique range sets with cardinalities as small as
possible. In 1996, Li (\cite{li96}), first pointed out that the cardinality of an URSE is at least five whereas Yang-Yi (\cite{1234})  showed that the cardinality of an URSM is at least six. Till date the URSE with seven elements (\cite{am12}) and URSM with eleven elements (\cite{am3}) are the smallest available URSE and URSM.\par
A recent advent in the uniqueness literature is the notion of \emph{weighted sharing} environment which implies a gradual change from sharing IM to sharing CM. This sharing notation was introduced by Lahiri (\cite{br6}) in 2001.
\begin{defi} (\cite{br6})\label{daug2}
Let $k$ be a nonnegative integer or infinity. For $a\in\mathbb{C}\cup\{\infty\}$, we denote by $E_{k}(a;f)$ the set of all $a$-points of $f$, where an $a$-point of multiplicity $m$ is counted $m$ times if $m\leq k$ and $k+1$ times if $m>k$. If $E_{k}(a;f)=E_{k}(a;g)$, we say that $f$ and $g$ share the value $a$ with weight $k$.
\end{defi}
The definition implies that if $f$ and $g$ share a value $a$ with weight $k$, then $z_{0}$ is an $a$-point of $f$ with multiplicity $m\;(\leq k)$ if and only if it is an $a$-point of $g$ with multiplicity $m\;(\leq k)$ and $z_{0}$ is an $a$-point of $f$ with multiplicity $m\;(>k)$ if and only if it is an $a$-point of $g$ with multiplicity $n\;(>k)$, where $m$ is not necessarily equal to $n$.\par
We write $f$ and $g$ share $(a,k)$ to mean that $f$ and $g$ share the value $a$ with weight $k$. Clearly if $f$ and $g$ share $(a,k)$, then $f$ and $g$ share $(a,p)$ for any integer $p$ with $0\leq p<k$. Thus $f$ and $g$ share a value $a$ IM or CM if and only if $f$ and $g$ share $(a,0)$ or $(a,\infty)$ respectively.
\begin{defi}
For $S\subset\mathbb{C}\cup\{\infty\}$, we define $E_{f}(S, k)$ as $E_{f}(S, k)=\cup_{a\in S} E_{k}(a; f)$,
where $k$ is a non-negative integer or infinity.\par
A set $S$ for which two meromorphic functions $f$, $g$ satisfying $E_{f} (S, k) = E_{g}(S, k)$ becomes identical is called a \emph{unique range set of weight $k$} for meromorphic functions.\par
Thus it is clear that $E_{f}(S,\infty)=E_{f}(S)$ and $E_{f}(S,0)=\ol{E}_{f}(S)$.
\end{defi}
Using the concept of \emph{weighted set sharing}, in (\cite{am11}, \cite{ami1}, \cite{bl}), researchers reduced the weight of existing smallest URSM (resp. URSE) from CM to two.\par
Inspired by the \emph{Gross Question} (\cite{am4}), the set sharing problem was started which later shifted to-wards characterization of the polynomial backbone of different unique range sets. In the mean time, the invention of \enquote{Critical Injection Property} by Fujimoto (\cite{am2}, \cite{am2a}) further add essence to the research. In chapters four and five, we shall discuss in details about unique range sets and its generating polynomial.\par
Here we conclude our first chapter. Necessary things which are needed in this sequel will be discussed in the subsequent chapters.
\newpage
\chapter{Some further study on  Br\"{u}ck Conjecture}
\fancyhead[l]{Chapter 2}
\fancyhead[r]{Some further study on  Br\"{u}ck Conjecture}
\fancyhead[c]{}
\section{Introduction}
\par
In chapter one, we have already discussed about Nevanlinna's Five-Value Theorem. Almost fifty years later, in 1977, Rubel-Yang (\cite{br10a}) first highlighted for entire functions that under the special situation where $g$ is the derivative of $f$, one usually needs sharing of only two values CM for their uniqueness.
\begin{theo 2.A} (\cite{br10a}) Let $f$ be a non-constant entire function. If $f$ and $f^{'}$ share two distinct values $a$, $b$ CM, then $f^{'}\equiv f$.
\end{theo 2.A}
Two years later, Mues-Steinmetz (\cite{br10}) proved that actually in the above case one does not even need to consider the multiplicities.
\begin{theo 2.B} (\cite{br10}) Let $f$ be a non-constant entire function. If $f$ and $f^{'}$ share two distinct values $a$, $b$ IM, then $f^{'}\equiv f$.
\end{theo 2.B}
\begin{rem} The number two in the above theorem is sharp. For example, we take $f(z)=e^{e^{z}}\int_{0}^{z}e^{-e^{t}}(1-e^t)dt$. Then $f$ and $f'$ share the value $1$ IM, but $(f'-1)=e^{z}(f-1)$.
\end{rem}
Natural question would be to investigate \emph{the relation between an entire function and its derivative counterpart sharing one value CM}.
In 1996, in this direction, the following famous conjecture was proposed by Br\"{u}ck (\cite{br3}):\\
\par
{\it {\bf Conjecture:} Let $f$ be a non-constant entire function such that the hyper order $\rho _{_{2}}(f)$ of $f$ is not a positive integer or infinite. If $f$ and $f^{'}$ share a finite value $a$ CM, then $\frac{f^{'}-a}{f-a}=c$, where $c$ is a non-zero constant.} \\
\par
Br\"{u}ck himself proved the conjecture for $a=0$. For $a\not=0$, Br\"{u}ck (\cite{br3}) showed that under the assumption $N(r,0;f^{'})=S(r,f)$ the conjecture was true without any growth condition when $a=1$.
\begin{theo 2.C} (\cite{br3}) Let $f$ be a non-constant entire function. If $f$ and $f^{'}$ share the value $1$ CM and if $N(r,0;f^{'})=S(r,f)$, then $\frac{f^{'}-1}{f-1}$ is a non-zero constant.
\end{theo 2.C}
Let $a=a(z)$ be a small function. Then we say that $f$ and $g$ share $a$ CM (resp. IM) if $f-a$ and $g-a$ share $0$ CM (resp. IM). Now the following example shows the fact that one can not simply replace the value $1$ by a small function $a(z)(\not\equiv 0,\infty)$ in Theorem 2.C.
\begin{exm}\label{CHBex1.1} Let $f(z)=1+{e^{e^{z}}}$ and $a=a(z)=\frac{1}{1-e^{-z}}$.
\end{exm}
Then {\it Lemma 2.6} of (\cite{2}) yields that $a(z)$ is a small function of $f(z)$. Also it can be easily seen that $f$ and $f^{'}$ share $a(z)$ CM and $N(r,0;f^{'})=0$ but $f-a\not=c\;(f^{'}-a)$ for every non-zero constant $c$. So in this case additional suppositions are required.\par
However for entire function of finite order, in 1999, Yang (\cite{br11}) removed the supposition $N(r,0;f^{'})=0$ and obtained the following result.
\begin{theo 2.D} (\cite{br11}) Let $f$ be a non-constant entire function of finite order and let $a(\not=0)$ be a finite constant. If $f$, $f^{(k)}$ share the value $a$ CM, then $\frac{f^{(k)}-a}{f-a}$ is a non-zero constant, where $k(\geq 1)$ is an integer.
\end{theo 2.D}
{\it Theorem 2.D} may be considered as a solution to the Br\"{u}ck conjecture.
Next we consider the following example which show that in {\it Theorem 2.C}, one can not simultaneously replace \enquote{CM} by \enquote{IM} and \enquote{entire function} by \enquote{meromorphic function}.
\begin{exm}\label{CHBex1.3} Let $f(z)=\frac{2}{1-e^{-2z}}$.\end{exm}
Then $f(z)$ and $f^{'}(z)$ share $1$ IM and $N(r,0;f^{'})=0$ but $(f'-1)=(f-1)^{2}$.
Thus the conclusion of {\it Theorem 2.C} ceases to hold.\par
Thus from the above discussion it is natural to ask the following question.
\begin{ques} Can the conclusion of {\em Theorem 2.C} be obtained for a non-constant meromorphic function sharing a small function IM together with its $k$-th derivative counterpart? \end{ques}
We now recall the following two theorems due to Liu-Yang (\cite{br8b}) in the direction of IM sharing related to {\it Theorem 2.C}. It will be convenient to let $I$ denote any set of infinite linear measure of $r\in(0,\infty)$, not necessarily the same at each occurrence.
\begin{theo 2.E} (\cite{br8b}) Let $f$ be a non-constant meromorphic function. If $f$ and $f^{'}$ share $1$ IM and if \be\label{CHBe1.2} \ol N(r,\infty;f)+\ol N\left(r,0;f^{'}\right)<(\lambda +o(1))\;T\left(r,f^{'}\right)\ee for $r\in I$, where $0<\lambda <\frac{1}{4}$, then $\frac{f^{'}-1}{f-1}\equiv c$  for some constant $c\in\mathbb{C}/\{0\} $.
\end{theo 2.E}
\begin{theo 2.F} (\cite{br8b}) Let $f$ be a non-constant meromorphic function and $k$ be a positive integer. If $f$ and $f^{(k)}$ share $1$ IM and \be\label{CHBe1.3} (3k+6)\ol N(r,\infty;f)+5N(r,0;f)<(\lambda +o(1))\;T\left(r,f^{(k)}\right)\ee for $r\in I$, where $0<\lambda <1$, then $\frac{f^{(k)}-1}{f-1}\equiv c$  for some constant $c\in\mathbb{C}/\{0\}$.
\end{theo 2.F}
In the mean time, Zhang (\cite{br14}) studied {\it Theorem 2.C} for meromorphic function and also studied the CM value sharing of a meromorphic function with its $k$-th derivative.\par
In 2005, Zhang (\cite{br15}) further extended his results (\cite{br14}) in connection to Br\"{u}ck Conjecture to a small function and proved the following result for IM sharing. To state Zhang's (\cite{br15}) result, we need following two definitions.
\begin{defi} (\cite{br8}) Let $p$ be a positive integer and $a\in\mathbb{C}\cup\{\infty\}$.
\begin{enumerate}
\item [i)] $N(r,a;f\mid \geq p)$ \big(resp. $\ol N(r,a;f\mid \geq p)$\big) denotes the counting function (resp. reduced counting function) of those $a$-points of $f$ whose multiplicities are not less than $p$.
\item [ii)] $N(r,a;f\mid \leq p)$ \big(resp. $\ol N(r,a;f\mid \leq p)$\big) denotes the counting function (resp. reduced counting function) of those $a$-points of $f$ whose multiplicities are not greater than $p$.
\end{enumerate}
\end{defi}
\begin{defi} (\cite{br12})
 For $a\in\mathbb{C}\cup\{\infty\}$ and $p\in\mathbb{N}$, we denote by $N_{p}(r,a;f)$ the sum $\ol N(r,a;f)+\ol N(r,a;f\mid\geq 2)+\ldots+\ol{N}(r,a;f\mid\geq p)$. Clearly $N_{1}(r,a;f)=\ol N(r,a;f)$.
\end{defi}
\begin{theo 2.G} (\cite{br15}) Let $f$ be a non-constant meromorphic function and $k(\geq 1)$ be integer. Also let $a\equiv a(z)$ ($\not\equiv 0,\infty$) be a meromorphic small function. Suppose that $f-a$ and $f^{(k)}-a$ share $0$ IM.
If \be\label{CHBe1.1}4\ol N(r,\infty;f)+3N_{2}\left(r,0;f^{(k)}\right)+2\ol N\left(r,0;(f/a)^{'}\right)<(\lambda +o(1))\;T\left(r,f^{(k)}\right)\ee for $r\in I$, where $0<\lambda <1$, then $\frac{f^{(k)}-a}{f-a}=c$  for some constant $c\in\mathbb{C}/\{0\}$.
\end{theo 2.G}
In 2008, Zhang-L$\ddot u$ (\cite{br16}) further improved the result of Zhang (\cite{br15}) in connection to the Br\"{u}ck conjecture for the $n$-th power of a meromorphic function sharing a small function with its $k$-th derivative and obtained the following theorem.
\begin{theo 2.H} (\cite{br16}) Let $f$ be a non-constant meromorphic function and $k(\geq 1)$ and $n(\geq 1)$ be integers. Also let $a\equiv a(z)$ ($\not\equiv 0,\infty$) be a meromorphic small function. Suppose that $f^{n}-a$ and $f^{(k)}-a$ share $0$ IM. If  \be\label{CHBe1.4}4\ol N(r,\infty;f)+\ol N\left(r,0;f^{(k)}\right)+2N_{2}\left(r,0;f^{(k)}\right)+2\ol N\left(r,0;(f^{n}/a)^{'}\right)<(\lambda +o(1))\;T\left(r,f^{(k)}\right)\ee for $r\in I$, where $0<\lambda <1$, then $\frac{f^{(k)}-a}{f^{n}-a}=c$  for some constant $c\in\mathbb{C}/\{0\}$.
\end{theo 2.H}
At the end of the paper (\cite{br16}), Zhang-L$\ddot u$ (\cite{br16}) raised the following question:
\begin{ques}
What will happen if $f^{n}$ and $[f^{(k)}]^{m}$ share a small function?
\end{ques}
In order to answer the above question, Liu (\cite{br8a}) obtained the following result.
\begin{theo 2.I} (\cite{br8a}) Let $f$ be a non-constant meromorphic function and $k(\geq 1)$, $n(\geq 1)$ and $m(\geq 2)$ be integers. Also let $a\equiv a(z)$ ($\not\equiv 0,\infty$) be a meromorphic small function. Suppose that $f^{n}-a$ and $(f^{(k)})^{m}-a$ share $0$ IM. If \be\label{CHBe1.7}\frac{4}{m}\ol N(r,\infty;f)+\frac{5}{m}\ol N\left(r,0;f^{(k)}\right)+\frac{2}{m}\ol N\left(r,0;(f^{n}/a)^{'}\right)<(\lambda +o(1))\;T\left(r,f^{(k)}\right)\ee for $r\in I$, where $0<\lambda <1$, then $\frac{(f^{(k)})^{m}-a}{f^{n}-a}=c$  for some constant $c\in\mathbb{C}/\{0\}$.
\end{theo 2.I}
As $(f^{(k)})^{m}$ is a simple form of differential monomial in $f$, it will be interesting to investigate whether {\it Theorems 2.G - 2.I} can be extended up to differential polynomial generated by $f$.
In this direction, Li-Yang (\cite{br8c}) improved {\it Theorem 2.G} as follows.
\begin{theo 2.J} (\cite{br8c}) Let $f$ be a non-constant meromorphic function and $P[f]$ be a differential polynomial generated by $f$. Also let $a\equiv a(z)$ ($\not\equiv 0,\infty$) be a small meromorphic function. Suppose that $f-a$ and $P[f]-a$ share $0$ IM and $(t-1)\ol {d}(P)\leq \sum\limits_{j=1}^{t}d(M_{j})$. If  \bea\label{CHBe1.8} 4\ol N(r,\infty;f)+3N_{2}\left(r,0;P[f]\right)+2\ol N\left(r,0;(f/a)^{'}\right)<(\lambda +o(1))\;T\left(r,P[f]\right)\eea for $r\in I$, where $0<\lambda <1$, then $\frac{P[f]-a}{f-a}=c$  for some constant $c\in\mathbb{C}/\{0\}$.
\end{theo 2.J}
We see that \emph{Theorem 2.J always holds for a differential monomial without any condition on its degree}. But for a general differential polynomial one can not eliminate the supposition $(t-1)\ol {d}(P)\leq \sum_{j=1}^{t}d(M_{j})$ in the Theorem 2.J. So the following questions are open:
\begin{enumerate}
\item [i)] whether in {\it Theorem 2.J}, the condition over the degree can be removed,
\item [ii)] sharing notion can further be relaxed,
\item [iii)] inequality (\ref{CHBe1.8}) can further be weakened.
\end{enumerate}
The main aim of this chapter is to obtain the possible answer of the above questions in such a way that it improves, unifies and generalizes all the {\it Theorems 2.G - 2.J}.
\section{Main Result}
\par
\begin{theo}\label{CHBt1} Let $f$ be a non-constant meromorphic function and $n(\geq 1)$ be an integer. Let $m(\geq 1)$ be a positive integer or infinity and $a\equiv a(z)$ ($\not\equiv 0,\infty$) be a small function with respect to $f$. Suppose that $P[f]$ be a differential polynomial generated by $f$ such that $P[f]$ contains at least one derivative. Further suppose that $\ol E_{m)}(a;f^{n})=\ol E_{m)}(a;P[f])$. If
\bea\label{CHBe1.9} && 4\ol N(r,\infty;f)+N_{2}\left(r,0;P[f]\right)+2\ol N\left(r,0;P[f]\right)+\ol N\left(r,0;(f^{n}/a)^{'}\right)\\&& + \ol N\left(r,0;(f^{n}/a)^{'}\mid (f^{n}/a)\not =0\right)<(\lambda +o(1))\;T\left(r,P[f]\right)\nonumber\eea for $r\in I$, where $0<\lambda <1$, then $\frac{P[f]-a}{f^{n}-a}=c$  for some constant $c\in\mathbb{C}\setminus\{0\}$.
\end{theo}
\begin{rem}\label{CHBr1.1} Thus if we put $m=\infty$ in {\em Theorem {\ref{CHBt1}}}, then we note that $f^{n}-a$ and $P[f]-a$ share $0$ IM and hence we obtain the improved, extended and generalized version of {\em Theorem 2.J} in the direction of {\em Question 2.1.1}.
\end{rem}
Now, we explain some definitions and notations which we need in this sequel.
\begin{defi} For $k\in\mathbb{N}\cup\{\infty\}$ and $a\in\mathbb{C}\backslash\{0\}$, let $\ol E_{k)}(a;f)=\ol E_{k)}(a;g)$. If $z_{0}$ be a zero of $f(z)-a$ of multiplicity $p$ and a zero of $g(z)-a$ of multiplicity $q$, then
\begin{enumerate}
\item [i)] by $\ol N_{L}(r,a;f)$ \big(resp. $\ol N_{L}(r,a;g)$\big), we mean the counting function of those $a$-points of $f$ and $g$ where $p>q\geq 1$ \big(resp. $q>p\geq 1$\big),
\item [ii)] by $N^{1)}_{E}(r,a;f)$, we mean the counting function of those $a$-points of $f$ and $g$ where $p=q=1$,
\item [iii)] by $\ol N^{(2}_{E}(r,a;f)$, we mean the counting function of those $a$-points of $f$ and $g$ where $p=q\geq 2$, each point in these counting functions is counted only once,
\item [iv)] by $\ol N_{f>s}(r,a;g)$ \big(resp. $\ol N_{g>s}(r,a;f)$\big), we mean the counting functions of those $a$-points of $f$ and $g$ for which $p>q=s$ (resp. $q>p=s$) and
\item [v)] by $\ol N_{f\geq k+1}(r,a;f\mid\; g\not=a)$ \big(resp. $\ol N_{g\geq k+1}(r,a;g\mid\; f\not=a)$\big), we mean the reduced counting functions of those $a$-points of $f$ and $g$ for which $p\geq k+1$ and $q=0$ (resp. $q\geq k+1$ and $p=0$).
\end{enumerate}
Clearly $N^{1)}_{E}(r,a;f)=N^{1)}_{E}(r,a;g)$ and $\ol N^{(2}_{E}(r,a;f)=\ol N^{(2}_{E}(r,a;g).$
\end{defi}
\begin{defi} (\cite{br7}) Let $a,b \in\mathbb{C}\;\cup\{\infty\}$. We denote by $N(r,a;f\mid\; g\neq b)$ the counting function of those $a$-points of $f$, counted according to multiplicity, which are not the $b$-points of $g$.
\end{defi}
\begin{defi} (\cite{br6}) Let $f$, $g$ share a value $a$ IM. We denote by $\ol N_{*}(r,a;f,g)$ the reduced counting function of those $a$-points of $f$ whose multiplicities differ from the multiplicities of the corresponding $a$-points of $g$.\\
Clearly $\ol N_{*}(r,a;f,g)\equiv\ol N_{*}(r,a;g,f)$ and $\ol N_{*}(r,a;f,g)=\ol N_{L}(r,a;f)+\ol N_{L}(r,a;g)$.
\end{defi}
\section{Lemmas}
For any two non-constant meromorphic functions $F$ and $G$, we shall denote by $H$ the following function: \be\label{CHB} H=\left(\frac{\;\;F^{''}}{F^{'}}-\frac{2F^{'}}{F-1}\right)-\left(\frac{\;\;G^{''}}{G^{'}}-\frac{2G^{'}}{G-1}\right).
\ee
Throughout this chapter, we take $F=\frac{f^{n}}{a}$ and $G=\frac{P[f]}{a}$. Now we present some lemmas which will be needed in the sequel.
\begin{lem}\label{CHBl2.1a} Let $\ol E_{m)}(1;F)=\ol E_{m)}(1;G)$; $F$, $G$ share $\infty$ IM and $H\not\equiv 0$. Then
\beas && N(r,\infty;H)\\
&\leq&\ol N(r,0;F\mid\geq 2)+\ol N(r,0;G\mid\geq 2)+\ol N_{*}(r,\infty;F,G)+\ol N_{F\geq m+1}(r,1;F \mid G\not=1)\\
&+&\ol N_{G\geq m+1}(r,1;G\mid F\not=1)+\ol N_{*}(r,1;F,G)+\ol N_{0}(r,0;F^{'})+\ol N_{0}(r,0;G^{'}),
\eeas where $\ol N_{0}(r,0;F^{'})$ is the reduced counting function of those zeros of $F^{'}$ which are not the zeros of $F(F-1)$ and $\ol N_{0}(r,0;G^{'})$ is similarly defined.
\end{lem}
\begin{proof} We can easily verify that possible poles of $H$ occur at (i) multiple zeros of $F$ and $G$, (ii) poles of $F$ and $G$ with different multiplicities, (iii) the common zeros of $F-1$ and $G-1$ with different multiplicities, (iv) zeros of $F-1$ which are not the zeros of $G-1$, (v)  zeros of $G-1$ which are not the zeros of $F-1$, (vi) zeros of $F^{'}$ which are not the zeros of $F(F-1)$, (vii) zeros of $G^{'}$ which are not zeros of $G(G-1)$. Since $H$ has simple pole the lemma follows from above.
\end{proof}
\begin{lem}\label{CHBl2.2} (\cite{br7a}) If $N(r,0;f^{(k)}\mid f\not=0)$ denotes the counting function of those zeros of  $f^{(k)}$ which are not the zeros of $f$, where a zero of $f^{(k)}$ is counted according to its multiplicity, then $$N(r,0;f^{(k)}\mid f\not=0)\leq k\ol N(r,\infty;f)+N(r,0;f\mid <k)+k\ol N(r,0;f\mid\geq k)+S(r,f).$$
\end{lem}
\begin{lem}\label{CHBl2.1} (\cite{br15}) Let $f$ be a non-constant meromorphic function and $k\in\mathbb{N}$, then $$N_{p}(r,0;f^{(k)})\leq N_{p+k}(r,0;f)+k\ol N(r,\infty;f)+S(r,f).$$
\end{lem}
The following lemma is known as Mokhon'ko's Lemma which plays a vital role throughout the thesis.
\begin{lem}\label{ML} (\cite{br9}) Let $f$ be a non-constant meromorphic function and let $P(f)=\sum\limits _{k=0}^{n} a_{k}f^{k}$ and $Q(f)=\sum \limits_{j=0}^{m} b_{j}f^{j}$ be two mutually prime polynomials in $f$. If the coefficients $\{a_{k}\}$ and $\{b_{j}\}$ are small functions of $f$ and  $a_{n}\not=0~,~b_{m}\not=0$, then $$T\bigg(r,\frac{P(f)}{Q(f)}\bigg)=\max\{n,m\}\;T(r,f)+S(r,f).$$
\end{lem}
\begin{lem} \label{CHBl2.4} (\cite{br3a}) Let $f$ be a meromorphic function and $P[f]$ be a differential polynomial. Then
$$ m\left(r,\frac{P[f]}{f^{\ol {d}(P)}}\right)\leq \left(\ol {d}(P)-\underline {d}(P)\right) m\left(r,\frac{1}{f}\right)+S(r,f).$$
\end{lem}
\begin{lem} \label{CHBl2.5} Let $f$ be a meromorphic function and $P[f]$ be a differential polynomial. Then we have
\beas N\left(r,\infty;\frac{P[f]}{f^{\ol {d}(P)}}\right)&\leq& (\Gamma _{P}-\ol {d}(P))\;\ol N(r,\infty;f)+(\ol {d}(P)-\underline {d} (P))\; N(r,0;f\mid\geq k+1)\\&&+Q\ol N(r,0;f\mid\geq k+1)+\ol{d}(P)N(r,0;f|\leq k)+S(r,f).\eeas
\end{lem}
\begin{proof} Let $z_{0}$ be a pole of $f$ of order $r$, such that $b_{j}(z_{0})\not=0,\infty$ ($1\leq j\leq t$). Then it would be a pole of $P[f]$ of order at most $r\ol {d}(P)+\Gamma _{P}-\ol {d}(P)$. Since $z_{0}$ is a pole of $f^{\ol {d}(P)}$ of order $r\ol {d}(P)$, it follows that $z_{0}$ would be a pole of $\frac{P[f]}{f^{\ol {d}(P)}}$ of order at most $\Gamma _{P}-\ol {d}(P)$. Next suppose $z_{1}$ is a zero of $f$ of order $s(>k)$, such that $b_{j}(z_{1})\not=0,\infty$ ($1\leq j\leq t$). Clearly it would be a zero of $M_{j}(f)$ of order $s.n_{0j}+(s-1)n_{1j}+\ldots+(s-k)n_{kj}= s.d(M_{j})-(\Gamma _{M_{j}}-d(M_{j}))$.
 Hence $z_{1}$ be a pole of $\frac{M_{j}[f]}{f^{\ol {d}(P)}}$ of order $$s.\ol {d}(P)-s.d(M_{j})+(\Gamma _{M_{j}}-d(M_{j}))=s(\ol {d}(P)-d(M_{j}))+(\Gamma _{M_{j}}-d(M_{j})).$$ So $z_{1}$ would be a pole of $\frac{P[f]}{f^{\ol {d}(P)}}$ of order at most $$\text{max} \{s(\ol {d}(P)-d(M_{j}))+(\Gamma _{M_{j}}-d(M_{j})): 1\leq j\leq t)\}=s(\ol {d}(P)-\underline{d}(P))+Q.$$
Since the poles of $\frac{P[f]}{f^{\ol {d}(P)}}$ comes from the poles or zeros of $f$ and poles or zeros of $b_{j}(z)$ ($1\leq j\leq t$) only, it follows that
\beas  N\left(r,\infty;\frac{P[f]}{f^{\ol {d}(P)}}\right)&\leq& (\Gamma _{P}-\ol {d}(P))\;\ol N(r,\infty;f)+(\ol {d}(P)-\underline {d} (P))\; N(r,0;f\mid \geq k+1)\\&&+Q\;\ol N(r,0;f\mid\geq k+1)+\ol{d}(P)N(r,0;f|\leq k)+S(r,f).\eeas
Hence the proof.
\end{proof}
\begin{lem} \label{CHBl2.7} Let $f$ be a non-constant meromorphic function and $P[f]$ be a differential polynomial. Then $T(r,P[f])=O(T(r,f))$ and $S(r,P[f])=S(r,f)$.
\end{lem}
\begin{proof} For a differential polynomial $P[f]$, from (\cite{br3b}), we know that $T(r,P[f])\leq\Gamma_{P}T(r,f)+S(r,f)$. From this inequality, the lemma follows.
\end{proof}
\begin{lem} \label{newlem1} If $\ol E_{m)}(1;F)=\ol E_{m)}(1;G)$ except the zeros and poles of $a(z)$ and $ H\not\equiv0$, then
\bea T(r,G) &\leq& 4\ol{N}(r,\infty;F)+2\ol{N}(r,0;G)+N_{2}(r,0;G)\\
\nonumber &+& \ol{N}(r,0;F')+\ol{N}(r,0;F'\mid F\not=0)+S(r,f).\eea
\end{lem}
\begin{proof}
Let $z_{0}$ be a simple zero of $F-1$. Then by a simple calculation, we see that $z_{0}$ is a zero of $H$ and hence \be\label{CHBe3.2}N^{1)}_{E}(r,1;F)=N^{1)}_{E}(r,1;G)\leq N(r,0;H)\leq N(r,\infty;H)+S(r,F).
\ee
We note that $\ol N(r,\infty;F)=\ol N(r,\infty;G)+S(r,f)$ and $\ol N_{F>1}(r,1;G)+\ol N(r,1;G\mid\geq 2)=\ol N_{E}^{(2}(r,1;G)+\ol N_{L}(r,1;G)+\ol N_{L}(r,1;F)+\ol N_{G\geq m+1}(r,1;G\mid\;F\not=1)+S(r,f)$.\par
Using (\ref{CHBe3.2}), {\it Lemmas \ref{CHBl2.1a}}, {\it \ref{CHBl2.7}} and the Second Fundamental Theorem, we obtain
\bea\label{CHBe3.3}&& T(r,G)\\&\leq& \ol N(r,\infty;G)+\ol N(r,0;G)+N^{1)}_{E}(r,1;G)+\ol N_{F>1}(r,1;G)+\ol N(r,1;G\mid\geq 2)\nonumber\\&&-N_{0}(r,0;G^{'})+S(r,G)\nonumber\\&\leq& 2\ol N(r,\infty;F)+\ol N(r,0;G)+\ol N(r,0;G\mid\geq 2)+\ol N(r,0;F\mid\geq 2)+2\ol N_{L}(r,1;F)\nonumber\\& &+2\ol N_{L}(r,1;G)+\ol N_{F\geq m+1}(r,1;F\mid\;G\not=1)+2\ol N_{G\geq m+1}(r,1;G\mid\;F\not=1)\nonumber\\& &+\ol N_{E}^{(2}(r,1;G)+\ol N_{0}(r,0;F^{'})+S(r,f).\nonumber
\eea
Using {\it Lemmas \ref{CHBl2.1}}, {\it \ref{CHBl2.2}}, we see that
\bea \label{CHBe3.4}\nonumber&& \ol N(r,0;G\mid \geq 2)+2\ol N_{G\geq m+1}(r,1;G\mid\;F\not=1)+2\ol N_{L}(r,1;G)+\ol N_{E}^{(2}(r,1;G)\\
\nonumber&\leq& \ol N(r,0;G^{'}\mid G\not=0)+\ol N(r,0;G^{'})+S(r,f)\\
&\leq& 2\ol N(r,\infty;G)+\ol N(r,0;G)+N_{2}(r,0;G)+S(r,f),\eea
and \bea \label{CHBe3.5}\nonumber&& \ol N(r,0;F\mid\geq 2)+\ol N_{F\geq m+1}(r,1;F\mid\;G\not=1)+2\ol N_{L}(r,1;F)+\ol N_{0}(r,0;F^{'})\\
&\leq& \ol N(r,0;F^{'}\mid F\not=0)+\ol N(r,0;F^{'})+S(r,f).\eea
Using (\ref{CHBe3.4}) and (\ref{CHBe3.5}) in (\ref{CHBe3.3}), we have
\beas T(r,G) &\leq& 4\ol{N}(r,\infty;F)+2\ol{N}(r,0;G)+N_{2}(r,0;G)+\ol{N}(r,0,F')\\
&+& \ol{N}(r,0,F'\mid F\not=0)+S(r,f).\eeas
Hence the proof.
\end{proof}
\newpage
\begin{lem} \label{newlem2} If $ H\equiv0$, then  one of the following conditions hold:
\begin{enumerate}
\item [i)] $T(r,G)=N_{2}(r,0;G)+S(r,G)$, or
\item [ii)] $(F-1-\frac{1}{C})G\equiv-\frac{1}{C}$ for some constant $C\in \mathbb{C}\setminus\{0\}$, or
\item [iii)] $\frac{G-1}{F-1}\equiv C$ for some constant $C\in\mathbb{C}\setminus\{0\}$.
\end{enumerate}
\end{lem}
\begin{proof}
Given $H\equiv 0$. On integration, we get from (\ref{CHB}), \be\label{CHBe3.8}\frac{1}{F-1}\equiv\frac{C}{G-1}+D,\ee where $C$, $D$ are constants and $C\not=0$. From (\ref{CHBe3.8}) it is clear that $F$ and $G$ share $1$ CM.\\
Next we discuss following two cases.\\
\textbf{Case-1.} If $D\not=0$, then by (\ref{CHBe3.8}), we get
\be\label{CHBe3.8a}\ol N(r,\infty;f)=S(r,f),\ee
and
\be\label{CHBe3.9}\frac{1}{F-1}=\frac{D\left(G-1+\frac{C}{D}\right)}{G-1}.\ee
Clearly from (\ref{CHBe3.9}), we have \be\label{CHBe3.10}\ol N\left(r,1-\frac{C}{D};G\right)=\ol N(r,\infty;F)=\ol N(r,\infty;G)=\ol N(r,\infty;f)=S(r,f).\ee
If $\frac{C}{D}\not=1$, then the Second Fundamental Theorem, {\it Lemma \ref{CHBl2.7}} and (\ref{CHBe3.10}) yields
\beas T(r,G)&\leq& \ol N(r,\infty;G)+\ol N(r,0;G)+\ol N\left(r,1-\frac{C}{D};G\right)+S(r,G)\\&\leq&\ol N(r,0;G)+S(r,f)\leq N_{2}(r,0;G)+S(r,f)\\&\leq& T(r,G)+S(r,f),
\eeas
which gives (i); i.e., $T(r,G)= N_{2}(r,0;G)+S(r,f)$.\\
If $\frac{C}{D}=1$, then from (\ref{CHBe3.10}) we obtain (ii); i.e., $\left(F-1-\frac{1}{C}\right)G\equiv -\frac{1}{C}$.\\
\textbf{Case-2.} If $D=0$, from (\ref{CHBe3.8}), we get (iii); i.e., $\frac{G-1}{F-1}=C$. Hence the lemma.
\end{proof}
\begin{lem} \label{newlem3} If $n\geq 1$, then $(F-1-\frac{1}{C})G\not\equiv-\frac{1}{C}$ for any non-zero complex constant $C$, where $F$ and $G$ are defined previously.
\end{lem}
\begin{proof} If possible, let for for some non-zero complex constant $C$,
\bea\label{CHBeg3.12}  \left(F-1-\frac{1}{C}\right)G\equiv -\frac{1}{C},\eea
Then  from (\ref{CHBeg3.12}), we get that $\ol{N}(r,\infty;f)=S(r,f)$ and
\be \label{CHBe3.13} \ol{N}(r,0;f\mid \geq k+1)\leq N(r,0;P[f])\leq N(r,0;G)\leq N(r,0;a)=S(r,f).\ee
Again from (\ref{CHBeg3.12}), we see that $$\frac{1}{f^{\ol {d}(P)}\left(f^{n}-(1+1/C)a\right)}\equiv -\;\frac{C}{a^{2}}\;\;\frac{P[f]}{f^{\ol {d}(P)}}.$$
Using First Fundamental Theorem, (\ref{CHBe3.13}), {\it Lemmas \ref{ML}}, {\it \ref{CHBl2.4}} and {\it \ref{CHBl2.5}}, we get
\bea\label{CHBe3.14} (n+\ol {d}(P))T(r,f)&=&T\left(r,f^{\ol {d}(P)}(f^{n}-(1+\frac{1}{C})a)\right)+S(r,f)\nonumber\\
&=& T\left(r,\frac{1}{f^{\ol {d}(P)}(f^{n}-(1+\frac{1}{C})a)}\right)+S(r,f)\nonumber\\
&=&T\left(r,\frac{P[f]}{f^{\ol {d}(P)}}\right)+S(r,f)\nonumber\\&\leq& m\left(r,\frac{P[f]}{f^{\ol {d}(P)}}\right)+N\left(r,\frac{P[f]}{f^{\ol {d}(P)}}\right)+S(r,f)\nonumber\\
&\leq& \left(\ol {d}(P)-\underline {d}(P)\right) T(r,f)-\left(\ol {d}(P)-\underline {d}(P)\right)N(r,0;f\mid\leq k)\nonumber\\&&+\ol {d}(P)\;N(r,0;f\mid\leq k)+Q\;\ol N(r,0;f\mid\geq k+1)+S(r,f)\nonumber\\&\leq& (\ol {d}(P)-\underline {d}(P)) T(r,f)+\underline {d}(P)N(r,0;f|\leq k)+S(r,f)\nonumber\\&\leq& \ol {d}(P) T(r,f)+S(r,f),\nonumber\eea
which is absurd. Hence the proof.
\end{proof}
\section{Proof of the theorem}
\begin{proof} [\textbf{Proof of Theorem \ref{CHBt1}}]
Here $F-1=\frac{f^{n}-a}{a}$ and $G-1=\frac{P[f]-a}{a}$. Since $\ol E_{m)}(a,f^{n})=\ol E_{m)}(a,P[f])$, it follows that $\ol E_{m)}(1,F)=\ol E_{m)}(1,G)$ except the zeros and poles of $a(z)$.\par
First we suppose that $H\not\equiv 0$. Then in view of Lemma \ref{newlem1}, we have
\beas T\left(r,P[f]\right)&\leq& 4\ol N(r,\infty;f)+2\ol N\left(r,0;P[f]\right)+N_{2}\left(r,0;P[f]\right)+ \ol N\left(r,0;(f/a)^{'}\right)\\&&+\ol N\left(r,0;(f^{n}/a)^{'}\mid (f^{n}/a)\not=0\right)+S(r,f),
\eeas
which contradicts (\ref{CHBe1.9}).\par
Thus $H\equiv 0$. Now applying the Lemma \ref{newlem3} and condition (\ref{CHBe1.9}) in Lemma \ref{newlem2}, we get
$$\frac{G-1}{F-1}=C;~~\text{i.e.},~~\frac{P[f]-a}{f^{n}-a}=C.$$ This proves the theorem.
\end{proof}
$~~$
\vspace{3.5 cm}
\\
------------------------------------------------
\\
\textbf{The matter of this chapter has been published in An. S\c{t}iin\c{t}. Univ. Al. I. Cuza Ia\c{s}i Mat. (N.S.), Vol. 62, No. 2, f.2 (2016), pp. 501-511.}
\newpage
\chapter{On the generalizations of Br\"{u}ck Conjecture}
\fancyhead[l]{Chapter 3}
\fancyhead[r]{On the generalizations of Br\"{u}ck Conjecture}
\fancyhead[c]{}
\section{Introduction}
\par
This chapter is the continuation of chapter two. One can note that the afterward research on Br\"{u}ck conjecture and its generalization, one setting among the sharing functions has been restricted to only various powers of $f$, not involving any other variants such as derivatives of $f$, where as the generalization have been made on the second setting. In fact, in chapter two, we have encountered this situation. So the natural query would be:
\begin{ques}\label{anikul} Can Br\"{u}ck type conclusion be obtained when two different differential polynomials share a small functions IM or even under relaxed sharing notions?\end{ques}
The content of this chapter has been oriented to obtain the possible answer of the above question in such a way that it improves and generalizes Theorem \ref{CHBt1}.\par
Henceforth by $b_{j}~(j=1,2,\ldots, t)$ and $c_{i}~(i=1,2,\ldots, l)$,
 we denote small functions in $f$ and we also suppose that $P[f]=\sum_{j=1}^{t}b_{j}M_{j}[f]$ and $Q[f]=\sum_{i=1}^{l}c_{i}M_{i}[f]$ be two differential polynomials generated by $f$.
\section{Main Result}
\begin{theo}\label{brt1} Let $f$ be a non-constant meromorphic function, $m(\geq 1)$ be a positive integer or infinity and $a\equiv a(z)$ ($\not\equiv 0,\infty$) be a small function of $f$. Suppose that $P[f]$ and $Q[f]$ be two differential polynomials generated by $f$ such that $Q[f]$ contains at least one derivative. Suppose further that $\ol E_{m)}(a;P[f])=\ol E_{m)}(a;Q[f])$. If
\bea\label{bre1.9}&& 4\ol N(r,\infty;f)+N_{2}\left(r,0;Q[f]\right)+2\ol N\left(r,0;Q[f]\right)+\ol N\left(r,0;(P[f]/a)^{'}\right)\\&&+\ol N\left(r,0;(P[f]/a)^{'}\mid (P[f]/a)\not =0\right)<(\lambda +o(1))\;T\left(r,Q[f]\right)\nonumber\eea for $r\in I$, where $0<\lambda <1$, then either
\begin{enumerate}
\item $\frac{Q[f]-a}{P[f]-a}=c,$ for some constant $c\in\mathbb{C}\setminus\{0\}$, or
\item $P[f]Q[f]-aQ[f](1+d)\equiv -da^{2}$ for a non-zero constant $d\in \mathbb{C}$.
\end{enumerate}
In particular, if
\begin{enumerate}
\item [i)] $P[f]=b_{1}f^{n}+b_{2}f^{n-1}+b_{3}f^{n-2}+\ldots+b_{t-1}f$, or
\item [ii)] $\underline {d}(Q) >2\ol {d}(P)-\underline {d}(P)$ and each monomial of $Q[f]$ contains a term involving a power of $f$,
\end{enumerate}
then the conclusion (2) does not hold.
\end{theo}
\begin{rem}\label{brr1.1} Thus if we put $m=\infty$ in {\em Theorem {\ref{brt1}}}, then $P[f]-a$ and $Q[f]-a$ share $0$ IM where $P[f]=b_{1}f^{n}+b_{2}f^{n-1}+b_{3}f^{n-2}+\ldots+b_{t-1}f$ and we obtain the improved, extended and generalized version of Theorem \ref{CHBt1} in the direction of Question \ref{anikul}.\end{rem}
Following examples show that condition (\ref{bre1.9}) of Theorem \ref{brt1} is not necessary.
\begin{exm}\label{brex1.6}  Let $f(z)=\frac{e^{z}}{e^{z}+1}$. $P[f]=f-f^{'}$, $Q[f]=f^{2}-3f{f^{'}}^{3}+f^{3}{f^{'}}^{2}-ff^{'}f^{'''}+ff^{'}f^{''}$. Then clearly $P[f]$ and $Q[f]$ share $1$ CM and $\frac{Q[f]-1}{P[f]-1}=1$, but condition (\ref{bre1.9})  is not satisfied. Here we note that $2=\underline{d}(Q)>2\ol{d}(P)-\underline{d}(P)=1$.\end{exm}
\begin{exm}\label{brex1.8}  Let $f(z)=\frac{1}{e^{z}+1}$. $P[f]={f^{'}}^{2}$, $Q[f]=ff^{''}-f^{2}f^{'}$. Then clearly $P[f]=Q[f]=\frac{e^{2z}}{{(e^{z}+1)}^{4}}$ share $\frac{1}{z}$ CM and $\frac{Q[f]-\frac{1}{z}}{P[f]-\frac{1}{z}}=1$, but condition (\ref{bre1.9})  is not satisfied.
\end{exm}
Next two examples show that both the conditions stated in (ii) are essential in order to obtain conclusion (1) in Theorem \ref{brt1} if $P[f]$ is a differential polynomial.
\begin{exm}\label{brex1.10}  Let $f(z)=\sin z$. $P[f]=3{f}^{2}+{f^{'}}^{2}-2iff^{'}$, $Q[f]=(f')^{2}-2iff^{'}-f^{2}$. Then clearly $P[f]=2-e^{2iz}$ and $Q[f]=e^{-2iz}$ share $1$ CM. Here condition (\ref{bre1.9}) is satisfied, but $\frac{Q[f]-1}{P[f]-1}=e^{-2iz}$, rather $P[f]Q[f]-2Q[f]+1=0$.
\end{exm}
\begin{exm}\label{brex1.13}  Let $f(z)=\cos z$. $P[f]=-f-if^{'}+(1+i){f^{'}}^{2}+(1+i){f^{''}}^{2}$, $Q[f]=if-f^{'''}$. Then clearly $P[f]=1+i-e^{-iz}$ and $Q[f]=ie^{iz}$ share  $i$ CM. Here condition (\ref{bre1.9}) is satisfied, but $\frac{Q[f]-i}{P[f]-i}=ie^{iz}$, rather $P[f]Q[f]-(1+i)Q[f]+i=0$. We also note that $\ol{d}(P)\not =\underline{d}(P)$ and $\underline{d}(Q)\not>2\ol{d}(P)-\underline{d}(P)$.
\end{exm}
The following two examples show that in order to obtain conclusions (1) or (2) of {\it Theorem \ref{brt1}}, condition (\ref{bre1.9}) is essential.
\begin{exm}\label{brex1.14}  Let $f(z)=\sin z$. $P[f]=if+f^{'}$, $Q[f]=2f^{'}-(f^{2}+{f^{'}}^{2})$. Then $P[f]=e^{iz}$ and $Q[f]=e^{iz}+e^{-iz}-1$ share $1$ IM. Here neither of the conclusions of {\em Theorem \ref{brt1}} is satisfied, nor condition (\ref{bre1.9}) is satisfied. We note that $\frac{Q[f]-1}{P[f]-1}=\frac{(e^{iz}-1)}{e^{iz}}$ and $P[f]Q[f]-\lambda Q[f]$ is non-constant function for any $\lambda\in\mathbb{C}$.
\end{exm}
\begin{exm}\label{brex1.15}  Let $f(z)=\cos z$. $P[f]=f-if^{'}$, $Q[f]=2f-({f^{'}}^{2}+{f^{''}}^{2})$. Then $P[f]=e^{iz}$ and $Q[f]=e^{iz}+e^{-iz}-1$ share $1$ IM. Here neither of the conclusions of {\em Theorem \ref{brt1}} is satisfied, nor condition (\ref{bre1.9}) is satisfied. We note that $\frac{Q[f]-1}{P[f]-1}=\frac{(e^{iz}-1)}{e^{iz}}$ and $P[f]Q[f]-\lambda Q[f]$ is non-constant function for any $\lambda\in\mathbb{C}$.
\end{exm}
\section{Lemmas}
Throughout this chapter, we take $F=\frac{P[f]}{a}$, $G=\frac{Q[f]}{a}$ and  $H$ is defined by the equation (\ref{CHB}). Now we present some lemmas which will be needed in this sequel.
\begin{lem} \label{newlem4} If $\ol E_{m)}(1;F)=\ol E_{m)}(1;G)$ except the zeros and poles of $a(z)$ and $ H\not\equiv0$, then
\bea T(r,G) &\leq& 4\ol{N}(r,\infty;F)+2\ol{N}(r,0;G)+N_{2}(r,0;G)\\
\nonumber &+& \ol{N}(r,0;F')+\ol{N}(r,0;F'\mid F\not=0)+S(r,f).\eea
\end{lem}
\begin{proof}
Proof is similar to the proof of Lemma \ref{newlem1}.
\end{proof}
\begin{lem} \label{brl2.8} Let $f$ be a non-constant meromorphic function and $P[f]$, $Q[f]$ be two differential polynomials. Then
\beas  N(r,0;P[f]) &\leq& \frac{\ol {d}(P)-\underline {d}(P)}{\underline {d}(Q)} m\left(r,\frac{1}{Q[f]}\right)+(\Gamma _{P}-\ol {d}(P))\;\ol N(r,\infty;f)\\
&&+(\ol {d}(P)-\underline {d} (P))\; N(r,0;f\mid\geq k+1)+\mu \ol N(r,0;f\mid\geq k+1)\\
&&+\ol {d}(P) N(r,0;f\mid\leq k)+\ol{d}(P)N(r,0;f)+S(r,f).\eeas
\end{lem}
\begin{proof} For a fixed value of $r$, let $E_{1}=\{\theta \in [0,2\pi]: \left|f(re^{i\theta })\right|\leq 1 \}$ and  $E_{2}$ be its complement. Since by definition $$\sum\limits_{i=0}^{k} n_{ij}\geq \underline {d}(Q),$$ for every $j=1,2,\ldots,l$; it follows that on $E_{1}$
$$\left|\frac{Q[f]}{f^{\underline{d}(Q)}}\right| \leq \sum\limits_{j=1}^{l}\left| c_{j}(z)\right| \prod\limits_{i=1}^{k}\left| \frac{f^{(i)}}{f} \right|^{n_{ij}} \left|f\right|^{^{\sum\limits_{i=0}^{k}n_{}{ij}-\underline{d}(Q)}}\leq \sum\limits_{j=1}^{l}\left| c_{j}(z)\right| \prod\limits_{i=1}^{k}\left| \frac{f^{(i)}}{f} \right|^{n_{ij}}.$$
Also we note that $$\frac{1}{f^{\underline{d}(Q)}}=\frac{Q[f]}{f^{\underline{d}(Q)}}\;\frac{1}{Q[f]}.$$
Since on $E_{2}$, $\frac{1}{\left|f(z)\right|}< 1$, we have
\beas && \underline{d}(Q)m\left(r,\frac{1}{f}\right) \\&=& \frac{1}{2\pi}\int\limits_{E_{1}}\log ^{+}\frac{1}{\left|f(re^{i\theta })\right|^{\underline{d}(Q)}}d\theta+\frac{1}{2\pi}\int\limits_{E_{2}}\log ^{+}\frac{1}{\left|f(re^{i\theta })\right|^{\underline{d}(Q)}}d\theta\\ &\leq&\frac{1}{2\pi}\sum\limits_{j=1}^{l}\left[\int\limits_{E_{1}}\log ^{+}\left|c_{j}(z)\right|d\theta+\sum\limits_{i=1}^{k}\int\limits_{E_{1}}\log ^{+}\left|\frac{f^{(i)}}{f}\right|^{n_{ij}}d\theta\right]+\frac{1}{2\pi}\int\limits_{E_{1}}\log ^{+}\left|\frac{1}{Q[f(re^{i\theta})]}\right|d\theta\\&\leq &\frac{1}{2\pi}\int\limits_{0}^{2\pi }\log ^{+}\left|\frac{1}{Q[f(re^{i\theta})]}\right|d\theta+S(r,f)=m\left(r,\frac{1}{Q[f]}\right)+S(r,f).\eeas
So using {\it Lemmas {\ref{CHBl2.4}}}, {\it \ref{CHBl2.5}} and the First Fundamental Theorem, we get
\beas  N(r,0;P[f]) &\leq& N\left(r,\infty;\frac{f^{\overline{d}(P)}}{P[f]}\right)+\overline{d}(P)N(r,0;f)\\
&\leq& m\left(r,\frac{P[f]}{f^{\overline{d}(P)}}\right)+N\left(r,\infty;\frac{P[f]}{f^{\overline{d}(P)}}\right)+\overline{d}(P)N(r,0;f)+S(r,f)\\
&\leq& \left(\ol {d}(P)-\underline {d}(P)\right)m\left(r,\frac{1}{f}\right)+\left(\Gamma _{P}-\ol {d}(P)\right)\;\ol N(r,\infty;f)\\
&&+\left(\ol {d}(P)-\underline {d} (P)\right)\; N(r,0;f\mid\geq k+1)+\mu \ol N(r,0;f\mid\geq k+1)\\
&&+\ol {d}(P) N(r,0;f\mid\leq k)+\ol{d}(P)N(r,0;f)+S(r,f)\\
&\leq& \frac{\left(\ol {d}(P)-\underline {d}(P)\right)}{\underline {d}(Q)}m\left(r,\frac{1}{Q[f]}\right)+\left(\Gamma _{P}-\ol {d}(P)\right)\;\ol N(r,\infty;f)\\
&&+\left(\ol {d}(P)-\underline {d} (P)\right)\; N(r,0;f\mid\geq k+1)+\mu \ol N(r,0;f\mid\geq k+1)\\&&+\ol {d}(P) N(r,0;f\mid\leq k)+\ol{d}(P)N(r,0;f)+S(r,f).\eeas
Hence the proof follows.
\end{proof}
The proof of the following lemma can be done in the line of Lemma \ref{newlem2}. So we omit the details.
\begin{lem} \label{newlem5} If $ H\equiv0$, then  one of the following conditions hold:
\begin{enumerate}
\item [i)] $T(r,G)=N_{2}(r,0;G)+S(r,G)$, or
\item [ii)] $(F-1-\frac{1}{C})G\equiv-\frac{1}{C}$ for some constant $C\in \mathbb{C}\setminus\{0\}$, or
\item [iii)] $\frac{G-1}{F-1}\equiv C$ for some constant $C\in\mathbb{C}\setminus\{0\}$.
\end{enumerate}
\end{lem}
\section {Proof of the theorem}
\begin{proof} [\textbf{Proof of Theorem \ref{brt1}}] Since $\ol E_{m)}(a;P[f])=\ol E_{m)}(a;Q[f])$, it follows that $\ol E_{m)}(1;F)=\ol E_{m)}(1;G)$ except the zeros and poles of $a(z)$.\par
If $H\not\equiv 0$, then in view of Lemma \ref{newlem4}, we obtain
\beas T\left(r,Q[f]\right) &\leq& 4\ol N(r,\infty;f)+2\ol N\left(r,0;Q[f]\right)+N_{2}\left(r,0;Q[f]\right)+ \ol N\left(r,0;(P[f]/a)^{'}\right)\\&&+\ol N\left(r,0;(P[f]/a)^{'}\mid (P[f]/a)\not=0\right)+S(r,f),
\eeas
which contradicts (\ref{bre1.9}).\\
Thus $H\equiv 0$. Then applying the given condition (\ref{bre1.9}) in Lemma \ref{newlem5}, we obtain two following cases:\\
\textbf{Case-1.}
\bea\label{bre3.12} \left(F-1-\frac{1}{C}\right)G\equiv -\frac{1}{C}.
\eea
That is, \bea\label{sri} P[f]Q[f]-a Q(1+d)\equiv -da^{2},\eea for a non-zero constant $d=\frac{1}{C}\in \mathbb{C}$.\\
\textbf{Case-2.}
\beas\frac{G-1}{F-1}=C.\eeas
That is,
\bea\frac{Q[f]-a}{P[f]-a}=C.
\eea
Next we have to show that if
\begin{enumerate}
\item [i)] $P[f]=b_{1}f^{n}+b_{2}f^{n-1}+b_{3}f^{n-2}+\ldots+b_{t-1}f$, or
\item [ii)] $\underline {d}(Q) >2\ol {d}(P)-\underline {d}(P)$ and each monomial of $Q[f]$ contains a term involving a power of $f$,
\end{enumerate}
then the conclusion in \emph{case-1} does not hold.\\
\medbreak
If $P[f]=b_{1}f^{n}+b_{2}f^{n-1}+b_{3}f^{n-2}+\ldots+b_{t-1}f$, then proceeding as in Lemma \ref{newlem3}, we get a contradiction when $n\geq1$.\par
Next we assume that $P[f]$ is a differential polynomial satisfying $\underline {d}(Q) >2\ol {d}(P)-\underline {d}(P)$ and each monomial of $Q[f]$ contains a term involving a power of $f$.\par
Under this situation, we discuss the following two cases:\\
First we assume that $C=-1$. Then from (\ref{bre3.12}), we get $FG\equiv 1$; i.e., $P[f] Q[f]\equiv a^{2}$.\par
Then clearly $\ol N(r,\infty;P[f])=\ol N(r,\infty;Q[f])=S(r,f)$. Also $N(r,0;f)=S(r,f)$, since each monomial of $Q[f]$ contains a term involving a power of $f$.\par
Thus from the First Fundamental Theorem, {\it Lemma \ref{CHBl2.4}} and noting that $m\left(r,\frac{1}{f}\right)\leq \frac{1}{\underline {d}(Q)} m\left(r,\frac{1}{Q[f]}\right)$, we have
\beas T(r,Q[f])&\leq& T(r,P[f])+S(r,f)\\
&\leq& m\left(r,\frac{P[f]}{f^{\ol {d}(P)}}\right)+\ol {d}(P) m(r,f)+S(r,f)\\
&\leq& \left(\ol {d}(P)-\underline {d}(P)\right)m\left(r,\frac{1}{f}\right)+\ol {d}(P) m(r,f)+S(r,f)\\
&\leq&\frac{\left(\ol {d}(P)-\underline {d}(P)\right)}{\underline {d}(Q)} m\left(r,\frac{1}{Q[f]}\right)+\ol {d}(P)\left\{m\left(r,\frac{1}{f}\right)+N(r,0;f)\right\}+S(r,f)\\
&\leq& \frac{\left(\ol {d}(P)-\underline {d}(P)\right)}{\underline {d}(Q)} m\left(r,\frac{1}{Q[f]}\right)+\frac{\ol {d}(P)}{\underline {d}(Q)}m\left(r,\frac{1}{Q[f]}\right)+S(r,f),
\eeas which is a contradiction as $\underline  {d}(Q) >2\ol {d}(P)-\underline {d}(P)$.\par
Next we assume $C\not =-1$. Then from (\ref{bre3.12}), we have $$\ol N(r,1+\frac{1}{C};F)=\ol N(r,\infty;G)=S(r,f).$$\\
So again noticing the fact that each monomial of $Q[f]$ contains a term involving a power of $f$, by the Second Fundamental Theorem and {\it Lemma \ref{brl2.8}}, we get
\bea\label{bre3.15} T(r,P[f])&\leq& \ol N(r,\infty;F)+\ol N(r,0;F)+\ol N\left(r,1+\frac{1}{C};F\right)+S(r,f)\nonumber\\
&\leq& N(r,0;P[f])+S(r,f)\nonumber\\
&\leq &\frac{\ol {d}(P)-\underline {d}(P)}{\underline {d}(Q)} T(r,P[f])+S(r,f),\nonumber\eea
i.e., \be\label{bre3.16} \frac{\underline{d}(Q)+\underline{d}(P)-\ol {d}(P)}{\underline{d}(Q)}T(r,P[f])\leq S(r,f).\ee  Since by the given condition $\underline{d}(Q)>2\ol {d}(P)-\underline{d}(P)>\ol {d}(P)-\underline{d}(P)$; so the inequality (\ref{bre3.16}) leads to a contradiction. This proves the theorem.
\end{proof}
$~~$
\vspace{13 cm}
\\
------------------------------------------------
\\
\textbf{The matter of this chapter has been published in Commun. Korean Math. Soc., Vol. 31, No. 2, (2016), pp. 311-327.}
\newpage
\chapter{Further investigations on a question of Zhang and L\"{u}}
\fancyhead[l]{Chapter 4}
\fancyhead[r]{Further investigations on a question of Zhang and L\"{u}}
\fancyhead[c]{}
\section{Introduction}
\par
Unlike the previous chapters, throughout this chapter our sole intention is to find the relation between power of a meromorphic function and its differential monomial, taking into background that derivative of a function is a special case of differential monomial.\par
In 1998, Gundersen-Yang (\cite{gunyang}) showed that Br\"{u}ck conjecture holds for \emph{entire functions of finite order}. But researchers observed that to consider Br\"{u}ck conjecture for entire functions of infinite order, one need some restrictions on the growth of the functions.\par
The pioneer in this respect was Yu (\cite{6117}) who established some results related to Br\"{u}ck conjecture and posed some open questions in the same paper. Motivated by Yu's questions (\cite{6117}), Lahiri-Sarkar (\cite{br8}) and Zhang (\cite{br15}) studied the problem of a meromorphic or an entire function sharing one small function with its derivative with the notion of weighted sharing of values.
\begin{defi} (\cite{br15}) For $a \in \mathbb{C}\cup\{\infty\}$ and a positive integer $p$, we put
$$\delta_{p}(a,f)=1-\limsup\limits_{r \to \infty} \frac{{N}_{p}(r,a;f)}{T(r,f)}.$$
Thus $0\leq \delta(a,f) \leq \delta_{p}(a,f) \leq \delta_{p-1}(a,f) \leq\ldots\leq \delta_{2}(a,f) \leq \delta_{1}(a,f)=\Theta(a,f)\leq1$.
\end{defi}
In continuation of this content, in 2008, Zhang-L\"{u} (\cite{br16}) considered the uniqueness of the $n$-th power of a meromorphic function sharing a small function with its $k$-th derivative and proved  the following theorem.
\begin{theo 4.A} (\cite{br16}) Let $ k(\geq 1)$, $n(\geq 1)$ be integers and $f$ be a non-constant meromorphic function. Also let $a(z) (\not\equiv 0,\infty )$ be a small function with respect to $f$. Suppose $f^{n}-a$ and $f^{(k)}-a$ share $(0,l)$. If $l=\infty$ and
\be \label {e1.1}(3+k)\Theta(\infty,f)+2\Theta(0,f)+\delta_{2+k}(0,f) > 6+k-n, \ee
or $l=0$ and
\be \label {e1.2}(6+2k)\Theta(\infty,f)+4\Theta(0,f)+2\delta_{2+k}(0,f) > 12+2k-n, \ee
then $f^{n}$ $\equiv$ $f^{(k)}$.
\end{theo 4.A}
In the same paper Zhang-L\"{u} (\cite{br16}) raised the following question:
\begin{ques}
What will happen if  $f^{n}$ and $[f^{(k)}]^{s}$ share a small function?
\end{ques}
In  2010, Chen-Zhang (\cite{zl2}) gave a answer to the above question. Unfortunately there were some gaps in the proof of the theorems in (\cite{zl2}) which was latter rectified by Banerjee-Majumder (\cite{zl1}).\par
In  2010, Banerjee-Majumder (\cite{zl1}) proved two theorems one of which further improved Theorem 4.A whereas the other answered the question of Zhang-L\"{u} (\cite{br16}) in the following manner:
\begin{theo 4.B} (\cite{zl1}) Let $ k(\geq 1)$, $n(\geq 1)$ be integers and $f$ be a non-constant meromorphic function. Also let $a(z) (\not\equiv 0,\infty )$ be a small function with respect to $f$. Suppose $f^{n}-a$ and $f^{(k)}-a$ share $(0,l)$. If $l\geq 2$ and
\be \label {zle1.3} (3+k)\Theta(\infty,f)+2\Theta(0,f)+\delta_{2+k}(0,f) > 6+k-n, \ee
or $l=1$ and
\be \label {zle1.4} \left(\frac{7}{2}+k\right)\Theta(\infty,f)+\frac{5}{2}\Theta(0,f)+\delta_{2+k}(0,f) > 7+k-n, \ee
or $l=0$ and
\be \label {zle1.5}(6+2k)\Theta(\infty,f)+4\Theta(0,f)+\delta_{2+k}(0,f)+\delta_{1+k}(0,f) > 12+2k-n, \ee
then $f^{n}=f^{(k)}$.
\end{theo 4.B}
\begin{theo 4.C} (\cite{zl1}) Let $k(\geq 1)$, $n(\geq 1)$, $m(\geq 2)$ be integers and $f$ be a non-constant meromorphic function.
Also let $a(z)(\not\equiv 0,\infty)$ be a small function with respect to $f$. Suppose $f^{n}-a$ and $[f^{(k)}]^{m}-a$ share $(0,l)$.
If $l=2$ and \be\label{zle1.3a}(3+2k)\;\Theta (\infty,f)+2\;\Theta (0,f)+2\delta_{1+k}(0,f)> 7+2k-n,\ee or
 $l=1$ and \be\label{zle1.4a}\left(\frac{7}{2}+2k\right)\;\Theta (\infty,f)+\frac{5}{2}\;\Theta (0,f)+2\delta_{1+k}(0,f)> 8+2k-n,\ee
or $l=0$ and \be\label{zle1.5a}(6+3k)\;\Theta (\infty,f)+4\;\Theta (0,f)+3\delta_{1+k}(0,f)> 13+3k-n,\ee then $f^{n}\equiv [f^{(k)}]^{m}$.
\end{theo 4.C}
For $m=1$, it can be easily proved that Theorem 4.B is a better result than Theorem 4.C. Also we observe that in the conditions (\ref{zle1.3a})-(\ref{zle1.5a}), there was no influence of $m$. \par
Very recently, in order to improve the results of Zhang (\cite{br15}), Li-Huang (\cite{zl5a}) obtained the following theorem. In view of {\it Lemma \ref{zll1.1}} proved later, we see that the following result obtained in (\cite{zl5a}) is better than that of {\it Theorem 4.B} for $n=1$.
\begin{theo 4.D} (\cite{zl5a}) Let $f$ be a non-constant meromorphic function, $ k(\geq 1)$, $l(\geq 0)$ be be integers and also let $a(z) (\not\equiv 0,\infty )$ be a small function with respect to $f$. Suppose $f-a$ and $f^{(k)}-a$ share $(0,l)$. If $l\geq 2$ and
\be \label {zle1.6} (3+k)\Theta(\infty,f)+\delta_{2}(0,f)+\delta_{2+k}(0,f) > k+4,\ee
or $l=1$ and
\be \label {zle1.7} \left(\frac{7}{2}+k\right)\Theta(\infty,f)+\frac{1}{2}\Theta(0,f)+\delta_{2}(0,f)+\delta_{2+k}(0,f) > k+5, \ee
or $l=0$ and
\be \label {zle1.8} (6+2k)\Theta(\infty,f)+2\Theta(0,f)+\delta_{2}(0,f)+\delta_{1+k}(0,f)+\delta_{2+k}(0,f) > 2k+10, \ee
then $f \equiv f^{(k)}$.
\end{theo 4.D}
Recently Charak-Lal (\cite{zl2a}) considered the possible extension of Theorem 4.B in the direction of the question of Zhang-L\"{u} (\cite{br16}) up to differential polynomial. They proved the following result:
\begin{theo 4.E} (\cite{zl2a}) Let $f$ be a non-constant meromorphic function and $n$  be a positive integer and $a(z) (\not\equiv 0,\infty )$ be a meromorphic function satisfying $T(r,a)=o(T(r,f))$ as $r \to \infty$. Let $P[f]$ be a non-constant differential polynomial in $f$. Suppose $f^{n}$ and $P[f]$ share $(a,l)$. If $l\geq 2$ and
\be \label {zle1.9}(3+Q)\Theta(\infty,f)+2\Theta(0,f)+\ol{d}(P)\delta(0,f) > Q+5+2\ol{d}(P)-\underline{d}(P)-n, \ee
or $l=1$ and
\be \label {zle1.10} \left(\frac{7}{2}+Q\right)\Theta(\infty,f)+\frac{5}{2}\Theta(0,f)+\ol{d}(P)\delta(0,f) > Q+6+2\ol{d}(P)-\underline{d}(P)-n, \ee
or $l=0$ and
\be \label {zle1.11} (6+2Q)\Theta(\infty,f)+4\Theta(0,f)+2\ol{d}(P)\delta(0,f) > 2Q+4\ol{d}(P)-2\underline{d}(P)+10-n, \ee
then $f^{n} \equiv P[f]$.
\end{theo 4.E}
This is a supplementary result corresponding to Theorem 4.B because putting $P[f]=f^{(k)}$ one can't obtain Theorem 4.B, rather in this case a set of stronger conditions are obtained  as particular case of Theorem 4.D. So it is natural to ask the next question.
\begin{ques} \label{zlq1}
Is it possible to improve {\it Theorem 4.B} in the direction of {\it Theorem 4.D} up to differential monomial so that the result give a positive answer to the question of Zhang-L\"{u} (\cite{br16})?
\end{ques}
Before going to state our main result, we need to introduce another notation.
\begin{defi} For two positive integers $n$ and $p$, we define
$$\mu_{p}= min\{n,p\}~~\text{and}~~\mu_{p}^{*}= p+1-\mu_{p}.$$
With the above notation, it is clear that $ N_{p}(r,0;f^{n}) \leq \mu_{p}N_{\mu_{p}^{*}}(r,0;f).$
\end{defi}
\section{Main Result}
The following theorem is the main result of this chapter which gives a positive answer to the question of Zhang-L\"{u} (\cite{br16}).
\begin{theo}\label{zlt1} Let $ k(\geq 1)$, $n(\geq 1)$ be integers and $f$ be a non-constant meromorphic function and $M[f]$ be a differential monomial of degree $d_{M}$ and weight $\Gamma_{M}$ and $k$ is the highest derivative in $M[f]$. We put $\lambda=\Gamma_{M}-d_{M}$.  Also let $a(z) (\not\equiv 0,\infty )$ be a small function with respect to $f$. Suppose $f^{n}-a$ and $M[f]-a$ share $(0,l)$.\\ If $l\geq 2$ and
\be \label{zle1.12} (3+\lambda)\Theta(\infty,f)+\mu_{2}\delta_{\mu_{2}^{*}}(0,f)+d_{M}\delta_{2+k}(0,f) > 3+\Gamma_{M}+\mu_{2}-n, \ee
or $l=1$ and
\be\label {zle1.13}  \left(\frac{7}{2}+\lambda\right)\Theta(\infty,f)+\frac{1}{2}\Theta(0,f)+\mu_{2}\delta_{\mu_{2}^{*}}(0,f)+d_{M}\delta_{2+k}(0,f) > 4+\Gamma_{M}+\mu_{2}-n, \ee
or $l=0$ and
\be \label {zle1.14} (6+2\lambda)\Theta(\infty,f)+2\Theta(0,f)+\mu_{2}\delta_{\mu_{2}^{*}}(0,f)+d_{M}\delta_{2+k}(0,f)+d_{M}\delta_{1+k}(0,f) > 8+2\Gamma_{M}+\mu_{2}-n, \ee
then $f^{n} \equiv M[f]$.
\end{theo}
\begin{cor} If we consider  a non-constant polynomial $p(f)$ of degree $n$  with $p(0)=0$ instead of $f^{n}$ in above Theorem, then similar type conclusions hold.
\end{cor}
Following example shows that in Theorem \ref{zlt1}, $a(z) \not\equiv 0,\infty $ is necessary.
\begin{exm} \par
Let us take $f(z)=e^{e^{z}}$ and $M[f]=f'$, then $M[f]$ and $f$ share $0,\infty$ and the deficiency conditions stated in Theorem \ref{zlt1} is satisfied as $0$, $\infty$ both are exceptional values of f but $f \not\equiv M[f]$.
\end{exm}
The next example shows that the deficiency conditions stated in Theorem \ref{zlt1} are sufficient but not necessary.
\begin{exm} \par
Let $f(z)=Ae^{z}+Be^{-z}$, $AB \neq 0$. Then $\ol{N}(r,\infty;f)=S(r,f)$ and $\ol{N}(r,0;f)=\ol{N}(r,-\frac{B}{A};e^{2z})\sim T(r,f)$. Thus $\Theta(\infty,f)=1$ and $\Theta(0,f)=\delta_{q}(0,f)=0$. \par
It is clear that $M[f]=f^{''}$ and $f$ share $a(z)=\frac{1}{z}$ and the deficiency conditions in Theorem \ref{zlt1} is not satisfied, but $M[f] \equiv f$.
\end{exm}
In the next example, we see that $f^{n}$ can't be replaced by arbitrary polynomial $p(f)=a_{0}f^{n}+a_{1}f^{n-1}+\ldots+a_{n}$ in Theorem \ref{zlt1} for IM ($l=0$) sharing case.
\begin{exm} \label{zlex1.3} \par
If we take $f(z)=e^{z}$, $p(f)=f^{2}+2f$ and $M[f]=f^{(3)}$, then $p(f)+1=(M[f]+1)^{2}$. Thus $p(f)$ and $M[f]$ share $(-1,0)$. Also $\Theta(0,f)=\Theta(\infty,f)=\delta_{q}(0,f)=\delta(0,f)=1$ as $0$ and $\infty$ are exceptional values of $f$. Thus condition (\ref{zle1.14}) of Theorem \ref{zlt1} is satisfied but $p(f)\not\equiv M[f]$.
\end{exm}
In view of example \ref{zlex1.3} the following question is inevitable.
\begin{ques}
Is it possible to replace $f^{n}$ by arbitrary polynomial $p(f)=a_{0}f^{n}+a_{1}f^{n-1}+\ldots+a_{n}$ with $p(0)\not=0$ in Theorem \ref{zlt1} for $l\geq1$?
\end{ques}
\section{Lemmas}
Throughout this chapter, we take $F=\frac{f^{n}}{a}$, $G=\frac{M[f]}{a}$ and $H$ is defined by the equation (\ref{CHB}). Now present some lemmas which are necessary to proceed further.
\begin{lem}\label{zll7} For any two non-constant meromorphic functions $f_{1}$ and $f_{2}$,
$$N_{p}(r,\infty;f_{1}f_{2}) \leq N_{p}(r,\infty;f_{1})+N_{p}(r,\infty;f_{2}).$$
\end{lem}
\begin{proof} Let $z_{0}$ be a pole of $f_{i}$ of order $t_{i}$ for $i=1,2.$ Then $z_{0}$ be a pole of $f_{1}f_{2}$ of order at most $t_{1}+t_{2}$.\par
\textbf{Case-1.} Let $t_{1} \geq p $ and $t_{2} \geq p $. Then $t_{1}+t_{2}\geq p$. So $z_{0}$ is counted at most $p$ times in the left hand side of the above counting function, whereas the same is counted $p+p$ times in the right hand side of the above counting function.\par
\textbf{Case-2.}  Let $t_{1} \geq p $ and $t_{2} < p $.\\
 \textbf{Subcase-2.1.} Let $t_{1}+t_{2}\geq p$. So $z_{0}$ is counted at most $p$ times in the left hand side of the above counting function, whereas the same is counted as $p+\max\{0,t_{2}\}$ times in the right hand side of the above counting function.\\
 \textbf{Subcase-2.2.} Let $t_{1}+t_{2}< p$. This case is occurred if $t_{2}$ is negative; i.e., if $z_{0}$ is a zero of $f_{2}$. Then $z_{0}$ is counted at most $\max\{0,t_{1}+t_{2}\}$ times whereas the same is counted $p$  times in the right hand side of the above expression. \par
\textbf{Case-3.} Let $t_{1} < p $ and $t_{2} \geq p $. Then $t_{1}+t_{2}\geq p$. This case can be disposed off as done in Case 2.\par
\textbf{Case-4.} Let $t_{1} < p $ and $t_{2} < p$. \\
\textbf{Subcase-4.1.} Let $t_{1}+t_{2}\geq p$.  Then $z_{0}$ is counted at most $p$ times whereas the same is counted $\max\{0,t_{1}\}+\max\{0,t_{2}\}$ times in the right hand side of the above expression. \\
\textbf{Subcase-4.2.} Let $t_{1}+t_{2} < p$. Then $z_{0}$ is counted at most $\max\{0,t_{1}+t_{2}\}$ times whereas $z_{0}$ is counted $\max\{0,t_{1}\}+\max\{0,t_{2}\}$ times in the right hand side of the above counting functions. Combining all the cases, Lemma \ref{zll7} follows.
\end{proof}
\begin{lem}\label{zll1.1} $1+\delta_{2}(0,f) \geq 2\Theta(0,f)$.
\end{lem}
\begin{proof}
\beas 2\Theta(0,f)-\delta_{2}(0,f)-1 &=& \limsup_{r \to \infty}\frac{N_{2}(r,0;f)}{T(r,f)}-\limsup_{r \to \infty}\frac{2\ol{N}(r,0;f)}{T(r,f)}\\
& \leq& \limsup_{r \to \infty}\frac{N_{2}(r,0;f)-2\overline{N}(r,0;f)}{T(r,f)}\\
& \leq& 0. \eeas
\end{proof}
\begin{lem} (\cite{br8})\label{zll8} Let $p$ and $k$ be two positive integers. Then for a non-constant meromorphic function $f$, the following inequality holds:
$$N_{p}(r,0;f^{(k)}) \leq N_{p+k}(r,0;f)+k\overline{N}(r,\infty;f)+S(r,f).$$
\end{lem}
\begin{lem} \label{zll5}  Let $f$ be a non-constant  meromorphic function and $M[f]$  be a differential monomial of degree $d_{M}$ and weight $\Gamma_{M}$. Then $$N\left(r,\infty;\frac{M}{f^{d_{M}}}\right) \leq d_{M}N(r,0;f)+\lambda\overline{N}(r,\infty;f)+S(r,f).$$
\end{lem}
\begin{proof} If $z_{0}$ be a pole of $f$  of order $t$, then it is a pole of $\frac{M}{f^{d_{M}}}$ of order $\lambda=n_{1}+2n_{2}+\ldots+kn_{k}$. Again if $z_{1}$ be a zero of $f$  of order $s$, then it is a pole of $\frac{M}{f^{d_{M}}}$ of order at most $sd_{M}$. Thus $$N\left(r,\infty;\frac{M}{f^{d_{M}}}\right) \leq d_{M}N(r,0;f)+\lambda\overline{N}(r,\infty;f)+S(r,f).$$
Hence the proof.
\end{proof}
\begin{lem}\label{zll9} Let $p$ and $k$ be two positive integers. Also let $M[f]$ be a differential monomial generated by a non-constant meromorphic function $f$. Then $$N_{p}(r,0;M[f]) \leq d_{M}N_{p+k}(r,0;f)+\lambda\overline{N}(r,\infty;f)+S(r,f).$$
\end{lem}
\begin{proof}
In view of Lemma \ref{zll7} and  Lemma \ref{zll8}, we can write
\beas N_{p}(r,0;M[f]) &\leq& \sum\limits_{i=0}^{k} n_{i}N_{p}(r,0;f^{(i)})+S(r,f)\\
&\leq& \sum\limits_{i=0}^{k} n_{i}\{N_{p+i}(r,0;f)+i\overline{N}(r,\infty;f)\}+S(r,f)\\
&\leq& d_{M}N_{p+k}(r,0;f)+\lambda\overline{N}(r,\infty;f)+S(r,f).
\eeas
Hence the proof.
\end{proof}
\begin{lem}\label{zll10}  Let $f$ be a non-constant  meromorphic function and $a(z)$ be a small function with respect to $f$. Then $FG \not\equiv 1$, where $F$ and $G$ are defined previously.
\end{lem}
\begin{proof} On contrary assume that $FG \equiv 1$. Then Lemma \ref{zll5} and the First Fundamental Theorem yields that
\begin{eqnarray*} (n+d_{M})T(r,f) &=&T\left(r,\frac{M}{f^{d_{M}}}\right)+S(r,f)\\
 &\leq& d_{M}N(r,0;f)+\lambda\overline{N}(r,\infty;f)+S(r,f)\\ &=&S(r,f),
\end{eqnarray*} which is a contradiction. Thus the lemma follows.
\end{proof}
\begin{lem}\label{zll11} (\cite{zl1}) Let $F$ and $G$ share $(1,l)$ and $\overline{N}(r,\infty;F)=\overline{N}(r,\infty;G)$ and $H\not\equiv 0$, where  $F$, $G$ and $H$ are defined previously.
 Then \beas N(r,\infty;H) &\leq& \overline{N}(r,\infty;F)+\overline{N}(r,0;F|\geq 2)+\overline{N}(r,0;G|\geq 2)+\overline{N}_{0}(r,0;F')\\
 &&+\overline{N}_{0}(r,0;G')+\overline{N}_{L}(r,1;F)+\overline{N}_{L}(r,1;G)+S(r,f). \eeas
 \end{lem}
 \begin{lem}\label{zll13} Let $F$ and $G$ share $(1,l)$ and $H \not\equiv 0$. Then \beas \overline{N}(r,1;F)+ \overline{N}(r,1;G) &\leq& N(r,\infty;H) + \overline{N}^{(2}_{E}(r,1;F)+\overline{N}_{L}(r,1;F)+\overline{N}_{L}(r,1;G)\\ && +\overline{N}(r,1;G)+S(r,f).\eeas
\end{lem}
\begin{proof} By simple calculations, it is clear that
$$N(r,1;F|=1)\leq N(r,0;H)+S(r,f)\leq N(r,\infty;H)+S(r,f).$$
Thus proof is obvious if we keep the above inequality in our mind.
\end{proof}
\newpage
\begin{lem}\label{zll12} If $F$ and $G$ share $(1,l)$, then
$$\overline{N}_{L}(r,1;F)\leq \frac{1}{2}\overline{N}(r,\infty;F)+\frac{1}{2}\overline{N}(r,0;F)+S(r,F)~~~~\text{when}~~l\geq 1~~\text{and}$$
$$\overline{N}_{L}(r,1;F)\leq \overline{N}(r,\infty;F)+\overline{N}(r,0;F)+S(r,F)~~~~\text{when}~~l=0.$$
\end{lem}
\begin{proof} If $l\geq1$, then multiplicity of any 1-point of $F$ counted in $ \overline{N}_{L}(r,1;F)$ is at least $3$. Therefore $\overline{N}_{L}(r,1;F) \leq \frac{1}{2}\overline{N}(r,0;F'|F\neq 0) \leq \frac{1}{2}\overline{N}(r,\infty;F)+\frac{1}{2}\overline{N}(r,0;F)+S(r,F)$.
\vspace{.2 cm}
\\
Again if $l=0$, then multiplicity of any 1-point of $F$ counted in $ \overline{N}_{L}(r,1;F)$ is at least $2$. Thus $\overline{N}_{L}(r,1;F) \leq \overline{N}(r,0;F'|F\neq 0) \leq \overline{N}(r,\infty;F)+\overline{N}(r,0;F)+S(r,F)$.
\end{proof}
\begin{lem}\label{zll14}
Let $f$ be a non-constant meromorphic function and $a(z)$ be a small function of $f$. Also let $F=\frac{f^{n}}{a}$ and $G=\frac{M}{a}$. If  $F$ and $G$ share $(1,\infty)$ except the zeros and poles of $a(z)$ and
$$N_{2}(r,0;F)+N_{2}(r,0;G)+\ol{N}(r,\infty;F)+\ol{N}(r,\infty;G)+\ol{N}_{*}(r,\infty;F,G)<(\nu+o(1))\widetilde{T}(r),$$
where $\nu<1,~\widetilde{T}(r)=\max\{T(r,F),T(r,G)\}$ and $\widetilde{S}(r)=o(\widetilde{T}(r))$, $r\in I$, $I$ is a set of infinite linear measure of $r\in(0,\infty)$, then  $F\equiv G$ or $FG\equiv 1$.
\end{lem}
\begin{proof}
 Let $z_{0}$ be a pole of $f$ which is not a pole or zero of $a(z)$. Then $z_{0}$ is a pole of $F$ as well as $G$. Thus $F$ and $G$ share those pole of $f$ which is not zero or pole of $a(z)$. Thus
\beas N(r,H) &\leq& \ol{N}(r,0;F\geq2)+\ol{N}(r,0;G\geq2)+\ol{N}_{L}(r,\infty;F)+\ol{N}_{L}(r,\infty;G)\\
&+& \ol{N}_{0}(r,0;F')+\ol{N}_{0}(r,0;G')+S(r,f).\eeas
Rest part of the proof can be carried out in the line of proof of Lemma 2.13 of (\cite{zl1.1}).
\end{proof}
\section{Proof of the theorem}
\begin{proof}[\textbf{Proof of Theorem \ref{zlt1}}] Since $f^{n}$ and $M[f]$ share $(a,l)$, it follows that $F$ and $G$ share $(1,l)$ except the zeros and poles of $a(z)$. Now we consider the following cases:\\
\textbf{Case-1.} Let $H\not\equiv 0$.\\
\textbf{Subcase-1.1.} If $l\geq 1$, then using Second Main Theorem, Lemmas \ref{zll13} and \ref{zll11}, we get
\bea\label{zlt2} &&T(r,F)+T(r,G)\\
&\leq&\nonumber \overline{N}(r,\infty;F)+\overline{N}(r,\infty;G)+\overline{N}(r,0;F)+\overline{N}(r,0;G)+N(r,\infty;H)+\overline{N}^{(2}_{E}(r,1;F)\\
&& \nonumber +\overline{N}_{L}(r,1;F)+\overline{N}_{L}(r,1;G)+\overline{N}(r,1;G)-\overline{N}_{0}(r,0;F^{'})-\overline{N}_{0}(r,0;G^{'})+S(r,f)\\
&\leq &\nonumber 2\overline{N}(r,\infty;F)+\overline{N}(r,\infty;G)+N_{2}(r,0;F)+N_{2}(r,0;G)+\overline{N}^{(2}_{E}(r,1;F)\\
\nonumber&&  +2\overline{N}_{L}(r,1;F)+2\overline{N}_{L}(r,1;G)+\overline{N}(r,1;G) +S(r,f).\eea
\textbf{Subsubcase-1.1.1.} If $l=1$, then using Lemmas \ref{zll9} and \ref{zll12}, inequality (\ref{zlt2}) can be written as
\beas && T(r,F)+T(r,G)\\
& \leq& \frac{5}{2}\overline{N}(r,\infty;F)+\overline{N}(r,\infty;G)+\frac{1}{2}\overline{N}(r,0;F)+\mu_{2}N_{\mu_{2}^{*}}(r,0;f)+N_{2}(r,0;G)\\
&&+\overline{N}^{(2}_{E}(r,1;F)+\overline{N}_{L}(r,1;F)+2\overline{N}_{L}(r,1;G)+\overline{N}(r,1;G)+S(r,f)\\
& \leq& \frac{5}{2}\overline{N}(r,\infty;F)+\overline{N}(r,\infty;G)+\frac{1}{2}\overline{N}(r,0;F)+\mu_{2}N_{\mu_{2}^{*}}(r,0;f)+N_{2}(r,0;G)\\
&&+N(r,1;G)+S(r,f).\eeas
That is, for any $\varepsilon > 0$
\beas && n\;T(r,f)\\
&\leq& \big(\lambda+\frac{7}{2}\big)\overline{N}(r,\infty;f)+\frac{1}{2}\overline{N}(r,0;f)+\mu_{2}N_{\mu_{2}^{*}}(r,0;f)+d_{M}N_{2+k}(r,0;f)+S(r,f)\\
&\leq& \big\{\Gamma_{M}+\mu_{2}+4-(\lambda+\frac{7}{2})\Theta(\infty,f)-\frac{1}{2}\Theta(0,f)-\mu_{2}\delta_{\mu_{2}^{*}}(0,f)
-d_{M}\delta_{2+k}(0,f)\big\}T(r,f)\\
&&+\left(\varepsilon+o(1)\right)T(r,f),\eeas
which contradicts the condition (\ref{zle1.13}).
\\
\textbf{Subsubcase-1.1.2.} If $l\geq 2$, then with the help of Lemma \ref{zll9}, we can write (\ref{zlt2}) as
\beas && T(r,F)+T(r,G)\\
&\leq& 2\overline{N}(r,\infty;F)+\overline{N}(r,\infty;G)+\mu_{2}N_{\mu_{2}^{*}}(r,0;f)+N_{2}(r,0;G)+ N(r,1;G)+S(r,f).\eeas
i.e., for any $\varepsilon > 0$
\beas && n\;T(r,f)\\
&\leq & (\lambda+3)\overline{N}(r,\infty;f)+\mu_{2}N_{\mu_{2}^{*}}(r,0;f)+d_{M}N_{2+k}(r,0;f)+S(r,f)\\
& \leq & \big\{\Gamma_{M}+\mu_{2}+3-(\lambda+3)\Theta(\infty,f)-\mu_{2}\delta_{\mu_{2}^{*}}(0,f)-d_{M}\delta_{2+k}(0,f)\big\}\;T(r,f)\\
&& +\left(\varepsilon+o(1)\right)\;T(r,f),\eeas
which contradicts the condition (\ref{zle1.12}).\\
\textbf{Subcase-1.2.} If $l=0$, then using the Second Main Theorem and Lemmas \ref{zll13}, \ref{zll11}, \ref{zll12} and \ref{zll9}, we get
\bea \label{zlr11} && T(r,F)+T(r,G)\\
\nonumber &\leq& \overline{N}(r,\infty;F)+\overline{N}(r,0;F)+\overline{N}(r,1;F)+\overline{N}(r,\infty;G) +\overline{N}(r,0;G)\\
\nonumber && +\overline{N}(r,1;G)-\overline{N}_{0}(r,0;F')-\overline{N}_{0}(r,0;G')+S(r,F)+S(r,G)\\
\nonumber & \leq & \overline{N}(r,\infty;F)+\overline{N}(r,0;F)+\overline{N}(r,\infty;G)+\overline{N}(r,0;G)+N(r,\infty;H)+\overline{N}^{(2}_{E}(r,1;F)\\
\nonumber && +\overline{N}_{*}(r,1;F,G)+\overline{N}(r,1;G)-\overline{N}_{0}(r,0;F')-\overline{N}_{0}(r,0;G')+S(r,F)+S(r,G)\\
\nonumber & \leq& 2\overline{N}(r,\infty;F)+\overline{N}(r,\infty;G)+N_{2}(r,0;F)+N_{2}(r,0;G)+\overline{N}^{(2}_{E}(r,1;F)\\
\nonumber && +2\overline{N}_{L}(r,1;F)+2\overline{N}_{L}(r,1;G)+\overline{N}(r,1;G)+S(r,f) \\
\nonumber & \leq &
2\overline{N}(r,\infty;F)+2\overline{N}(r,\infty;G)+\mu_{2}N_{\mu_{2}^{*}}(r,0,f)+N_{2}(r,0;G)+\overline{N}(r,0;G)\\
\nonumber && +2\left(\overline{N}(r,\infty;F)+\overline{N}(r,0;F)\right)+\overline{N}^{(2}_{E}(r,1;F)+\overline{N}_{L}(r,1;G)+\overline{N}(r,1;G)+S(r,f)\\
\nonumber &\leq&  4\overline{N}(r,\infty;F)+\mu_{2}N_{\mu_{2}^{*}}(r,0,f)+N_{2}(r,0;G)+2\overline{N}(r,\infty;G)\\
\nonumber && +\overline{N}(r,0;G)+2\overline{N}(r,0;F)+N(r,1;G)+S(r,f).\eea
That is, for any $\varepsilon > 0$
\beas  n\;T(r,f) &\leq& (2\lambda+6)\overline{N}(r,\infty;f)+2\overline{N}(r,0;f)+\mu_{2}N_{\mu_{2}^{*}}(r,0,f)+d_{M}N_{1+k}(r,0;f)\\
\nonumber && +d_{M}N_{2+k}(r,0;f)+S(r,f)\\
& \leq& \{2\Gamma_{M}+\mu_{2}+8-(2\lambda+6)\Theta(\infty,f)-2\Theta(0,f)-\mu_{2}\delta_{\mu_{2}^{*}}(0,f)\\
&&-d_{M}\delta_{1+k}(0,f)-d_{M}\delta_{2+k}(0,f)\}\;T(r,f)+\left(\varepsilon + o(1) \right)\;T(r,f),\eeas
which contradicts the condition (\ref{zle1.14}).
\newpage
\textbf{Case-2.} Let $H\equiv 0$.\\
On Integration, we get,
\begin{equation}\label{zl3.3} \frac{1}{G-1}\equiv\frac{A}{F-1}+B,\end{equation}
 where $A(\neq 0)$, $B$ are complex constants.
Then $F$ and $G$ share $(1,\infty)$. Also by construction of $F$ and $G$, we see that $F$ and $G$ share $(\infty,0)$ also. \par
So using Lemma \ref{zll9} and condition (\ref{zle1.12}), we obtain
\beas && N_{2}(r,0;F)+N_{2}(r,0;G)+\ol{N}(r,\infty;F)+\ol{N}(r,\infty;G)+\ol{N}_{L}(r,\infty;F)+\ol{N}_{L}(r,\infty;G)\\
 &\leq& \mu_{2}N_{\mu_{2}^{*}}(r,0;f)+d_{M}N_{2+k}(r,0;f)+(\lambda+3)\overline{N}(r,\infty;f)+\widetilde{S}(r)\\
 &\leq&\{(3+\lambda+d_{M}+\mu_{2})-((\lambda+3)\Theta(\infty,f)+\delta_{\mu_{2}^{*}}(0,f)+d_{M}\delta_{2+k}(0,f))\}\;T(r,f)+\widetilde{S}(r)\\
 &<& T(r,F)+\widetilde{S}(r).\eeas
Thus in view of Lemma \ref{zll14} and Lemma \ref{zll10}, we get  $F\equiv G$; i.e., $f^{n} \equiv M[f]$. This proves the theorem.
\end{proof}
$~~$
\vspace{13 cm}
\\
------------------------------------------------
\\
\textbf{The matter of this chapter has been published in Ann. Univ. Paedagog. Crac. Stud. Math., Vol. 14 (2015), pp. 105-119.}
\newpage
\chapter{A new type of unique range set with deficient values}
\fancyhead[l]{Chapter 5}
\fancyhead[r]{A new type of unique range set with deficient values}
\fancyhead[c]{}
\section{Introduction}
In the introductory part, we have already noticed that Gross (\cite{am4}) was the pioneer of the concept of unique range sets. For the sake of brevity, here we reconsider the definition of unique range sets as a generalization of the definition of value sharing as follows:
\begin{defi} (\cite{padic})\label{rabi}
Let $f$ and $g$ be two non-constant meromorphic functions and $S\subset \mathbb{C}\cup\{\infty\}$, we define
$$E_{f}(S)=\bigcup\limits_{a \in S}\{(z,p) \in \mathbb{C}\times\mathbb{N}~ |~ f(z)=a ~with~ multiplicity~ p\},$$
$$\ol{E}_{f}(S)=\bigcup\limits_{a \in S}\{(z,1) \in \mathbb{C}\times\mathbb{N}~ |~ f(z)=a ~with~ multiplicity~ p\}.$$
If $E_{f}(S)=E_{g}(S)$ \big(resp. $\ol{E}_{f}(S)=\ol{E}_{g}(S)$\big), then it is said that $f$ and $g$ share the set $S$ counting multiplicities or in short CM (resp. ignoring multiplicities or in short IM).\par
Thus, if $S$ is singleton, then it coincides with the usual definition of value sharing.
\end{defi}
In 1977, Gross (\cite{am4}) proposed the following problem which has later became popular as \enquote{\emph{Gross Problem}}. The problem was as follows:\par
\emph{Does there exist a finite set $S$ such that any two non-entire functions $f$ and $g$ sharing the set $S$ must be $f=g$?}\par
In 1982, Gross-Yang (\cite{am5}) gave the affirmative answer to the bove question as follows:
\begin{theo 5.A} (\cite{am5}) Let $S=\{z\in \mathbb{C} : e^{z}+z=0\}$. If two entire functions $f$, $g$ satisfy  $E_{f}(S)=E_{g}(S)$, then $f\equiv g$.
\end{theo 5.A}
In that paper (\cite{am5}) they first introduced the terminology unique range set for entire function (in short URSE). Later the analogous definition for meromorphic functions was introduced in a similar fashion.
\begin{defi} (\cite{1234})
Let  $S\subset \mathbb{C}\cup\{\infty\} $ and $f$ and $g$ be two non-constant meromorphic (resp. entire) functions. If $E_{f}(S)=E_{g}(S)$ implies $f\equiv g$, then $S$ is called a unique range set for meromorphic (resp. entire) functions or in brief URSM (resp. URSE).
\end{defi}
In 1997, Yi (\cite{am15a}) introduced the analogous definition for unique range sets with ignoring multiplicities.
\begin{defi} (\cite{am15a})
A set $S\subset \mathbb{C}\cup\{\infty\}$ is said to be a unique range set for meromorphic (resp. entire) functions ignoring multiplicity or in short URSM-IM (resp. URSE-IM) or a reduced unique range set for meromorphic (resp. entire) functions or in short R-URSM (resp. R-URSE) if $\ol E_{f}(S)=\ol E_{g}(S)$ implies $f\equiv g$ for any pair of non-constant meromorphic (resp. entire) functions.
\end{defi}
During the last few years the notion of unique as well as reduced unique range sets have been generating an increasing interest among the researchers and naturally a new area of research have been developed under the aegis of uniqueness theory.\par
The prime concern of the researchers have been to find new unique range sets or to make the cardinalities of the existing range sets as small as possible imposing some restrictions on the deficiencies of the generating meromorphic functions. To see the remarkable progress in this regard one can make a glance to (\cite{am11}, \cite{am1e}, \cite{am3}, \cite{am9b}, \cite{bd5}, \cite{am12}, \cite{am14}, \cite{am15}, \cite{am15a}).\par
Till date the URSM with $11$ elements and R-URSM with $17$ elements are the smallest available URSM and R-URSM obtained by Frank-Reinders (\cite{am3}) and Bartels (\cite{am1e}) respectively. Similarly URSE with $7$ elements (\cite{am12}) and R-URSE with $10$ elements (see \cite{1234}, Theorem 10.76) are the smallest available URSE and R-URSE.\par
Also it is observed that a URSE must contain at least $5$ elements whereas a URSM must contain at least $6$ elements (see \cite{1234}, p. 517 and p. 527).\par
In 1995, Li-Yang (\cite{am12}) first elucidated the fact that the finite URSM's are nothing but the set of distinct zeros of some suitable polynomials and subsequently the study of the characteristics of these underlying polynomials should also be given utmost priority. Consequently they (\cite{am12}) introduced the following definition:
\begin{defi} (\cite{am12})
 A polynomial $P$ in $\mathbb{C}$ is called an uniqueness polynomial for meromorphic (resp. entire) functions or in short, UPM (resp. UPE), if for any two non-constant meromorphic (resp. entire) functions $f$ and $g$, $P(f)\equiv P(g)$ implies $f\equiv g$.
\end{defi}
Suppose that $P$ is a polynomial of degree $n$ in $\mathbb{C}$ having only simple zeros and $S$ be the set of all zeros of $P$. If $S$ is a URSM (resp. URSE), then from the definition it follows that $P$ is UPM (resp. UPE). However the converse is not, in general, true.
\begin{exm} Suppose that $f(z)=-\frac{b}{a}e^{z}$ and $g(z)=-\frac{b}{a}e^{-z}$. Further suppopse that $S=\{z~\mid~P(z)=0\}$, where $P(z)=az+b$ ($a\not =0$). Then clearly the polynomial $P$ is a UPM but $E_{f}(S)=E_{g}(S)$.
\end{exm}
In 2000, Fujimoto (\cite{am2}) first discovered a special property of a polynomial, which was called by Fujimoto himself as \enquote{poperty H}.\par
A monic polynomial $P(z)$ without multiple zero is said to be satisfy \enquote{poperty H} if $P(\alpha )\not =P(\beta )$ for any two distinct zeros $\alpha $, $\beta $ of the derivative $P'(z)$.\par
In 2003, Fujimoto (\cite{am2a}) obtained the following characterization of a critically injective polynomial to be a uniqueness polynomial.
\begin{theo 5.B}\label{1q1q} (\cite{am2a}) Suppose $P(z)$ is a polynomial satisfying \enquote{poperty H}. Further suppose
that $d_{1}, d_{2}, \ldots, d_{k}$ are the distinct zeros of $P'$ with respective multiplicities $q_{1}, q_{2}, \ldots, q_{k}$. Then $P(z)$ will be a uniqueness polynomial if and only if the following inequality holds:
\bea\label{19012017} \sum \limits_{1\leq l<m\leq k}q_{_{l}}q_{m}>\sum \limits_{l=1}^{k} q_{_{l}}.\eea
In particular, the inequality (\ref{19012017}) is always satisfied whenever $k\geq 4$. Also when $k=3$ and $\max \{q_{1},q_{2},q_{3}\}\geq 2$ or when $k=2$, $\min \{q_{1},q_{2}\}\geq 2$ and $q_{1}+q_{2}\geq 5$, then also the inequality (\ref{19012017}) holds.
\end{theo 5.B}
In continuation to the definition \ref{rabi}, we would like to present the definition of weighted set sharing in a different angle as that was already provided in introduction part.
\begin{defi}
Let $f$ and $g$ be two non-constant meromorphic functions and set $S\subset \mathbb{C}\cup\{\infty\}$. For $l\in\mathbb{N}\cup\{0\}\cup\{\infty\}$, we define
$$E_{f}(S,l)=\bigcup\limits_{a \in S}\{(z,t) \in \mathbb{C}\times\mathbb{N}~ |~ f(z)=a ~with~ multiplicity~ p\},$$
where $t=p$ if $p\leq l$ and $t=p+1$ if $p>l$.\par
If $E_{f}(S,l)=E_{g}(S,l)$, then it is said that $f$ and $g$ share the set $S$ with weight $l$.
\end{defi}
The main intention of this chapter is to introduce a new type of unique range set for meromorphic function which improve all the previous results in this aspect specially those of (\cite{am1c}) and (\cite{am1d}) by removing the \enquote{max} conditions in deficiencies.
\section{Main Result}
Henceforth for two positive integers $n$ and $m$, we shall denote by $P_{\ast}(z)$ the following polynomial:
\beas P_{\ast}(z)=\displaystyle\sum_{i=0}^{m}\hspace{.05in}\binom{m}{i}\frac{(-1)^{i}}{n+m+1-i} z^{n+m+1-i}+1=Q(z)+1,
\eeas
where $Q(z)=\displaystyle\sum_{i=0}^{m}\hspace{.05in}\binom{m}{i}\frac{(-1)^{i}}{n+m+1-i} z^{n+m+1-i}$ and $P_{\ast}^{'}(z)=z^{n}(z-1)^{m}$.
\begin{theo}\label{amt1.1} Let $n(\geq 3)$, $m(\geq 3)$ be two positive integers. Suppose that $S_{\ast}=\{z: P_{\ast}(z)=0\}$. Let $f$ and $g$ be two non-constant meromorphic functions such that $E_{f}(S_{\ast},l)=E_{g}(S_{\ast},l)$.
Now if one of the following conditions holds:
\begin{enumerate}
\item[(a)] $l\geq 2$  and $\Theta _{f}+\Theta _{g}>\left(9-(n+m)\right),$
\item[(b)] $l=1$ and  $\Theta _{f}+\Theta _{g}>\left(10-(n+m)\right),$
\item[(c)] $l=0$ and  $\Theta _{f}^{*}+\Theta _{g}^{*}>\left(15-(n+m)\right),$
\end{enumerate}
then $f\equiv g$, where $$\Theta _{f}=2\Theta (0,f)+2\Theta (\infty,f)+\Theta (1,f)+\frac{1}{2}\min\{\Theta (1;f),\Theta (1;g)\}$$ and $$\Theta _{f}^{*}=3\Theta (0,f)+3\Theta (\infty,f)+\Theta (1,f)+\frac{1}{2}\min\{\Theta (1;f),\Theta (1;g)\}.$$ Similarly $\Theta _{g}$ and $\Theta _{g}^{*}$ are defined.
\end{theo}
Before going to state some necessary lemmas, we explain some notations which will be needful in this sequel.\par
For $a\in\mathbb{C}\cup\{\infty\}$, we denote by $N(r,a;f\mid=1)$ the counting function of simple $a$-points of $f$. Thus if $f$ and $g$ share $(a,m)$, $m\geq 1$, then $N^{1)}_{E}(r,a;f)=N(r,a;f\mid=1)$.\par
For a positive integer $m$, we denote by $N(r,a;f\mid\leq m)$ (resp. $N(r,a;f\mid\geq m$) by the counting function of those $a$-points of $f$ whose multiplicities are not greater (resp. less) than $m$ where each $a$-point is counted according to its multiplicity. Also $\ol N(r,a;f\mid\leq m)$ and $\ol N(r,a;f\mid\geq m)$ are the reduced counting function of $N(r,a;f\mid\leq m)$ and $N(r,a;f\mid\geq m)$ respectively.\par
Analogously, one can define $N(r,a;f\mid <m), N(r,a;f\mid >m), \ol N(r,a;f\mid <m)$ and $\ol N(r,a;f\mid >m)$.\par
We denote by $\ol N(r,a;f\mid=k)$ the reduced counting function of those $a$-points of $f$ whose multiplicities is exactly $k$, where $k\geq 2$ is an integer.
\section{Lemmas}
Let, unless otherwise stated $F$ and $G$ be two non-constant meromorphic functions given by $F=P_{\ast}(f)$ and $G=P_{\ast}(g)$ and $H$ is defined by the equation (\ref{CHB}). Also we denote by $T_{\ast}(r)=\max\{T(r,f),~T(r,g)\}$ and $S_{\ast}(r)=o(T_{\ast}(r))$.
\begin{lem}\label{aml2} If $F$ and $G$ are two non-constant meromorphic functions such that they share $(0,0)$ and $H\not\equiv 0$, then $$N^{1)}_{E}(r,0;F\mid=1)=N^{1)}_{E}(r,0;G\mid=1)\leq N(r,\infty;H)+S(r,f)+S(r,g).$$
\end{lem}
\begin{proof} In view of Corollary \ref{ghgh}, we can write $m(r,H)=S(r,f)+S(r,g).$\par
Now by Laurent expansion of $H$, we can easily verify that each simple zero of $F$ (and so of $G$) is a zero of $H$. Hence
\beas N^{1)}_{E}(r,0;F\mid=1)=N^{1)}_{E}(r,0;G\mid=1)&\leq& N(r,0;H)\nonumber\\&\leq& T(r,H)+O(1)\nonumber\\&=& N(r,\infty;H)+S(r,f)+S(r,g).\nonumber\eeas
This proves the lemma.
\end{proof}
\begin{lem}\label{aml3} Let $S_{\ast}$ be the set of zeros of $P_{\ast}$. If for two non-constant meromorphic functions $f$ and $g$, $E_{f}(S_{\ast},0)=E_{g}(S_{\ast},0)$ and $H\not\equiv 0$, then \beas \ol N(r,\infty;H)&\leq& \ol N(r,0;f)+\ol N(r,1;f)+\ol N(r,0;g)+\ol N(r,1;g)+\ol N_{*}(r,0;F,G)\\& &+\ol N(r,\infty;f)+\ol N(r,\infty;g)+\ol N_{0}(r,0;f^{'})+\ol N_{0}(r,0;g^{'}),\eeas
where by $\ol N_{0}(r,0;f^{'})$, we mean the reduced counting function of those zeros of $f^{'}$ which are not the zeros of $Ff(f-1)$ and $\ol N_{0}(r,0;g^{'})$ is similarly defined.
\end{lem}
\begin{proof} Since $E_{f}(S_{\ast},0)=E_{g}(S_{\ast},0)$, it follows that $F$ and $G$ share $(0,0)$. Also we observe that $F^{'}=f^{n}(f-1)^{m}f^{'}$.
It can be easily verified that possible poles of $H$ occur at (i) poles of $f$ and $g$, (ii) those $0$-points of $F$ and $G$ with different multiplicities, (iii) zeros of $f^{'}$ which are not the zeros of $Ff(f-1)$, (iv) zeros of $g^{'}$ which are not zeros of $Gg(g-1)$, (v) $0$ and $1$ points of $f$ and $g$.
Since $H$ has only simple poles, the lemma follows from above. This proves the lemma.
\end{proof}
\begin{lem}\label{aml4} $Q(1)$ is not an integer. In particular, $P_{\ast}(1)\not =-1$, where $n(\geq 3)$ and $m(\geq 3)$ are integers.
\end{lem}
\begin{proof} We claim that \beas S_{n}(m)&=& \sum\limits_{i=0}^{m}\binom{m}{i} \frac{(-1)^i}{n+m+1-i}\\&=&\frac{(-1)^m m!}{(n+m+1)(n+m)\ldots(n+1)}.\eeas
We prove the claim by method of induction on $m$. \par
At first for $m=3$, we get \beas S_{n}(3)&=&\frac{1}{n+4}-\frac{3}{n+3}+\frac{3}{n+2}-\frac{1}{n+1}\\&=&\frac{(-1)^3\cdot 3!}{(n+4)(n+3)(n+2)(n+1)}.\eeas
So, $S_{n}(m)$ is true for $m=3$. Now we assume that $S_{n}(m)$ is true for $m=k$, where $k$ is any given positive integer such that $k\geq3$. Now we will show that $S_{n}(m)$ is true for $m=k+1$.
i.e., \beas {\frac{\binom{k+1}{0}}{n+k+2}} - {\frac{\binom{k+1}{1}}{n+k+1}} +\ldots + (-1)^{k+1}\frac{\binom{k+1}{k+1}}{n+1}=\frac{(-1)^{(k+1)} {(k+1)!}}{(n+k+2)(n+k+1)\ldots(n+1)}.\eeas
Using induction hypothesis, we have \beas && S_{n}(k+1)\\
&=&\frac{1}{n+k+2} - \frac{k+1}{n+k+1} + \frac{(k+1)k}{2(n+k)}- \ldots + \frac{(-1)^{k+1}}{n+1}\\
&=&\left[\frac{1}{n+k+2} - \frac{k}{n+k+1} + \frac{k(k-1)}{2(n+k)} - \ldots+ \frac{(-1)^{k}}{n+2}\right] \\
&&-\left[\frac{1}{n+k+1} - \frac{2k}{2(n+k)} + \frac{3k(k-1)}{2.3(n+k-1)} - \ldots+ \frac{(-1)^{k}}{n+1}\right]\\
&=& \bigg[\frac{\binom{k}{0}}{(n+1)+k+1} - \frac{\binom{k}{1}}{(n+1)+k} + \frac{\binom{k}{2}}{(n+1)+k-1} - \ldots + (-1)^{k}\frac{\binom{k}{k}}{(n+1)+1}\bigg]\\
&&-\left[\frac{\binom{k}{0}}{n+k+1} - \frac{\binom{k}{1}}{(n+k)} + \frac{\binom{k}{2}}{(n+k-1)} - \ldots + (-1)^{k}\frac{\binom{k}{k}}{n+1}\right]\\
&=& S_{n+1}(k)-S_{n}(k)\\
&=&\frac{(-1)^{k} k!}{(n+k+2)(n+k+1)\ldots(n+2)}-\frac{(-1)^{k} k!}{(n+k+1)(n+k)\ldots(n+1)}\\
&=& \frac{(-1)^{(k+1)} (k+1)!}{(n+k+2)(n+k+1)\ldots(n+1)}.\eeas
So our claim has been established. We note that $S_{n}(m)=(-1)^{m}\prod\limits_{i=1}^{m}\frac{i}{(n+i)}\frac{1}{(n+m+1)}$ and hence it can not be an integer.
In particular, we have proved that $Q(1)\not =-2$; i.e., $P_{\ast}(1)\not=-1$.
\end{proof}
\begin{lem}\label{aml6} (\cite{br7a}) If $N(r,0;f^{(k)}\mid f\not=0)$ denotes the counting function of those zeros of $f^{(k)}$ which are not the zeros of $f$, where a zero of $f^{(k)}$ is counted according to its multiplicity, then $$N(r,0;f^{(k)}\mid f\not=0)\leq k\ol N(r,\infty;f)+N(r,0;f\mid <k)+k\ol N(r,0;f\mid\geq k)+S(r,f).$$
\end{lem}
\section {Proof of the theorem}
\begin{proof} [\textbf{Proof of Theorem \ref{amt1.1}}]  First we observe that since $P_{\ast}(0)=1\not =P_{\ast}(1)=Q(1)+1$, $P_{\ast}(z)$ is satisfying \enquote{poperty H}. Also $P_{\ast}(z)-1$ and $P_{\ast}(z)-P_{\ast}(1)$ have a zero of multiplicity $n+1$ and $m+1$ respectively at $0$ and $1$, it follows that the zeros of $P_{\ast}(z)$ are simple.
Let the zeros be given by $\alpha _{j}$, $j=1,2,\ldots,n+m+1$.
Since $E_{f}(S_{\ast},l)=E_{g}(S_{\ast},l)$, it follows that $F$, $G$ share $(0,l)$.\\
\textbf{Case-1.} First we suppose that $H\not\equiv 0$. \\
\textbf{Subcase-1.1.} If $l\geq 2$, then applying {\it Lemma \ref{aml6}}, we note that
\bea \label{ame3.1} & & \ol N_{0}(r,0;g^{'})+\ol N(r,0;G\mid\geq 2)+ \ol N_{*}(r,0;F,G) \\&\leq& \ol N_{0}(r,0;g^{'})+\ol N(r,0;G \mid\geq 2)+\ol N(r,0;G\mid\geq 3)\nonumber\\ &\leq& N(r,0;g^{'} \mid g\not=0)+S(r,g)\nonumber\\&\leq&\ol N(r,0;g)+\ol N(r,\infty;g)+S(r,g).\nonumber \eea
Hence for $\varepsilon >0$, using (\ref{ame3.1}), {\it Lemmas \ref{ML}}, {\it \ref{aml2}} and {\it \ref{aml3}}, we get from Second Fundamental Theorem that
\bea \label{ame3.1a} &&(n+m+2)\;T(r,f)\\
\nonumber &\leq & \ol N(r,\infty;f)+\ol N(r,0;f)+\ol N(r,1;f)+N(r,0;F\mid=1)+\ol N(r,0;F\mid\geq 2)\\
\nonumber & &-N_{0}(r,0;f^{'})+S(r,f)\\
\nonumber &\leq & 2\ol N(r,0;f)+2\ol N(r,1;f)+\ol N(r,0;g)+\ol N(r,1;g)+ 2\ol N(r,\infty;f) + \ol N(r,\infty; g)\\
\nonumber & & + \ol N_{0}(r,0;g') + \ol N(r,0;G\mid\geq 2) + \ol N_{*}(r,0;F,G) + S_{\ast}(r)\\
\nonumber &\leq& 2\{\ol N(r,0;f)+\ol N(r,1;f)+ \ol N(r,\infty;f)+\ol N(r,0;g)+ \ol N(r,\infty; g)\}+\ol N(r,1;g)\\
\nonumber && + S_{\ast}(r)\\
\nonumber &\leq & (11-2\Theta (0;f)-2\Theta (\infty;f)-2\Theta (1;f)-2\Theta (0;g)-2\Theta (\infty;g)-\Theta (1;g)+\varepsilon )T_{\ast}(r)\\
\nonumber && +S_{\ast}(r).\nonumber\eea \par
In a similar way, we can obtain
\bea \label{ame3.1aa} (n+m+2)T(r,g)&\leq& (11-2\Theta (0;f)-2\Theta (\infty;f)-\Theta (1;f)-2\Theta (0;g)\\&&-2\Theta (\infty;g)-2\Theta (1;g)+\varepsilon )T_{\ast}(r)+S_{\ast}(r).\nonumber\eea
Combining (\ref{ame3.1a}) and (\ref{ame3.1aa}), we see that
\bea\label{ame3.1aaa}&&(n+m-9+2\Theta (0;f)+2\Theta (\infty;f)+\Theta (1;f)+2\Theta (0;g)\\
\nonumber &&+2\Theta (\infty;g)+\Theta (1;g)+\min\{\Theta (1;f),\Theta (1;g)\}-\varepsilon )T_{\ast}(r) \leq S_{\ast}(r).\eea
Since $\varepsilon > 0$ is arbitrary, inequality  (\ref{ame3.1aaa}) leads to a contradiction.
\newpage
\textbf{Subcase-1.2.} If $l=1$, then using {\it Lemma \ref{aml6}}, we can write (\ref{ame3.1}) as
\bea\label{ame3.1e} && \ol N_{0}(r,0;g^{'})+\ol N(r,0;G\mid\geq 2)+ \ol N_{*}(r,0;F,G)\\
\nonumber&\leq& \ol N_{0}(r,0;g^{'})+\ol N(r,0;G \mid\geq 2)+\ol N_{L}(r,0;G)+\ol N(r,0;F\mid\geq 3)\\
\nonumber&\leq& N(r,0;g^{'}\mid g\not=0)+\sum_{j=1}^{n+m+1}\ol N(r,\alpha _{j};f\mid\geq 3)\\
\nonumber&\leq&\ol N(r,0;g)+\ol N(r,\infty;g)+\frac{1}{2}\sum_{j=1}^{n+m+1}\{ N(r,\alpha_{j};f)-\ol N(r,\alpha_{j};f)\}+S(r,g)\nonumber\\&\leq& \ol N(r,0;g)+\ol N(r,\infty;g)+\frac{1}{2} N(r,0;f^{'}\mid f\not = 0)+S(r,g)\nonumber\\&\leq& \ol N(r,0;g)+\ol N(r,\infty;g)+\frac{1}{2}\{\ol N(r,0;f)+\ol N(r,\infty;f)\}+S(r,f)+S(r,g).\nonumber\eea
Thus for $\varepsilon >0$, using (\ref{ame3.1e}), {\it Lemmas \ref{aml2}} and {\it \ref{aml3}} and proceeding as in (\ref{ame3.1a}), we get from Second Fundamental Theorem that
\bea \label{ame3.1b} (n+m+2)T(r,f) &\leq& (12-2\Theta (0;f)-2\Theta (\infty;f)-2\Theta (1;f)\\
&-&2\Theta (0;g)-2\Theta (\infty;g)-\Theta (1;g)+\varepsilon )T_{\ast}(r)+S_{\ast}(r)\nonumber.\nonumber\eea
Similarly, we can obtain
\bea \label{ame3.1bb} (n+m+2)T(r,g) &\leq& (12-2\Theta (0;f)-2\Theta (\infty;f)-\Theta (1;f)\\
&-&2\Theta (0;g)-2\Theta (\infty;g)-2\Theta (1;g)+\varepsilon )T_{\ast}(r)+S_{\ast}(r).\nonumber\eea
Combining (\ref{ame3.1b}) and (\ref{ame3.1bb}), we see that
\bea\label{ame3.1bbb} &&(n+m-10+2\Theta (0;f)+2\Theta (\infty;f)+\Theta (1;f)+2\Theta (0;g)\\
\nonumber&&+2\Theta (\infty;g)+\Theta (1;g)+\min\{\Theta (1;f),\Theta (1;g)\}-\varepsilon )T_{\ast}(r) \leq S_{\ast}(r).\eea
Since $\varepsilon > 0$ is arbitrary, inequality (\ref{ame3.1bbb}) leads to a contradiction.\\
\textbf{Subcase-1.3.} If $l=0$, then using {\it Lemma \ref{aml6}}, we note that
\bea \label{ame3.1ee}& &\ol N_{0}(r,0;g^{'})+\ol N^{(2}_{E}(r,0;F)+2\ol N_{L}(r,0;G)+2\ol N_{L}(r,0;F)\\&\leq & \ol N_{0}(r,0;g^{'})+\ol N^{(2}_{E}(r,0;G)+\ol N_{L}(r,0;G)+\ol N_{L}(r,0;G)+2\ol N_{L}(r,0;F)\nonumber \\&\leq& \ol N_{0}(r,0;g^{'})+\ol N(r,0;G \mid\geq 2)+\ol N_{L}(r,0;G)+2\ol N_{L}(r,0;F)\nonumber\\&\leq& N(r,0;g^{'}\mid g\not=0)+\ol N(r,0;G \mid\geq 2)+2\ol N(r,0;F \mid\geq 2)\nonumber\\&\leq & 2\ol N(r,0;g)+2\ol N(r,\infty;g)+2\ol N(r,0;f)+2\ol N(r,\infty;f)+S_{\ast}(r).\nonumber\eea
Using Second Fundamental Theorem, we have
\bea \label{ame3.1csa}\;\;\;\;& & (n+m+2)T(r,f)\\ &\leq & \ol N(r,\infty;f)+\ol N(r,0;f)+\ol N(r,1;f)+ N^{1)}_{E}(r,0;F)+\ol N_{L}(r,0;F)+\ol N_{L}(r,0;G)\nonumber\\
\nonumber &&+\ol N^{(2}_{E}(r,0;F)-N_{0}(r,0;f^{'})+S(r,f).
\eea
Hence using (\ref{ame3.1csa}), (\ref{ame3.1ee}), {\it Lemmas \ref{ML}}, {\it \ref{aml2}} and {\it \ref{aml3}}, we get for $\varepsilon >0$ that
\bea \label{ame3.1c}\;\;\;\;& & (n+m+2)T(r,f)\\
\nonumber &\leq & 2\{\ol N(r,0;f)+\ol N(r,1;f)\}+\ol N(r,0;g)+\ol N(r,1;g)+ 2\ol N(r,\infty;f) + \ol N(r,\infty; g)\\
\nonumber & &+\ol N^{(2}_{E}(r,0;F) +2\ol N_{L}(r,0;G)+2\ol N_{L}(r,0;F)+\ol N_{0}(r,0;g^{'})+S(r,f)+S(r,g)\\
\nonumber &\leq & (17-3\Theta (0;f)-3\Theta (\infty;f)-2\Theta (1;f)-3\Theta (0;g)\nonumber-3\Theta (\infty;g)\\
\nonumber &&-\Theta (1;g)+\varepsilon )T_{\ast}(r)+S_{\ast}(r).\eea
In a similar manner, we can obtain
\bea \label{ame3.1cc} && (n+m+2)T(r,g)\\
\nonumber &\leq& (17-3\Theta (0;f)-3\Theta (\infty;f)-\Theta (1;f)-3\Theta (0;g)\\&&-3\Theta (\infty;g)-2\Theta (1;g)+\varepsilon )T_{\ast}(r)+S_{\ast}(r).\eea
Combining (\ref{ame3.1c}) and (\ref{ame3.1cc}), we see that
\bea \label{ame3.1ccc}&&(n+m-15+3\Theta (0;f)+3\Theta (\infty;f)+\Theta (1;f)+3\Theta (0;g)\\&&+3\Theta (\infty;g)+\Theta (1;g)+\min\{\Theta (1;f),\Theta (1;g)\}-\varepsilon )T_{\ast}(r) \leq S_{\ast}(r).\nonumber\eea
Since $\varepsilon >0$ be arbitrary, inequality (\ref{ame3.1ccc}) leads to a contradiction.\\
\textbf{Case-2.} Next we suppose that $H\equiv 0$.\\
On integration, we get from (\ref{CHB}) \be\label{ame3.4}\frac{1}{F}\equiv\frac{A}{G}+B,\ee where $A$, $B$ are constants with $A\not=0$.\\
From {\it Lemma \ref{ML}}, we get \be\label{ame3.5} T(r,f)=T(r,g)+S(r,g).\ee
\textbf{Subcase-2.1.} First suppose that $B\not =0$.\\
From (\ref{ame3.4}), we have $$\ol N\left(r,\frac{-A}{\;B};G\right)=\ol N(r,\infty;f).$$
\textbf{Subcase-2.1.1.} Let $\frac{-A}{\;B}\not =1$.\par
If $\frac{-A}{\;B}\not =Q(1)+1$, then in view of (\ref{ame3.5}) and the Second Fundamental Theorem, we get
\beas &&(n+2m+1)T(r,g)\\
&\leq& \ol N(r,\infty;g)+\ol N\left(r,1;G\right)+\ol N\left(r,\frac{-A}{\;B};G\right)+S(r,g)\\
&=&\ol N(r,\infty;g)+\ol N(r,0;g)+m T(r,g)+\ol N(r,\infty;f)+S(r,g)\\
&\leq &(m+2)T(r,g)+\ol N(r,\infty;f)+S(r,g)\\
&\leq& (m+3)T(r,g)+S(r,g),
\eeas which  is a contradiction for $n \geq 3$ and $m\geq 3$.\\
\par
If $\frac{-A}{\;B}=Q(1)+1$, then from (\ref {ame3.4}), we have
\bea\label {ame3.5a}\frac{G}{B F}=G-P_{\ast}(1)=G+\frac{A}{B}=(g-1)^{m+1}(g-\alpha^{'} _{1})(g-\alpha^{'} _{2})\ldots (g-\alpha^{'} _{n}),\eea where $\alpha^{'} _{i}$, $i=1,2,\ldots,n$ are the distinct simple zeros of $P_{\ast}(z)+\frac{A}{B}$.
As $B\not= 0$, $f$ and $g$ do not have any common pole. Let $z_{0}$ be a zero of $g-1$ of multiplicity p (say), then it must be a pole of $f$ with multiplicity $q\geq 1$ (say). So from (\ref{ame3.5a}), we have $$(m+1)p=(n+m+1)q\geq m+n+1.$$ i.e., $$p\geq \frac{n+m+1}{m+1}>1.$$
Next suppose $z_{i}$ be a zero of $g-\alpha^{'} _{i}$ of multiplicity $p_{i}$, then in view of (\ref{ame3.5a}), we have $z_{i}$ be a pole of $f$ of multiplicity $q_{i}$, (say) such that
$$p_{i}=(n+m+1)q_{i}\geq n+m+1.$$ Let $\beta _{j}$, $j=1,2,\ldots,m$ be the distinct simple zeros of $P_{\ast}(z)-1$. Now from the Second Fundamental Theorem, we get
\beas &&(n+m+1)T(r,g)\\&\leq& \ol N(r,\infty;g)+\ol N\left(r,1;G\right)+\ol N\left(r,\frac{-A}{\;B};G\right)+S(r,g)\\&\leq& \ol N(r,\infty;g)+\ol N(r,0;g)+\sum\limits_{j=1}^{m}\ol N(r,\beta _{j};g)+\ol N(r,1;g)+\sum\limits_{i=1}^{n}\ol N(r,\alpha^{'} _{i};g)+S(r,g)\\&\leq& \left(m+2+\frac{1}{2}+\frac{n}{n+m+1}\right)T(r,g)+S(r,g),\eeas which  is a contradiction for $n\geq 3$, $m\geq 3$.\\
\textbf{Subcase-2.1.2.}  Next let $\frac{-A}{\;B}=1$.\\
From (\ref{ame3.4}), we have $$\frac{1}{F}=\frac{B(G-1)}{G}.$$ Therefore in view of (\ref{ame3.5}), Second Fundamental Theorem yields
\beas&& T(r,g)+S(r,g)\geq \ol N(r,\infty;f)=\ol N(r,1;G)=\ol N(r,0;g)+\sum\limits_{j=1}^{m}\ol N(r,\beta _{j};g)\\&&\geq (m-1)T(r,g)+S(r,g),\eeas
a contradiction as $m\geq 3$.\\
\textbf{Subcase-2.2.} Now suppose that $B=0$.\\ From (\ref{ame3.4}), we get \be\label{ame3.6} AF\equiv G.\ee
\textbf{Subcase-2.2.1.} Suppose $A\not =1$. \\
If $A=P_{\ast}(1)$, then from (\ref{ame3.6}), we have $$P_{\ast}(1)\left(F-\frac{1}{P_{\ast}(1)}\right)\equiv G-1.$$
As $P_{\ast}(1)\not =1$ and {\it Lemma \ref{aml4}} implies $P_{\ast}(1)\not =-1$, we have $\frac{1}{P_{\ast}(1)}\not=P_{\ast}(1)$, it follows that $P_{\ast}(z)-\frac{1}{P_{\ast}(1)}$ has simple zeros. Let they be given by $\gamma _{i}$, $i=1,2,\ldots, n+m+1$.
So from the Second Fundamental Theorem and (\ref{ame3.5}), we get
\beas(n+m-1)T(r,f)&\leq& \sum\limits_{i=1}^{n+m+1}\ol N(r,\gamma _{i};f)+S(r,f)\\&\leq& \ol N(r,0;g)+\sum\limits_{j=1}^{m}\ol N(r,\beta _{j};g)\leq (m+1) T(r,f)+S(r,f),\eeas a contradiction since $n\geq 3$.\\
If $A\not =P_{\ast}(1)$, then we have from (\ref{ame3.6}) $$A(F-1)\equiv G-A.$$
Let the distinct zeros of $P_{\ast}(z)-A$ be given by $\delta _{i}$, $i=1,2,\ldots, n+m+1$.
 So from the Second Fundamental Theorem and (\ref{ame3.5}), we get
 \beas(n+m-1)T(r,g)&\leq& \sum\limits_{i=1}^{n+m+1}\ol N(r,\delta _{i};g)+S(r,g)\\&=& \sum\limits_{j=1}^{m}\ol N(r,\beta _{j};f)+\ol N(r,0;f)+S(r,f)\\&\leq& (m+1)T(r,g)+S(r,g),\eeas a contradiction since $n\geq 3$.\\
\textbf{Subcase-2.2.2.} Suppose $A=1$.\\
Then from (\ref{ame3.6}), we have $$F\equiv G;~~\text{i.e.},~~ P_{\ast}(f)\equiv P_{\ast}(g).$$
Here $k=2$, $d_{1}=0$, $d_{2}=1$, $q_{1}=n$, $q_{2}=m$. Since $\min \{q_{1},q_{2}\}=\min \{n,m \}\geq 2$ and $n+m\geq 5$, we see that $nm>n+m$.\\ So from  Theorem 5.B, we conclude that $P_{\ast}(z)$ is an uniqueness polynomial. Therefore $f\equiv g$. This proves the theorem.
\end{proof}
$~~$
\vspace{5 cm}
\\
------------------------------------------------
\\
\textbf{The matter of this chapter has been published in  Afr. Mat., Vol. 26, No. 7-8, (2015), pp. 1561-1572.}
\newpage
\chapter{On some sufficient conditions of the strong uniqueness polynomials}
\fancyhead[l]{Chapter 6}
\fancyhead[r]{On some sufficient conditions of the strong uniqueness polynomials}
\fancyhead[c]{}
\section{Introduction}
In quest for the minimum admissible cardinality of a finite unique range sets, Li-Yang (\cite{am12}) first realized that the finite unique range sets are nothing but the set of distinct zeros of some suitable polynomials. They observed that if $S=\{a_{1},a_{2},\ldots,a_{n}\}\subset \mathbb{C}$, with $a_{i}\not=a_{j}$, is a unique range set for meromorphic (resp. entire) functions then the polynomial $P(z)=(z-a_{1})(z-a_{2})\ldots(z-a_{n})$ has the property that $P(f)\equiv P(g)$ implies $f\equiv g$ for any pair of non-constant meromorphic (resp. entire) functions $f$ and $g$.\par
This type of polynomial was called as uniqueness polynomial (\cite{am12}) for meromorphic (resp. entire) functions, which was later called by Fujimoto (\cite{am2a}) as the uniqueness polynomial in a broad sense. Subsequently this realization of polynomial backbone of a finite unique range sets opened a new era in the uniqueness theory of entire and meromorphic functions.\par
The following results may be considered as the initial characterizations of a uniqueness polynomial. At first we note that every one degree polynomial is uniqueness polynomial.
\begin{exm} (\cite{am12}) No polynomial of degree two is a UPE. For example, if we take $P(z)=az^{2}+bz+c~(a\not=0)$ and for any entire function $f$, define $g:=-f-\frac{b}{a}$, then $P(f)=P(g)$ but $f\not=g$.
\end{exm}
\begin{exm} (\cite{am12}) No polynomial of degree three is a UPE. For example, if we take, $P(z)=z^{3}-az+b$, $f(z)=\frac{\omega_{2}e^{z}}{\omega_{2}-\omega_{1}}-\frac{a\omega_{1}e^{-z}}{\omega_{2}-\omega_{1}}$ and
$g(z)=\frac{e^{z}}{\omega_{2}-\omega_{1}}-\frac{ae^{-z}}{\omega_{2}-\omega_{1}}$, where $\omega_{i}$ are non-real cubic roots of unity, then $P(f)=P(g)$ but $f\not=g$.
\end{exm}
\begin{exm} (\cite{am12}) Let $P(z)=z^{4}+a_{3}z^{3}+a_{2}z^{2}+a_{1}z+a_{0}$. Then $P(z)$ is not a UPM but $P(z)$ is a UPE if and only if $(\frac {a_{3}}{2})^{3}-\frac {a_{2}a_{3}}{2}+a_{1}\not=0$.
\end{exm}
\begin{exm}\label{ame1} (\cite{am6.1}) Let $P(z)=z^{n}+a_{n-1}z^{n-1}+\ldots+a_{1}z+a_{0}~(n\geq4)$ be a monic polynomial. If there exist an integer $t$ with $1\leq t < n-2$ and $\gcd(n,t)=1$ such that $a_{n-1}=\ldots=a_{t+1}=0$ but $a_{t} \neq 0$, then $P(z)$ is a UPE.
\end{exm}
\begin{exm}\label{ame2} (\cite{am6.1}) Let $P(z)=z^{n}+a_{m}z^{m}+a_{0}$ be a monic polynomial such that $\gcd(n,m)=1$ and  $a_{m} \neq 0$. If $n\geq 5$ and  $1\leq m < n-1$, then $P(z)$ is a UPM.
\end{exm}
\begin{exm} (\cite{am12}) If $P_1(z)$ is a UPM (resp. UPE) and $P_2(z)$ is a polynomial, then $(P_1{}\circ P_{2})(z)$ is a UPM (resp. UPE) if and only if $P_2(z)$ is a UPM (resp. UPE).
\end{exm}
While searching some sufficient conditions for a polynomial to be a UPM or UPE, Fujimoto (\cite{am2, am2a}) introduced a variant of the notion of the uniqueness polynomial, which was later called by An-Wang-Wong (\cite{am1a}) as a strong uniqueness polynomial.
\begin{defi}
A polynomial $P(z)$ in $\mathbb{C}$ is called a strong uniqueness polynomial for meromorphic (resp. entire) functions if for any non-constant meromorphic (resp. entire) functions $f$ and $g$, $P(f)\equiv AP(g)$ implies $f\equiv g$, where $A$ is any non-zero constant. In short, we say $P(z)$ is a SUPM (resp. SUPE).
\end{defi}
It is clear from the above definitions that a strong uniqueness polynomial is a uniqueness polynomial but the following examples show that a uniqueness polynomial may not be strong uniqueness polynomial.
\begin{exm}\label{amex1} Let $P(z)=az+b$ $(a\neq 0)$. Clearly $P(z)$ is a UPM (resp. UPE) but for any non-constant meromorphic (resp. entire) function $g$, if we take $f := c g-\frac{b}{a}(1-c)$  $(c \neq 0,1)$, then $P(f)=cP(g)$ but $f \neq g$.
\end{exm}
\begin{exm}\label{kly} Consider $P(z)=z^{n-r}(z^{r}+a)$ where $a$ is a non-zero complex number and $\gcd(n,r)=1$, $r\geq 2$ and $n\geq 5$. Then $P(z)$ is a uniqueness polynomial as shown in Example \ref{ame2} but for any non-constant meromorphic function $g$ if we take $f:=\omega g$  where $\omega$ is non-real $r$-th root of unity. Then $P(f)=\omega^{n-r}P(g)$.
\end{exm}
The key discovery of Fujimoto (\cite{am2}) is to highlight a special property of a polynomial, which plays a pivotal role in characterizing a SUPM or SUPE. This property was called by Fujimoto himself as \enquote{Property (H)}. Later on An-Wang-Wong (\cite{am1a}) and An (\cite{an}) referred this property as \enquote{separation condition} and \enquote{injective condition} respectively. Recently Banerjee-Lahiri (\cite{bl}) renamed the \enquote{Property (H)} as \enquote{critical injection property}.
\begin{defi}
 A polynomial $P(z)$ is said to satisfy critical injection property if $P(\alpha )\not =P(\beta )$ for any two distinct zeros $\alpha $, $\beta $ of the derivative $P'(z)$.
 \end{defi}
Clearly the inner meaning of critical injection property is that the polynomial $P(z)$ is injective on the set of distinct zeros of $P^{'}(z)$, which are known as critical points of $P(z)$.\par
In this connection, Fujimoto (\cite{am2a}) gave some sufficient conditions for a critically injective polynomial, with $k(\geq 2)$-critical points, to be uniqueness polynomials (see Theorem 5.B), which help us to find many uniqueness polynomials.\par
Like uniqueness polynomial, Fujimoto also did some remarkable investigations to find some sufficient conditions for a critically injective polynomial to be a strong uniqueness polynomial. In this connection, Fujimoto proved the following theorems.
\begin{theo 6.A}\label{amih1} (\cite{am2}) A critically injective polynomial $P(z)$, with $k\geq 4$, is a strong uniqueness polynomial if $$P(d_{1})+P(d_{2})+\ldots+P(d_{k}) \neq 0.$$
\end{theo 6.A}
\begin{rem}
As an application of Theorem 6.A, Fujimoto (\cite{am2}) himself proved that $P(z)=z^{n}+az^{n-r}+b$  is a strong uniqueness polynomial if $r\geq3$, $n>r+1$, $\gcd (n, r)=1$ and  $ab\not=0$, which is an improvement of a result of Yi (\cite{am15}).
\end{rem}
\begin{theo 6.B} (\cite{am2a}) A critically injective polynomial $P(z)$, with $k=3$, is a strong uniqueness polynomial if $\max (q_{1},q_{2},q_{3})\geq 2$ and
\beas && \frac{P(d_{l})}{P(d_{m})}\neq \pm 1~~\text{for}~~1\leq l<m \leq 3,\\
&& \frac{P(d_{l})}{P(d_{m})}\neq \frac{P(d_{m})}{P(d_{n})}~~\text{for~~any~~~permutation}~~ (l,m,n)~~\text{ of}~~(1,2.3).\eeas
\end{theo 6.B}
\begin{theo 6.C} (\cite{am2a}) A critically injective polynomial $P(z)$, with $k=2$ and $q_{1}\leq q_{2}$, is a strong uniqueness polynomial if ether of the following conditions holds:
\begin{enumerate}
\item [i)]  $q_{1}\geq 3$ and $P(d_{1})+P(d_{2})\not=0$,
\item [ii)] $q_{1}\geq 2$ and $q_{2}\geq q_{1}+3$.
\end{enumerate}
\end{theo 6.C}
We noticed that Theorems 6.A, 6.B and 6.C are related to some sufficient conditions for a critically injective polynomial to be strong uniqueness polynomial. But none of the above theorems told about the case when a polynomial has exactly one critical point.\par The following example shows that if $k=1$, then that polnomial can't be a strong uniqueness polynomial.
\begin{exm}\label{amk}
For $k=1$, taking $P(z)=(z-a)^{q}-b$ for some constants $a$ and $b$ with $b\not=0$ and an integer $q\geq 2$, it is easy to verify that for an arbitrary non-constant meromorphic function $g$ and a constant $c(\not=1)$ with $c^{q}=1$, the function $g:=cf+(1-c)a(\not=f)$ satisfies the condition $P(f)=P(g)$.
\end{exm}
\begin{rem}
From the above example, it is also observed that for  $k=1$ no polynomial is uniqueness polynomial.\end{rem}
Though Fujimoto did investigations to find some sufficient conditions for a critically injective polynomial to be a strong uniqueness polynomial but so far no attempt have been made by any researchers to find some sufficient conditions for a UPM to be SUPM. To deal in this perspective is the main motivation of this chapter.
\section{Main Results}
\par
We have already seen from the Example \ref{amk} that a polynomial having only one critical points can't be a uniqueness polynomial. So uniqueness polynomials has at least two critical points. Now we state our results.
\begin{theo}\label{amith1} Suppose $P(z)$ is a critically injective uniqueness polynomial of degree $n$ with simple zeros. Further suppose that $P(z)$ has at least two critical points and among them let $\alpha$ and $\beta$ be the two critical points with maximum multiplicities.\par
If $z=\alpha$ is a $P(\alpha)$ point of $P(z)$ of order $p$ and $z=\beta$ is a $P(\beta)$ point of $P(z)$ of order $t$ with $\max\{t,p\}+t+p\geq 5+n$ and $\{P(\alpha)+P(\beta)\} \not=0$, then $P(z)$ is a strong uniqueness polynomial.
\end{theo}
\begin{rem} As $\alpha$ and $\beta$ are critical points of $P(z)$, so $t$, $p \geq 2$.
\end{rem}
\begin{exm}\label{amiex1} Firts we consider the Frank-Reinders (\cite{am3}) polynomial $$P_{FR}(z)=\frac{(n-1)(n-2)}{2}z^{n}-n(n-2)z^{n-1}+\frac{n(n-1)}{2}z^{n-2}-c,$$ where $n\geq6$ and $c \neq 0, ~1,~\frac{1}{2}$.
\end{exm}
It is clear that $P_{FR}(z)$ is a critically injective polynomial with only simple zeros as $c \neq 0$, $1$, $\frac{1}{2}$. Also we note that $P_{FR}(z)-P_{FR}(1)=(z-1)^{3}R_{1}(z)$ and  $P_{FR}(z)-P_{FR}(0)=z^{n-2}R_{2}(z)$, where $R_{i}(z)$ ($i=1,2$) has no multiple zero with $R_{1}(1)\not=0$ and $R_{2}(0)\not=0$.\par
Again in view of Theorem 5.B, $P_{FR}(z)$ is a uniqueness polynomial for $n\geq 5$. Then applying Theorem \ref{amith1}, we get that $ P_{FR}(z)$ is a SUPM if  $c \neq 0, 1,\frac{1}{2}$ and $\max\{n-2,3\}+(n-2)+3\geq 5+n$; i.e., $n\geq 6$.
\begin{exm}\label{amiabn}  Next we consider the polynomial which we already introduced in the previous chapter
$$P_{B}(z)=\sum\limits_{i=0}^{m} \binom{m}{i}\frac{(-1)^{i}}{n+m+1-i}z^{n+m+1-i} + c,$$
where $c \neq0,-\lambda,-\frac{\lambda}{2}$ and $\lambda=\sum\limits_{i=0}^{m} \binom{m}{i}\frac{(-1)^{i}}{n+m+1-i}$.
\end{exm}
It is clear that Lemma \ref{aml4} implies $\lambda \neq 0$. Also $P_{B}(z)$ is critically injective polynomial with only simple zeros, since $P'_{B}(z)=z^{n}(z-1)^{m}$ and $c \neq 0, -\lambda$.\par
Moreover, $P_{B}(z)-P_{B}(1)=(z-1)^{m+1}R_{3}(z)$ and $P_{B}(z)-P_{B}(0)=z^{n+1}R_{4}(z)$ where $R_{i}(z)$ ($i=3,4$) has no multiple zero with $R_{3}(1)\not=0$, and $R_{4}(0)\not=0$.\par
Again Theorem 5.B yields that $P_{B}(z)$ is a uniqueness polynomial when $\min\{m,n\}\geq2$ and $m+n\geq5$. Since $c \neq -\frac{\lambda}{2}$ and $P_{B}(1) + P_{B}(0) \not=0$, so in view of Theorem \ref{amith1}, $P_{B}(z)$ is a strong uniqueness polynomial if $\min\{m,n\}\geq2$ and $m+n\geq5$ and $\max\{m+1,n+1\}+(m+1)+(n+1)\geq 5+(m+n+1)$; i.e., $\max\{m,n\} \geq 3$.
\begin{rem} If we take $n=3,~m=2$ or $n=2,~m=3$, then by above discussion it is clear that $P_{B}(z)$ is a six degree strong uniqueness polynomial.
\end{rem}
Inspired by the above Example \ref{amiabn}, we first introduce most general form of $P_{B}(z)$ and we shall show that Theorem \ref{amith1} is also applicable to it.
\begin{exm}\label{amihii}
Let us define
$$P_{\star}(z)=\sum\limits_{i=0}^{m} \sum\limits_{j=0}^{n} \binom{m}{i}\binom{n}{j}\frac{(-1)^{i+j}}{n+m+1-i-j}z^{n+m+1-i-j}a^{j}b^{i}+c=Q(z)+c,$$
where $a, b$ be two complex numbers such that $b\neq0$, $a\not=b$ and
$$c\not\in \{0,-Q(a),-Q(b),-\frac{Q(a)+Q(b)}{2}\}.$$
\beas\text{Clearly,~~~~~~} P_{\star}'(z) &=& \sum\limits_{i=0}^{m} \sum\limits_{j=0}^{n} \binom{m}{i}\binom{n}{j}(-1)^{i+j}z^{m+n-i-j}a^{j}b^{i}\\
&=& \bigg(\sum\limits_{i=0}^{m}(-1)^{i}\binom{m}{i}z^{m-i}b^{i}\bigg)\bigg(\sum\limits_{j=0}^{n}\binom{n}{j}(-1)^{j}z^{n-j}a^{j}\bigg)\\
&=& (z-b)^{m}(z-a)^{n}.\eeas
To this end, first we note that\\
$P_{\star}(z)-P_{\star}(b)=(z-b)^{m+1}R_{5}(z)$ where $R_{5}(z)$ has no multiple zero and $R_{5}(b)\not=0$, and
$P_{\star}(z)-P_{\star}(a)=(z-a)^{n+1}R_{6}(z)$ where $R_{6}(z)$ has no multiple zero and $R_{6}(a)\not=0$.\par
So $P_{\star}(a)= P_{\star}(b)$ implies $(z-b)^{m+1}R_{5}(z)=(z-a)^{n+1}R_{6}(z)$. As we choose $a \not=b$ so $R_{5}(z)$ has a factor $(z-a)^{n+1}$ which implies the polynomial $P_{\star}(z)$ is of degree at least $m+1+n+1$, a contradiction.\par
Also by the assumption on $c$, it is clear that $P_{\star}(a)+P_{\star}(b) \not=0$ and $P_{\star}(a)P_{\star}(b) \not=0$. Thus $P_{\star}(z)$ has no multiple zero.\par
Again by Theorem 5.B, we see that $P_{\star}(z)$ is a uniqueness polynomial when $m+n\geq 5$ and $\min\{m,n\}\geq 2$. Thus if $m+n\geq 5$, $\max\{m,n\} \geq 3$ and $\min\{m,n\}\geq 2$, then $P_{\star}(z)$ is a strong uniqueness polynomial for meromorphic functions by Theorem \ref{amith1}.
\end{exm}
\begin{rem} If we take $n=3,~ m=2$ or $n=2,~ m=3$, then by above discussion $P_{\star}(z)$ is a six degree strong uniqueness polynomial.
\end{rem}
\begin{cor}\label{amihi}
If we take $a=0$ and $b=1$ in above example, then we get Example \ref{amiabn}.
\end{cor}
\begin{rem}
If we take $a=0$ and $b\not=0$ in the previous example, then we have the following polynomial:
$$P_{B}(z)=\sum\limits_{i=0}^{m} \binom{m}{i}\frac{(-1)^{i}}{n+m+1-i}z^{n+m+1-i}b^{i} + c,$$
where $bc \neq 0$, $c\not=-b^{n+m+1}\lambda,~\frac{-b^{n+m+1}\lambda}{2}$, where $\lambda$ is defined as in the previous example.
\end{rem}
Then clearly when $m+n\geq 5$, $\max\{m,n\}\geq 3$ and $\min\{m,n\}\geq 2$, $P_{B}(z)$ is a strong uniqueness polynomial.
\begin{rem}
The above examples are related to the strong uniqueness polynomials with two critical points. Now we are giving the following example where there are more than two critical points, and in view of Theorem \ref{amith1}, one can easily verify that it is a strong uniqueness polynomial.
\end{rem}
\begin{exm}
 Consider the polynomial $P(z)= z^{n}-\frac{n}{m}z^{m}+b$ where $n-m\geq 2$.
\end{exm}
Then clearly $P(z)$ has at least three critical points. Also $P(z)-P(1)=(z-1)^{2}T_{1}(z)$, where $T_{1}(1)\not=0$ and $P(z)-P(0)=z^{m}T_{2}(z)$, where $T_{2}(0)\not=0$.\par
Also we have already seen that $P(z)$ is a uniqueness polynomial for $n-m\geq2$, $\gcd(m,n)=1$ and $n\geq5$ (see Example \ref{ame2}). Thus by applying Theorem \ref{amith1}, we can conclude that $P(z)$ is a strong uniqueness polynomial when $b \not\in\{0,\frac{n-m}{m},\frac{n-m}{2m}\}$ and $\max\{m,2\}+m-n\geq3$, $n-m\geq2$, $\gcd(m,n)=1$, $n\geq5$.
\begin{rem}
For $n=7$, $m=5$ with proper choice of $b$, we can have seven degree Yi-type strong uniqueness polynomial.
\end{rem}
\begin{theo}\label{amith41} Suppose $P(z)$ is a critically injective uniqueness polynomial of degree $n$  with simple zeros having at least two critical points, say  $\gamma$ and $\delta$. Further suppose that the total number of $P(\gamma)$ and $P(\delta)$ points of $P(z)$ are respectively $p$ and $q$ with  $|p-q|\geq 3$.\par If for any complex number $d \not\in\{P(\gamma),P(\delta)\}$, $(P(z)-d)$ has at least $\min\{p+3,q+3\}$ distinct zeros, then $P(z)$ is a strong uniqueness polynomial.
\end{theo}
The following Corollary is an immediate consequence of Theorem \ref{amith41}.
\begin{cor}\label{amicor2.1} Suppose $P(z)$ is a critically injective uniqueness polynomial of degree $n$  with simple zeros having at least two critical points. Further suppose that $P(z)$ has two critical points, say  $\gamma$ and $\delta$ satisfying $P(\delta)=1$, $(P(\gamma))^{2} \not=0,~1$.\par
If for any complex number $d \not\in \{P(\gamma),P(\delta)\}$, $(P(z)-d)$ has at least $q+3$ distinct zeros, where $q$ is the total number of $P(\delta)$ points of $P(z)$, then $P(z)$ is a strong uniqueness polynomial.
\end{cor}
\begin{exm}\label{amih}
Consider the polynomial
$$P_{B_{1}}(z)=\sum\limits_{i=0}^{m} \binom{m}{i}\frac{(-1)^{i}}{n+m+1-i}z^{n+m+1-i}b^{i} + 1,$$
where we choose $b (\neq0)$ such a manner that $b^{n+m+1}\sum\limits_{i=0}^{m} \binom{m}{i}\frac{(-1)^{i}}{n+m+1-i}\not= -1,-2$.
\end{exm}
Here $P_{B_{1}}'(z)=(z-b)^{m}z^{n}$. Thus $P_{B_{1}}(z)$ is critically injective polynomial with no multiple zero when  $\min\{m,n\}\geq 2$.\par
If we take $m,n\in\mathbb{N}$ with $m+n\geq 5$ and $\min\{m,n\}\geq 2$, then by Theorem 5.B, $P_{B_{1}}(z)$ is a uniqueness polynomial. \par
Also $(P_{B_{1}}(z)-d)$ has exactly $m+n+1$ distinct zeros for any complex number $d \in \mathbb{C}\backslash\{P_{B_{1}}(b),P_{B_{1}}(0)\}$, otherwise there exist at least one complex number $\varsigma$ which is a zero of $\left(P_{B_{1}}(z)-d\right)$ of multiplicity at least 2. Consequently $P_{B_{1}}(\varsigma)=d$ and  $P_{B_{1}}'(\varsigma)=0$. That is, $d \in \{P_{B_{1}}(0),P_{B_{1}}(b)\}$, which is absurd.\par
Thus in view of Corollary \ref{amicor2.1}, $P_{B_{1}}(z)$ is a strong uniqueness polynomial if  $m+n\geq 5$, $\min\{m,n\}\geq 2$ and $n\geq 3$.
\begin{rem}
Example \ref{amih} gives the answer of the question raised in the paper (\cite{ami1}).
\end{rem}
We have observed from Example \ref{kly} that uniqueness polynomial may contain multiple zeros. However the two theorems so far stated are dealing with strong uniqueness polynomials with simple zeros. So natural question would be whether there exist a strong uniqueness polynomial which has multiple zeros? The next theorem shows that the answer is affirmative.\par
Next we shall demonstrate the following strong uniqueness polynomial with multiple zero of degree $n\geq 6$.
\begin{theo}\label{amith2}
Let $$\overline{P}(z)=z^{n}+az^{n-1}+bz^{n-2},$$ where $ab \neq0$ and $a^{2}=\lambda b$ where $\lambda =4\left(1-\frac{1}{(n-1)^{2}}\right)$, then $\overline{P}(z)$ is a strong uniqueness polynomial of degree $n\geq 6$.
\end{theo}
\begin{cor}
Let $$\overline{P}(z)=z^{n}+az^{n-1}+bz^{n-2}+c,$$ where $ab \neq0$ and $a^{2}=\lambda b$ where $\lambda =4\left(1-\frac{1}{(n-1)^{2}}\right)$, then $\overline{P}(z)$ is a uniqueness polynomial of degree $n\geq 6$.
\end{cor}
\begin{rem}  Let $P_{1}(z)$ be a SUPM. Then $(P_{1}oP_{2})(z)$ is a SUPM if and only if $P_{2}(z)$ is a UPM.
\end{rem}
It is easy to see that if $P(z)$ is strong uniqueness polynomial, then for any non-zero constants $a$ and $c$, $P(af+b)=cP(ag+b)$ gives $(af+b)=(ag+b)$, i.e, $P(az+b)$ is also strong uniqueness polynomial.\par
Already we have discussed in the introductory part of this chapter that finite URSM's are nothing but the set of distinct zeros of some suitable polynomials. Thus at the time of studying uniqueness polynomial, it is general curiosity of the researchers to investigate whether the zero set of the uniqueness polynomial forms unique range set or not.\par For example, Yi (\cite{am15}), Frank-Reinders (\cite{am3}), Banerjee (\cite{ami1}) simultaneously studied the corresponding unique range sets in connection to their uniqueness polynomial.\par
As we have already introduced some new type of uniqueness polynomials in Example \ref{amihii}, we also intend to follow the same direction. Before going to state our concerning results, we recall some well known definitions and results.\par
Fujimoto first observed that \enquote{critical injection property} of polynomials plays crutial role for the set of zeros of a strong uniqueness polynomial to be a unique range set.
\begin{theo 6.D}\label{ami1111} (\cite{am2})
 Let $P(z)$ be a critically injective polynomial of degree $n$ in $\mathbb{C}$ having only simple zeros. Let $P'(z)$ have $k$ distinct zeros and either $k\geq3$ or $k=2$ and $P'(z)$ have no simple zero. Further suppose that $P(z)$ is a SUPM (resp. SUPE). If $S$ is the set of zeros of $P(z)$, then $S$ is a URSM (resp. URSE) whenever $n>2k+6$ (resp. $n>2k+2$) while URSM-IM (resp. URSE-IM) whenever $n>2k+12$ (resp. $n>2k+5$).
 \end{theo 6.D}
\begin{defi} (\cite{am1})
A set $S\subset \mathbb{C}\cup \{\infty\}$ is called a $URSM_{l)}$ (resp. $URSE_{l)}$) if for any two non-constant meromorphic (resp. entire) functions $f$ and $g$, $E_{l)}(S,f)=E_{l)}(S,g)$ implies $f\equiv g$.
\end{defi}
In 2009, with the notion of  $URSM_{l)}$, Bai-Han-Chen (\cite{am1}) improved Theorem 6.D.
\begin{theo 6.E} (\cite{am1})
In addition to the hypothesis of Theorem 6.D, we suppose that $l$ is a positive integer or $\infty$. Let $S$ be the set of zeros of $P(z)$. If
\begin{enumerate}
\item [i)] $l\geq 3$ or $\infty$ and  $n>2k+6$ (resp. $n>2k+2$),
\item [ii)] $l=2$  and  $n>2k+7$ (resp. $n>2k+2$),
\item [iii)] $l=1$  and  $n>2k+10$ (resp. $n>2k+4$),
\end{enumerate}
then $S$ is a $URSM_{l)}$ (resp. $URSE_{l)}$).
\end{theo 6.E}
Recently Banerjee (\cite{ami1}) proved the following result in more general settings.
\begin{theo 6.F} (\cite{ami1})
In addition to the hypothesis of Theorem 6.D, we suppose that $l$ is a positive integer or $\infty$. Let $S$ be the set of zeros of $P(z)$. If
\begin{enumerate}
\item [i)] $l\geq 3$ or $\infty$ and  $\min\{\Theta(\infty;f),\Theta(\infty;g)\}>\frac{6+2k-n}{4}$,
\item [ii)] $l=2$  and  $\min\{\Theta(\infty;f),\Theta(\infty;g)\}>\frac{14+4k-2n}{9}$,
\item [iii)] $l=1$  and  $\min\{\Theta(\infty;f),\Theta(\infty;g)\}>\frac{10+2k-n}{6}$,
\end{enumerate}
then $S$ is a $URSM_{l)}$ ($URSE_{l)}$).
\end{theo 6.F}
We have already seen from Example \ref{amihii} that the polynomial
\begin{equation}\label{equ1} P_{\star}(z)=\sum\limits_{i=0}^{m} \sum\limits_{j=0}^{n} \binom{m}{i}\binom{n}{j}\frac{(-1)^{i+j}}{n+m+1-i-j}z^{n+m+1-i-j}a^{j}b^{i} + c,\end{equation}
is a critically injective strong uniqueness polynomial without any multiple zeros when $m+n\geq 5$, $\max\{m,n\}\geq 3$ and $\min\{m,n\}\geq 2$ with $a\not= b$, $b\not =0$.
Also we have defined
$$Q(z)=\sum\limits_{i=0}^{m} \sum\limits_{j=0}^{n} \binom{m}{i}\binom{n}{j}\frac{(-1)^{i+j}}{n+m+1-i-j}z^{n+m+1-i-j}a^{j}b^{i}$$
and choose
$$c\not\in \{0,-Q(a),-Q(b),-\frac{Q(a)+Q(b)}{2}\}.$$
Thus the following two Theorems are immediate in view of Theorems 6.E and 6.F.
\begin{theo}\label{amibcha} Let $m,n$ be two integers such that  $m+n\geq 5$, $\max\{m,n\}\geq 3$ and $\min\{m,n\}\geq 2$. Take $S_{\star}=\{z~:~P_{\star}(z)=0\}$ where $P_{\star}(z)$ is defined by \ref{equ1} with the already defined choice of $a,b,c$. Further suppose that $l$ is a positive integer or $\infty$. If
\begin{enumerate}
\item [i)] $l\geq 3$ or $\infty$ and  $m+n>9$ (resp. $5$),
\item [ii)] $l=2$  and  $m+n>10$ (resp. $5$),
\item [iii)] $l=1$  and  $m+n>13$ (resp. $7$),
\end{enumerate}
then $S_{\star}$ is a $URSM_{l)}$ (resp. $URSE_{l)})$.
\end{theo}
\begin{theo}\label{saptami} With the suppositions of Theorem \ref{amibcha}, if
\begin{enumerate}
\item [i)] $l\geq 3$ or $\infty$ and  $\min\{\Theta(\infty;f),\Theta(\infty;g)\}>\frac{9-m-n}{4}$,
\item [ii)] $l=2$  and  $\min\{\Theta(\infty;f),\Theta(\infty;g)\}>\frac{20-2m-2n}{9}$,
\item [iii)] $l=1$  and  $\min\{\Theta(\infty;f),\Theta(\infty;g)\}>\frac{13-m-n}{6}$,
\end{enumerate}
then $S_{\star}$ is a $URSM_{l)}$ $(URSE_{l)})$.
\end{theo}
Thus we have obtained URSM with weight $2$ with cardinality atleast $11$ and this is the best result in this direction. Researchers gave different kind of URSM but none of them is able to reduce its cardinality below $11$ (\cite{am3}).
\begin{cor}
Next we consider the polynomial
$$P_{B}(z)=\sum\limits_{i=0}^{m} \binom{m}{i}\frac{(-1)^{i}}{n+m+1-i}z^{n+m+1-i}b^{i} + c,$$
where $bc \neq 0$, $c\not=-b^{n+m+1}\lambda,~-\frac{1}{2}b^{n+m+1}\lambda$ where $\lambda=\sum\limits_{i=0}^{m} \binom{m}{i}\frac{(-1)^{i}}{n+m+1-i}$ and $m+n\geq 5$, $\max\{m,n\}\geq 3$, $\min\{m,n\}\geq 2$.
\end{cor}
By Lemma 2.2 of (\cite{ami1}), we have seen that $P_{B}(b)-P_{B}(0)=b^{n+m+1}\lambda\not=0$, which implies $P_{B}(z)$ is critically injective. Again $P_{B}(0)=c\not=0$ and $P_{B}(b)\not=0$, hence $P_{B}(z)$ have no multiple zeros.\par
Finally, as $P_{B}(b)+P_{B}(0)=b^{n+m+1}\lambda+2c\not=0$ and  $m+n\geq 5$, $\max\{m,n\}\geq 3$, $\min\{m,n\}\geq 2$, by Theorem \ref{amith1}, $P_{B}(z)$ is a strong uniqueness polynomial. Thus similar type conclusions of Theorems \ref{amibcha} and \ref{saptami} are applicable for this polynomial also.
\section{Lemmas}
\begin{lem} \label{amis1} If \bea\label{astomi2} \psi(t)=\lambda(t^{n-1}-A)^{2}-4(t^{n-2}-A)(t^{n}-A)\eea where $\lambda =4\left(1-\frac{1}{(n-1)^{2}}\right)$ and $A\neq 0,~1$, then $\psi(t)=0$ has no multiple roots.
\end{lem}
\begin{proof} Let $F(t):=\psi(e^{t})e^{(1-n)t}~~\text{for}~~t\in \mathbb{C}.$ Then by elementary calculations, we get
\bea\label{astomi} F(t)=(\lambda-4)\left(e^{(n-1)t}+A^{2}e^{-(n-1)t}\right)+4A(e^{t}+e^{-t})-2A\lambda.\eea
Hence
\bea\label{astomi1} F'(t)=\psi'(e^{t})e^{(1-n)t}e^{t}-(n-1)\psi(e^{t})e^{(1-n)t}.\eea
Clearly, if $t=0$, then $\psi(t)\not=0$. Now, if possible assume that $\psi(z_{0})=\psi'(z_{0})=0$. As $z_{0}\neq0$, there exist some $w_{0} \in \mathbb{C}$ such that $z_{0}=e^{w_{0}}$. Thus we have $F(w_{0})=F'(w_{0})=0$. Now, equations (\ref{astomi}) and (\ref{astomi1}) yields
\bea(\lambda-4)\left(e^{(n-1)w_{0}}+A^{2}e^{-(n-1)w_{0}}\right)=-4A(e^{w_{0}}+e^{-w_{0}})+2A\lambda,\eea
and
\bea(\lambda-4)\left(e^{(n-1)w_{0}}-A^{2}e^{-(n-1)w_{0}}\right)=-\frac{4A(e^{w_{0}}-e^{-w_{0}})}{n-1}.\eea
Therefore
\beas 4A^{2}(\lambda-4)^{2} &=& (\lambda-4)^{2}\left((e^{(n-1)w_{0}}+A^{2}e^{-(n-1)w_{0}})^{2}-(e^{(n-1)w_{0}}-A^{2}e^{-(n-1)w_{0}})^{2}\right)\\
&=& \left(-4A(e^{w_{0}}+e^{-w_{0}})+2A\lambda\right)^{2}-\left(-\frac{4A(e^{w_{0}}-e^{-w_{0}})}{n-1}\right)^{2}\\
&=& 4A^{2}\lambda^{2}-32A^{2}\lambda\cosh w_{0}+64A^{2}\cosh^{2} w_{0}-\frac{64A^{2}}{(n-1)^{2}}\sinh^{2}w_{0},
\eeas
i.e. $$(\cosh w_{0})^{2}\left(16-\frac{16}{(n-1)^{2}}\right)-8\lambda \cosh w_{0}+\left(8\lambda-16+\frac{16}{(n-1)^{2}}\right)=0,$$
so, $(\cosh w_{0}-1)^{2}=0$, that is, $\cosh w_{0}=1$ which implies $z_{0}+\frac{1}{z_{0}}=2$. Hence $z_{0}=1$ but $\psi(1)=(1-A)^{2}\neq0$ as $A\neq 1$. Thus our assumption is wrong.
\end{proof}
\begin{lem} (\cite{am3}) \label{amis1.1} If $~\Gamma(t)=\lambda(t^{n-1}-1)^{2}-4(t^{n-2}-1)(t^{n}-1)$ where $\lambda =4\left(1-\frac{1}{(n-1)^{2}}\right)$, then $\Gamma(1)=0$ with multiplicity four. All other zeros of $\Gamma(t)$ are simple.
\end{lem}
\begin{lem}\label{amis2} If $\lambda =4\left(1-\frac{1}{(n-1)^{2}}\right)$, $t \neq 1$ and $A\neq0,~1$, then $\psi(t)=0$ and $t^{n}-A=0$ has no common roots, where $\psi(t)$ is defined by equation (\ref{astomi2}).
\end{lem}
\begin{proof}
If $\psi(t)=0$ and $t^{n}-A=0$ has a common root, then $t^{n-1}-A=0$ and $t^{n}-A=0$. That is, $A=t^{n}=tt^{n-1}=tA$, which is absurd as $A\neq0$ and $t\neq1$.
\end{proof}
\section {Proofs of the theorems}
\par
\begin{proof} [\textbf{Proof of Theorem \ref{amith1} }]
By the given conditions on $P(z)$, we can write
\begin{enumerate}
\item [i)] $P(z)-P(\alpha)= (z-\alpha)^{p}Q_{n-p}(z)$ where $Q_{n-p}(z)$ is a polynomial of degree $(n-p)$, $Q_{n-p}(\alpha)\not=0$ and
\item [ii)] $P(z)-P(\beta)=(z-\beta)^{t}Q(z)$, where  $Q(z)$ is a polynomial of degree $(n-t)$ and $Q(\beta)\not=0$.
\end{enumerate}
As $\alpha$ and $\beta$ are critical points of $P(z)$, we have $P(\alpha)\not=P(\beta)$ and $t,p\geq 2$. Also $P(\alpha)P(\beta)\not=0$, as all zeros of $P(z)$ are simple.\par
Now suppose, for any two non-constant meromorphic functions $f$ and $g$ and a non-zero constant $A \in \mathbb{C}$, that
\bea P(f) = AP(g).\eea
We consider two cases:\\
\textbf{Case-1.} $A \neq 1$.\\
From the assumption of Theorem \ref{amith1}, $P(z)$ is satisfying $\max\{t,p\}+t+p\geq 5+n$ where $t$, $p$ are previously defined.\\
\textbf{Subcase-1.1.} First we assume that $t\geq p$. Thus $2t+p\geq 5+n$. We define
$$F:=\frac{(f-\beta)^{t}Q(f)}{P(\beta)}~~\text{and}~~G:=\frac{(g-\beta)^{t}Q(g)}{P(\beta)}.$$
Thus
\bea\label{e1.1} F=AG+A-1.\eea
So by Lemma \ref{ML}, we have \bea T(r,f)=T(r,g)+O(1).\eea
If  $A \neq \frac{P(\alpha)}{P(\beta)}$, then by applying the Second Fundamental Theorem, we get
\beas && 2nT(r,f)+O(1)=2T(r,F)\\
 &\leq& \overline{N}(r,\infty;F)+\overline{N}(r,0;F)+\overline{N}\left(r,\frac{P(\alpha)}{P(\beta)}-1;F\right)+\overline{N}(r,A-1;F)+S(r,F)\\
&\leq& \overline{N}(r,\infty;f)+\overline{N}(r,\beta;f)+(n-t) T(r,f)+\overline{N}(r,\alpha;f)+(n-p)T(r,f)+\\
&+&\overline{N}(r,0;G)+S(r,f)\\
&\leq& \overline{N}(r,\infty;f)+\overline{N}(r,\beta;f)+(n-t)T(r,f)+\overline{N}(r,\alpha;f)+(n-p)T(r,f)+\\
&+&\overline{N}(r,\beta;g)+(n-t) T(r,g)+S(r,f)\\
&\leq& (3n-2t-p+4)T(r,f)+S(r,f), \eeas
which is a contradiction as $2t+p\geq 5+n$.\\
If $A=\frac{P(\alpha)}{P(\beta)}$, then from equation (\ref{e1.1}) we have
\bea P(\beta)F=P(\alpha)G+\{P(\alpha)-P(\beta)\}.\eea
As $P(\alpha) \pm P(\beta) \not=0$ and $P(\alpha)P(\beta) \not=0$, we have $\frac{P(\alpha)-P(\beta)}{P(\beta)}\not=-\frac{P(\alpha)-P(\beta)}{P(\alpha)}$. Thus in view of the Second Fundamental Theorem, we obtain
\beas && 2nT(r,g)+O(1)=2T(r,G)\\
&\leq& \overline{N}(r,\infty;G)+\overline{N}(r,0;G)+\overline{N}\left(r,-\frac{P(\alpha)-P(\beta)}{P(\alpha)};G\right)+\overline{N}\left(r,\frac{P(\alpha)-P(\beta)}{P(\beta)};G\right)\\
&+&S(r,G)\\
&\leq& \overline{N}(r,\infty;g)+\overline{N}(r,\beta;g)+(n-t)T(r,g)+\overline{N}(r,0;F)+\overline{N}(r,\alpha;g)+(n-p)T(r,g)\\
&+& S(r,g)\\
&\leq& (3+2n-t-p)T(r,g)+\overline{N}(r,\beta;f)+(n-t)T(r,f)+S(r,g)\\
&\leq& (3n-2t-p+4)T(r,g)+S(r,g), \eeas
which is a contradiction as $2t+p\geq 5+n$.\\
\textbf{Subcase-1.2.} Now consider $t< p$. Thus $t+2p\geq 5+n$. We define
$$F:=\frac{(f-\alpha)^{p}Q_{n-p}(f)}{P(\alpha)}~~\text{and}~~G:=\frac{(g-\alpha)^{p}Q_{n-p}(g)}{P(\alpha)}.$$
Proceeding similarly as above, we reach at contradiction. Hence if a critically injective polynomial $P(z)$ with no multiple zeros satisfy $\max\{t,p\}+t+p\geq 5+n$, then $P(f)=AP(g)$ always imply $A=1$.\\
\textbf{Case-2.} $A=1$.Then, as $P(z)$ is a uniqueness polynomial, we have $f\equiv g$.
\end{proof}
\begin{proof} [\textbf{Proof of Theorem \ref{amith41}}]
By the given assumptions, we may write
\begin{enumerate}
\item [i)] $P(z)-P(\gamma)= (z-\xi_{1})^{l_{1}}(z-\xi_{2})^{l_{2}}\ldots(z-\xi_{p})^{l_{p}}$ with $\gamma=\xi_{1}$,
\item [ii)] $P(z)-P(\delta)=(z-\eta_{1})^{m_{1}}(z-\eta_{2})^{m_{2}}\ldots(z-\eta_{q})^{m_{q}}$ with $\delta=\eta_{1},$\\
where $\xi_{i}\not=\xi_{j}$, $\xi_{i}\not=\eta_{j}$  and $\eta_{i}\not=\eta_{j}$ for all $i,j$.
\end{enumerate}
As $P(z)$ has no multiple zeros and $\gamma$, $\delta$ are critical points of $P(z)$, we have $P(\gamma)P(\delta) \not=0$. Also $P(\gamma)\not=P(\delta)$, as $P(z)$ is critically injective.\par
Suppose, for any two non-constant meromorphic functions $f$ and $g$ and for any non-zero complex constant $A$, that
\bea\label{krishna} P(f) = AP(g).\eea
Then by Lemma \ref{ML}, \bea T(r,f)=T(r,g)+O(1)~~\text{and}~~S(r,f)=S(r,g).\eea
We consider two cases:\\
\textbf{Case-1.} $A \neq 1$ and $A = P(\gamma)$. Then $P(\gamma)\neq 1$. Thus we can write
\be \label{e66e dana} P(f)-P(\gamma) = P(\gamma)\left(P(g)-1\right).\ee
\textbf{Subcase-1.1.} $P(\delta)\not=1$. Let $\nu_k$ $(k=1,2,\ldots,l)$ are the $l$ distinct zeros of $(P(g)-1)$. Then by the Second Fundamental Theorem, we get
\beas (l-2)T(r,g) &\leq& \sum\limits_{k=1}^{l}\overline{N}(r,\nu_{k};g)+S(r,g)= \sum\limits_{i=1}^{p}\overline{N}(r,\xi_{i};f)+S(r,g)\\
&\leq& \big(p+o(1)\big)T(r,g),\eeas
which is a contradiction.
\newpage
\textbf{Subcase-1.2.} $P(\delta)=1$. Here we consider two subcases:\par
If $P(\gamma)=-1$, then $$P(f)-P(\delta) = P(\gamma)(P(g)-P(\gamma)).$$
In this case, applying the Second Fundamental Theorem, we obtain
\beas (p-1)T(r,g) &\leq& \overline{N}(r,\infty;g)+\sum\limits_{i=1}^{p}\overline{N}(r,\xi_{i};g)+S(r,g) \\
&\leq& T(r,g)+\sum\limits_{j=1}^{q}\overline{N}(r,\eta_{j};f)+S(r,g) \\
&\leq& (q+1)T(r,g)+S(r,g),\eeas
which leads to a contradiction.\par
If  $P(\gamma)\not=-1$, then  $$P(f)-P(\delta) = P(\gamma)\left(P(g)-\frac{P(\delta)}{P(\gamma)}\right),$$ where $\frac{P(\delta)}{P(\gamma)} \not\in \{1,P(\delta),P(\gamma)\}$. Let $\theta_{l}$ $(l=1,2,\ldots,t)$ be the distinct zeros of $\left(P(z)-\frac{P(\delta)}{P(\gamma)}\right)$. Then by the Second Fundamental Theorem, we get
\beas (t-2)T(r,g) &\leq& \sum\limits_{l=1}^{t}\overline{N}(r,\theta_{l};g)+S(r,g)=\sum\limits_{j=1}^{q}\overline{N}(r,\eta_{j};f)+S(r,f) \\
&\leq& q T(r,g)+S(r,g),\eeas
which is a again contradiction.\\
\textbf{Case-2.} $A \neq 1$ and $A \neq P(\gamma)$. In this case, equation (\ref{krishna}) can be written as
\bea P(f)-AP(\delta) = A(P(g)-P(\delta)).\eea
~~~~If $AP(\delta)\neq P(\gamma)$, then $AP(\delta)\not\in \{P(\gamma),P(\delta)\}$. Let $\zeta_{k}$ $(k=1,2,\ldots,m)$ be the distinct zeros of $(P(z)-AP(\delta))$. Then by the Second Fundamental Theorem, we get
\beas (m-2)T(r,f) &<& \sum\limits_{k=1}^{m}\overline{N}(r,\zeta_{k};f)+S(r,f) \\
&\leq& \sum\limits_{j=1}^{q}\overline{N}(r,\eta_{j};g)+S(r,g) \\
&\leq& q T(r,g)+S(r,g),\eeas
which is not possible.\par
If $AP(\delta)= P(\gamma)$, then $P(\delta) \not=1$ and $P(f)-P(\gamma) = A(P(g)-P(\delta))$. By the Second Fundamental Theorem, we get
\beas (p-2)T(r,f) &<& \sum\limits_{i=1}^{p}\overline{N}(r,\xi_{i};f)+S(r,f) \\
&\leq& \sum\limits_{j=1}^{q}\overline{N}(r,\eta_{j};g)+S(r,g) \\
&\leq& q T(r,g)+S(r,g),\eeas
Proceeding similarly, we get
$$(q-2)T(r,g)\leq p T(r,f)+S(r,f).$$
Since $|p-q|\geq 3$, in either cases, we get a contradiction.\\
Thus $A=1$. Hence, as $P(z)$ is a uniqueness polynomial, we have $f\equiv g$.
\end{proof}
\begin{proof} [\textbf{Proof of Theorem \ref{amith2} }]
Let $f$ and $g$ be two non-constant meromorphic functions such that $\overline{P}(g)=A\overline{P}(f)$, where $A\in\mathbb{C}\setminus\{0\}$. Then by Lemma \ref{ML}, \bea T(r,f)=T(r,g)+O(1)~~\text{and}~~S(r,f)=S(r,g).\eea
By putting $h=\frac{f}{g}$, we have \bea\label{null} g^{2}(h^{n}-A)+ag(h^{n-1}-A)+b(h^{n-2}-A)=0.\eea
If $h$ is a constant function, then as $g$ is non-constant, we get $(h^{n}-A)=(h^{n-1}-A)=(h^{n-2}-A)=0;$ i.e., $A=Ah=Ah^{2}$ which gives $h=1$, and hence $f=g$.\par
Next we consider $h$ as a non-constant. Then \bea \label{amidecem} \left(g+\frac{a}{2}\frac{h^{n-1}-A}{h^{n}-A}\right)^{2}&=& \frac{b\psi(h)}{4(h^{n}-A)^{2}},\eea
where $\psi(t)=\lambda(t^{n-1}-A)^{2}-4(t^{n-2}-A)(t^{n}-A)$.\\
\textbf{Case-1.} $A=1$. Clearly, in view of Lemmas \ref {amis1.1} and \ref {amis2}, from equation (\ref{amidecem}), we get
\bea\label{lt1} \left(g+\frac{a}{2}\frac{h^{n-1}-1}{h^{n}-1}\right)^{2}=\frac{b(h-1)^{4}\prod\limits_{i=1}^{2n-6}(h-\kappa_{i})}{4\{(h-1)\prod\limits_{j=1}^{n-1}(h-\rho_{j})\}^{2}},\eea
where $\kappa_{i}\not =\rho_{j}$ for $i=1,\ldots,2n-6;j=1,\ldots,(n-1)$. Now, by the Second Fundamental Theorem, we get
\beas (3n-9)T(r,h) &\leq& \sum\limits_{i=1}^{2n-6}\ol{N}(r,\kappa_{i};h)+\sum\limits_{j=1}^{n-1}\ol{N}(r,\rho_{j};h)+S(r,h)\\
&\leq& \frac{1}{2}\sum\limits_{i=1}^{2n-6}N(r,\kappa_{i};h)+\sum\limits_{j=1}^{n-1}\ol{N}(r,\rho_{j};h)+S(r,h)\\
&\leq& (2n-4)T(r,h)+S(r,h), \eeas
which is a contradiction for $n\geq 6$.\\
\textbf{Case-2.} $A\neq1$. From equation (\ref{amidecem}), we have
\bea\left(g+\frac{a}{2}\frac{h^{n-1}-A}{h^{n}-A}\right)^{2}=\frac{b\psi(h)}{4(h^{n}-A)^{2}}.\eea
By Lemma \ref{amis1}, $\psi(t)=0$ has $(2n-2)$ distinct zeros, say $\zeta_{i}$ for $i=1,\ldots,2n-2$. So, in view of Lemma \ref{amis2}, and the Second Fundamental Theorem, we get
\beas (2n-4)T(r,h) &\leq& \sum\limits_{i=1}^{2n-2}\ol{N}(r,\zeta_{i};h)+S(r,h)\\
&\leq& \frac{1}{2}\sum\limits_{i=1}^{2n-2}N(r,\zeta_{i};h)+S(r,h)\\
&\leq& (n-1)T(r,h)+S(r,h), \eeas
which is a contradiction when $n\geq 4$. Hence the proof.
\end{proof}
$~~$
\vspace{15 cm}
\\
------------------------------------------------
\\
\textbf{The matter of this chapter has been published in  Adv. Pure Appl. Math., Vol. 8, No. 1, (2017), pp. 1-13.}
\newpage
\chapter{Further results on the
uniqueness of meromorphic functions and their derivative counterpart sharing one or two sets}
\fancyhead[l]{Chapter 7}
\fancyhead[r]{Uniqueness of meromorphic functions and their derivative sharing sets}
\fancyhead[c]{}
\section{Introduction}
\par
We have alreay noticed that, in 1976, Gross (\cite{am4}) extended the study of value sharing by considering set sharing and introduced the notion of unique range set. Further, Gross proved that there exist three finite set $S_{j}$ ($j=1,2,3$) such that any two non-constant entire functions $f$ and $g$ satisfying $E_{f}(S_{j})=E_{g}(S_{j})$ ($j=1,2,3$)  must be identical, and posed the following question:\medbreak
\par
{\bf Question 7.A} {\it Can one find two finite set $S_{j}$ ($j=1,2$)  such that any two non-constant entire functions $f$ and $g$ satisfying $E_{f}(S_{j})=E_{g}(S_{j})$ ($j=1,2$)  must be identical?\par
If the answer to the above Question is affirmative, it would be interesting to know how large both sets would have to be.}\medbreak
\par
In 1997, Fang-Xu (\cite{jj5.1}) and in 1998, Yi (\cite{jj13}) obtained some interesting resluts in realation to the Question 7.A. Below we first provide the results of Yi (\cite{jj13}).
\begin{theo 7.A} (\cite{jj13}) Let $S_{1}=\{z:z^{n}+az^{n-1}+b=0\}$ and $S_{2}=\{0\}$, where $a$, $b$ are non-zero constants such that $z^{n}+az^{n-1}+b=0$ has no repeated root and $n\;(\geq 3)$ be a positive integer. Let $f$ and $g$ be two non-constant entire functions such that $E_{f}(S_{j},\infty)=E_{g}(S_{j},\infty)$ ($j=1,2$), then $f\equiv g$.
\end{theo 7.A}
\begin{theo 7.B} (\cite{jj13}) Let $S_{1}$ and $S_{2}$ are two finite sets such that any two non-constant entire functions $f$ and $g$ satisfying $E_{f}(S_{j},\infty)=E_{g}(S_{j},\infty)$ ($j=1,2$) must be identical, then $\max\{\sharp(S_{1}),\sharp(S_{2})\}\geq 3$, where $\sharp(S)$ denotes the cardinality of the set $S$.
\end{theo 7.B}
But the above theorems are invalid for meromorphic functions. Thus the following question is natural:\medbreak
\par
{\bf Question 7.B} (\cite{jj13.1}, \cite{jj12.1}, \cite{jj13.2}) {\it Can one find two finite sets $S_{j}$ $(j=1,2)$ such that any two non-constant meromorphic functions $f$ and $g$ satisfying $E_{f}(S_{j},\infty)=E_{g}(S_{j},\infty)$ for $j=1,2$  must be identical?}\medbreak
\par
In 1994, Yi (\cite{jj12.2}) proved that there exist two finite sets $S_1$ (with 2 elements) and $S_2$ (with 9 elements) such that any two
non-constant meromorphic functions $f$ and $g$ satisfying $E_{f}(S_{j},\infty)=E_{g}(S_{j},\infty)$ ($j=1,2$) must be identical.\par
In (\cite{ami5}), Li-Yang proved that there exist two finite sets $S_1$ (with 1 element) and $S_2$ (with 15 elements) such that any two non-constant meromorphic functions $f$ and $g$ satisfying $E_{f}(S_{j},\infty)=E_{g}(S_{j},\infty)$ ($j=1,2$) must be identical.\par
 In (\cite{jj5.0}), Fang-Guo proved that there exist two finite sets $S_1$ (with 1 element) and $S_2$ (with 9 elements) such
that any two non-constant meromorphic functions $f$ and $g$ satisfying $E_{f}(S_{j},\infty)=E_{g}(S_{j},\infty)$ ($j=1,2$) must be identical.\par
Also in 2002, Yi (\cite{jj13.1}) proved that there exist two finite sets $S_1$ (with 1 element) and $S_2$ (with 8 elements) such that any two non-constant meromorphic functions $f$ and $g$ satisfying $E_{f}(S_{j},\infty)=E_{g}(S_{j},\infty)$ ($j=1,2$) must be identical.\par
In 2008, Banerjee (\cite{jj2.1}) further improved the result of Yi (\cite{jj13.1}) by relaxing the nature of sharing the range sets by the notion of weighted sharing. He established that there exist two finite sets $S_1$ (with 1 element) and $S_2$ (with 8 elements) such that any two non-constant meromorphic functions $f$ and $g$ satisfying $E_{f}(S_{1},0)=E_{g}(S_{1},0)$ and $E_{f}(S_{2},2)=E_{g}(S_{2},2)$  must be identical.\par
In this context, the natural query would be \emph{whether there exists similar types of unique range sets corresponding to the derivatives of two meromorphic functions}. In this direction, the following uniqueness results have been obtained when the derivatives of meromorphic functions sharing two sets.
\begin{theo 7.C} (\cite{jj5.2, jj12.1}) Let $S_{1}=\{z:z^{n}+az^{n-1}+b=0\}$ and $S_{2}=\{\infty\}$, where $a$, $b$ are non-zero constants such that $z^{n}+az^{n-1}+b=0$ has no repeated root and $n\;(\geq 7)$, $k$ be two positive integers. Let $f$ and $g$ be two non-constant meromorphic functions such that $E_{f^{(k)}}(S_{1},\infty)=E_{g^{(k)}}(S_{1},\infty)$ and $E_{f}(S_{2},\infty)=E_{g}(S_{2},\infty)$, then $f^{(k)}\equiv g^{(k)}$.
\end{theo 7.C}
In 2010, Banerjee-Bhattacharjee (\cite{jj3.2}) improved the above results as follows:
\begin{theo 7.D} (\cite{jj3.2})
Let $S_{i}$ ($i=1,2$) and $k$ be given as in {\em Theorem 7.C}. Let $f$ and $g$ be two non-constant meromorphic functions  such that $E_{f^{(k)}}(S_{1},2)=E_{g^{(k)}}(S_{1},2)$ and $E_{f}(S_{2},1)=E_{g}(S_{2},1)$, then $f^{(k)}\equiv g^{(k)}$.
\end{theo 7.D}
\begin{theo 7.E} (\cite{jj3.2})
Let $S_{i}$ ($i=1,2$) be given as in {\em Theorem 7.C}. Let $f$ and $g$ be two non-constant meromorphic functions  such that $E_{f^{(k)}}(S_{1},3)=E_{g^{(k)}}(S_{1},3)$ and $E_{f}(S_{2},0)=E_{g}(S_{2},0)$, then $f^{(k)}\equiv g^{(k)}$.
\end{theo 7.E}
\begin{theo 7.F} (\cite{jj3.3})  Let $S_{i}$ ($i=1,2$) and $k$ be given as in {\em Theorem 7.C}. Let $f$ and $g$ be two non-constant meromorphic functions  such that $E_{f^{(k)}}(S_{1},2)=E_{g^{(k)}}(S_{1},2)$ and $E_{f}(S_{2},0)=E_{g}(S_{2},0)$, then $f^{(k)}\equiv g^{(k)}$.
\end{theo 7.F}
So far from the above discussions, we see that for the two set sharing problems, the best result has been obtained when one set contain $8$ elements and the other set contain $1$ element. On the other hand, when derivatives of the functions are considered, then the cardinality of one set can further be reduced to $7$.
So it will be natural query whether there can be a single result corresponding to uniqueness of the function sharing two sets which can accommodate the derivative counterpart of the main function as well under relaxed sharing hypothesis with smaller cardinalities than the existing results.\par
In this direction, to improve all the preceding theorems stated so far {\it Theorems 7.C-7.F} in some sense are the goal of this chapter.
\section{Main Results}
\par
Suppose for two positive integers $m$ and $n$, we shall denote by $\widehat{P}(z)$ the following polynomial
\bea\label{jjabcp1}
\widehat{P}(z)=z^{n}-\frac{2n}{n-m}z^{n-m}+\frac{n}{n-2m}z^{n-2m}+c,
\eea
where $c$ is any complex number satisfying $|c|\not=\frac{2m^2}{(n-m)(n-2m)}$ and $c\not=0,-\frac{1-\frac{2n}{n-m}+\frac{n}{n-2m}}{2}$.\par
\begin{theo}\label{jjthB3}
Suppose $n(\geq 1),~m(\geq1),~k(\geq0)$ be three positive integers such that $\gcd\{m,n\}=1$. Further suppose that $\widehat{S}=\{z : \widehat{P}(z)=0\}$ where the polynomial $\widehat{P}(z)$ is defined by (\ref{jjabcp1}). Let $f$ and $g$ be two non-constant meromorphic functions satisfying $E_{f^{(k)}}(\widehat{S},l)=E_{g^{(k)}}(\widehat{S},l)$. If one of the following conditions holds:
\begin{enumerate}
\item [i)] $l\geq2$ and $n>\max\{ 2m+4+\frac{4}{k+1},4m+1\}$,
\item [ii)] $l=1$ and $n> \max\{2m+4.5+\frac{4.5}{k+1},4m+1\}$,
\item [iii)] $l=0$ and  $n> \max\{2m+7+\frac{7}{k+1},4m+1\}$,
\end{enumerate}
then  $f^{(k)} \equiv g^{(k)}$.
\end{theo}
The next theorem focus on the two set sharing problem.
\begin{theo}\label{jjthB5}
Let $n(> 4m+1),~m(\geq1),~k(\geq0)$ be three positive integers satisfying $\gcd\{m,n\}=1$ and $\widehat{S}=\{z : \widehat{P}(z)=0\}$ where the polynomial $\widehat{P}(z)$ is defined by (\ref{jjabcp1}). Let $f$ and $g$ be two non-constant meromorphic functions satisfying $E_{f^{(k)}}(\widehat{S},l)=E_{g^{(k)}}(\widehat{S},l)$ and $E_{f^{(k)}}(0,q)=E_{g^{(k)}}(0,q)$ where $0\leq q < \infty$. If
\begin{enumerate}
\item [i)] $l\geq \frac{3}{2}+\frac{2}{n-2m-1}+\frac{1}{(n-2m)q+n-2m-1}$ and
\item [ii)] $n>2m+\frac{4}{k+1}+\frac{4}{(k+1)(n-2m-1)}+\frac{2}{(k+1)((n-2m)q+n-2m-1)}$,
\end{enumerate}
then $f^{(k)}\equiv g^{(k)}$.
\end{theo}
The following  example shows that for the two set sharing case,  choosing the set $\textsf{S}_{1}$ with one element and $\textsf{S}_{2}$ with two elements, Theorem \ref{jjthB5} ceases to hold.
\begin{exm}\label{jjabc2015.}
Let $\textsf{S}_{1}=\{a\}$ and $\textsf{S}_{2}=\{b,c\}$. Choose $f(z) =p(z)+(b-a)e^{z}$ and $g(z)=q(z)+(-1)^{k}(c-a)e^{-z}$, where $p(z)$ and $q(z)$ are polynomial of degree $k$ with the coefficient of $z^{k}$ in $p(z)$ and $q(z)$ is equal to $\frac{a}{k!}$. Here $E_{f^{(k)}}(\textsf{S}_j)=E_{g^{(k)}}(\textsf{S}_j)$ for $j=1,2$  but $f^{(k)}\not\equiv g^{(k)}$.
\end{exm}
\begin{rem} If we consider $k\geq1$ in Theorem \ref{jjthB5}, then we see that there exists two sets $\textsf{S}_{1}$ (with 1 element) and $\textsf{S}_{2}$ (with $6$ elements) such that when derivatives of any two non-constant meromorphic functions share them with finite weight yields $f^{(k)}\equiv g^{(k)}$, thus improve {\it Theorem 7.F} in the direction of {\it Question 7.B}.
\end{rem}
The next two examples show that specific form of choosing the set $\textsf{S}_{1}$ with five elements and $\textsf{S}_{2}=\{0\}$ for $k\geq1$, Theorem \ref{jjthB5} ceases to hold.
\begin{exm}\label{jjabc2015}
Let $f(z) =\frac{1}{(\sqrt{\alpha \beta \gamma })^{k-1}}e^{\sqrt{\alpha \beta \gamma}\;z}$ and $g(z)=\frac{(-1)^{k}}{(\sqrt{\alpha \beta \gamma})^{k-1}}e^{-\sqrt{\alpha \beta \gamma}\;z}$ ($k\geq 1$) and $\textsf{S}=\{\alpha \sqrt{\beta },\alpha \sqrt{\gamma},\beta \sqrt{\gamma },\gamma \sqrt{\beta },\sqrt(\alpha\beta\gamma)\}$, where $\alpha $, $\beta $ and $\gamma $ are three non-zero distinct complex numbers. Clearly $E_{f^{(k)}}(\textsf{S})=E_{g^{(k)}}(\textsf{S})$ and $E_{f^{(k)}}(0)=E_{g^{(k)}}(0)$  but $f^{(k)}\not\equiv g^{(k)}$.
\end{exm}
\begin{exm} Let $f(z)=\frac{1}{c^{k}}e^{cz}$ and $g(z)=\omega^{4}f(z)$ and $\textsf{S}=\{\omega^{4},\omega^{3},\omega^{2},\omega,1\}$, where $\omega$ is the non-real fifth root of unity and $c$ is a non-zero complex number. Clearly $E_{f^{(k)}}(\textsf{S})=E_{g^{(k)}}(\textsf{S})$ and $E_{f^{(k)}}(0)=E_{g^{(k)}}(0)$  but $f^{(k)}\not\equiv g^{(k)}$.
\end{exm}
\section{Lemmas}
Throughout this chapter, we take
\beas F=-\frac{1}{c}\left(f^{(k)}\right)^{n-2m}\left(\left(f^{(k)}\right)^{2m}-\frac{2n}{n-m}\left(f^{(k)}\right)^{m}+\frac{n}{n-2m}\right),\eeas
\beas G=-\frac{1}{c}\left(g^{(k)}\right)^{n-2m}\left(\left(g^{(k)}\right)^{2m}-\frac{2n}{n-m}\left(g^{(k)}\right)^{m}+\frac{n}{n-2m}\right),\eeas
and $H$ is defined by the equation (\ref{CHB}) where $n(\geq 1)$, $m(\geq 1)$ and $k(\geq 0)$ are integers. Let us also define by $\widehat{T}(r):=\max\{T\left(r,f^{(k)}\right),T\left(r,g^{(k)}\right)\}$ and $\widehat{S}(r):=o(\widehat{T}(r))$.
\begin{lem}\label{jjb0}  The polynomial $$\widehat{P}(z)=z^{n}-\frac{2n}{n-m}z^{n-m}+\frac{n}{n-2m}z^{n-2m}+c $$ is a critically injective polynomial having only simple zeros when $|c|\not=0,\frac{2m^2}{(n-m)(n-2m)}$.
\end{lem}
\begin{proof} We see that $\widehat{P}'(z)=nz^{n-2m-1}(z^{m}-1)^{2}$. Thus $\widehat{P}(z)$ is critically injective, because
\begin{enumerate}
\item [i)] $\widehat{P}(0)=\widehat{P}(\alpha)$ with $\alpha^{m}=1$ implies $\alpha=0$, which is absurd.
\item [ii)] $\widehat{P}(\beta)=\widehat{P}(\gamma)$ with $\beta^{m}=1,\gamma^{m}=1$ implies $\beta^{n}=\gamma^{n}$. Thus $\beta=\gamma$ as $\gcd\{m,n\}=1$.
\end{enumerate}
Next, on contrary, we assume that $\widehat{P}(z)$ has atleast one multiple zero. Then $\widehat{P}(\alpha)=\widehat{P}'(\alpha)=0$ holds for some $\alpha$. Then either $\alpha=0$ or $\alpha^{m}=1$.\par
If $\alpha=0$, then $\widehat{P}(\alpha)=0$; i.e., $c=0$, a contradiction  by assumption on $c$.\par
If  $\alpha^{m}=1$, then $\widehat{P}(\alpha)=0$ implies $\alpha^{n}\left(1-\frac{2n}{n-m}+\frac{n}{n-2m}\right)+c=0;$ i.e., $|c|=\frac{2m^{2}}{(n-m)(n-2m)}$, which is impossible by assumption on $c$. Hence the proof.
\end{proof}
\begin{lem}\label{jjb4} (\cite{jj3.3})
If $F$ and $G$ share $(1,l)$ where $0\leq l<\infty$, then\\
$$\ol{N}(r,1;F)+\ol{N}(r,1;G)-N_{E}^{1}(r,1,F)+(l-\frac{1}{2})\ol{N}_{*}(r,1;F,G)\leq\frac{1}{2}\left(N(r,1;F)+N(r,1;G)\right).$$
\end{lem}
\begin{lem}\label{jjb7} Suppose that $F\not\equiv G$. Further suppose that $f^{(k)}$ and $g^{(k)}$ share $(0,q)$ where $0\leq q<\infty$ and $F$, $G$ share $(1,l)$, then
\beas & &\{(n-2m)q+n-2m-1\}\;\ol N\left(r,0;f^{(k)}\mid\geq q+1\right)\\
&\leq& \ol{N}\left(r,\infty;f^{(k)}\right)+\ol{N}\left(r,\infty;g^{(k)}\right)+\ol{N}_{*}(r,1;F,G)+\widehat{S}(r).\eeas
Similar expressions hold for $g$ also.
\end{lem}
\begin{proof} Define \bea \label{CHP} \Phi &:=& \frac{F^{'}}{F-1}-\frac{G^{'}}{G-1}.\eea
If $\Phi=0$, then by integration, we get \bea F-1=A(G-1),\eea
where $A$ is non-zero constant. Since $F\not\equiv G$, we have $A\not= 1$. Thus $0$ is an e.v.P. of $f^{(k)}$ and $g^{(k)}$ and hence the lemma follows.\par
Next we consider as $\Phi\not=0$. If $z_{0}$ be a zero of $f^{(k)}$ of order $t(\geq q+1)$, then $z_{0}$ is a zero of $F$ of order atleast $(q+1)(n-2m)$ and hence $z_{0}$ is a zero of $\Phi$ of order at least $(q+1)(n-2m)-1$. Thus
\beas & &\{(q+1)(n-2m)-1\}\;\ol N\left(r,0;f^{(k)}\mid\geq q+1\right)\\
&\leq& N(r,0;\Phi) \leq T(r,\Phi)+O(1)\\
&\leq& N(r,\infty;\Phi)+\widehat{S}(r)\\
&\leq& \ol{N}\left(r,\infty;f^{(k)}\right)+\ol{N}\left(r,\infty;g^{(k)}\right)+\ol{N}_{*}(r,1;F,G)+\widehat{S}(r).\eeas
Hence the proof.
\end{proof}
\begin{lem}\label{jjbd2}  If $F\equiv G$ holds for $k\geq0$ and $n\geq 2m+4$, then $f^{(k)}\equiv g^{(k)}$.
\end{lem}
\begin{proof} In view of Lemma \ref{jjb0} and Theorem 5.B, the lemma follows.
\end{proof}
\begin{lem}\label{jjbb3} If $k\geq0$ and $n\geq5$, then $FG\not\equiv 1$.
\end{lem}
\begin{proof} On contrary, we suppose that $FG\equiv 1$. Then \bea\label{jjr0}\left(f^{(k)}\right)^{n-2m}\prod_{i=1}^{2m}\left(\left(f^{(k)}\right)-\gamma_{i}\right)\left(g^{(k)}\right)^{n-2m}\prod_{i=1}^{2m}\left(\left(g^{(k)}\right)-\gamma_{i}\right)=c^{2},\eea
where $\gamma_{i}$ (i=1,2,\ldots,2m) are the roots of the equation $z^{2m}-\frac{2n}{n-m}z^{m}+\frac{n}{n-2m}=0$.\\
Applying Lemma \ref{ML} in (\ref{jjr0}), we have $$T\left(r,f^{(k)}\right)=T\left(r,g^{(k)}\right)+O(1).$$
Let $z_{0}$ be a $\gamma_{i}$ point of $f^{(k)}$ of order $p$. Then $z_{0}$ is a pole of $g$ of order $q$ such that $p=n(1+k)q\geq n$. Thus
\bea \ol{N}\left(r,\gamma_{i};f^{(k)}\right)\leq\frac{1}{n}N\left(r,\gamma_{i};f^{(k)}\right).\eea
Again let $z_{0}$ be a zero of $f^{(k)}$ of order $t$. Then $z_{0}$ is a pole of $g$ of order $s$ such that $(n-2m)t=ns(1+k)$. Thus $t>s(1+k)$ and $2ms(1+k)=(n-2m)(t-s(1+k))\geq(n-2m)$. Consequently $(n-2m)t=ns(1+k)$ gives $t\geq\frac{n}{2m}.$
Hence \bea \ol{N}\left(r,0;f^{(k)}\right)\leq\frac{2m}{n}N\left(r,0;f^{(k)}\right).\eea
Again from equation (\ref{jjr0}), we get
\beas \ol{N}\left(r,\infty;f^{(k)}\right) &\leq& \ol{N}\left(r,0;g^{(k)}\right)+\sum\limits_{i=0}^{2m}\ol{N}\left(r,\gamma_{i};g^{(k)}\right)\\
 &\leq& \frac{2m}{n}N\left(r,0;g^{(k)}\right)+\frac{1}{n}\sum\limits_{i=0}^{2m}N\left(r,\gamma_{i};g^{(k)}\right)\\
 &\leq& \frac{4m}{n}T\left(r,g^{(k)}\right)+O(1).\eeas
Hence using the Second Fundamental Theorem, we get
 \bea\label{jjr1} && 2mT\left(r,f^{(k)}\right)\\
\nonumber  &\leq& \ol{N}\left(r,\infty;f^{(k)}\right)+\ol{N}\left(r,0;f^{(k)}\right)+\sum\limits_{i=0}^{2m}\ol{N}\left(r,\gamma_{i};f^{(k)}\right)+S\left(r,f^{(k)}\right)\\
\nonumber  &\leq& \frac{4m}{n}T\left(r,f^{(k)}\right)+\frac{2m}{n}T\left(r,f^{(k)}\right)+\frac{2m}{n}T\left(r,f^{(k)}\right)+S\left(r,f^{(k)}\right),\eea
which is a contradiction as $n\geq5$. Hence the proof.
\end{proof}
\begin{lem}\label{jjbiku} If $H\equiv 0$, $k\geq 0$ and $n\geq 4m+2$ with $\gcd\{m,n\}=1$, then $f^{(k)}\equiv g^{(k)}$.
\end{lem}
\begin{proof} Given that $H\equiv 0$. On integration, we have
\bea\label{jjpe1.1} F=\frac{AG+B}{CG+D},\eea
where $A,B,C,D$ are constant satisfying $AD-BC\neq 0 $. Thus $F$ and $G$ share $(1,\infty)$. Also by Lemma \ref{ML}, we get
\bea\label{jjpe1.2} T(r,f^{(k)})=T(r,g^{(k)})+\widehat{S}(r).\eea
Now we consider the following cases:\\
\textbf{Case-1.} If $AC\neq0$, then (\ref{jjpe1.1}) can be written as
\bea F-\frac{A}{C}=\frac{BC-AD}{C(CG+D)}.\eea
So, $$\overline{N}\left(r,\frac{A}{C};F\right)=\overline{N}(r,\infty;G).$$
Thus applying Second Fundamental Theorem and (\ref{jjpe1.2}), we get
\beas && n T\left(r,f^{(k)}\right)+O(1)=T(r,F)\\
 &\leq& \overline{N}(r,\infty;F)+\overline{N}(r,0;F)+\overline{N}(r,\frac{A}{C};F)+S(r,F)\\
&\leq& \overline{N}\left(r,\infty;f^{(k)}\right)+\overline{N}\left(r,0;f^{(k)}\right)+2mT\left(r,f^{(k)}\right)+\overline{N}\left(r,\infty;g^{(k)}\right)+S\left(r,f^{(k)}\right)\\
&\leq& \left(2m+1+\frac{2}{k+1}\right)T\left(r,f^{(k)}\right)+S\left(r,f^{(k)}\right),\eeas
which is a contradiction as $n\geq 4m+2$.
\newpage
\textbf{Case-2.} Next suppsoe that $AC=0$. Since $AD-BC\neq0$, so $A=C=0$ never occur.\par
\textbf{Subcase-2.1.} If $A=0$ and $C\neq0$, then $B\neq0$ and (\ref{jjpe1.1}) can be written as
\bea\label{jjms19} F=\frac{1}{\gamma G+\delta},\eea
where $\gamma=\frac{C}{B}$ and $\delta=\frac{D}{B}$.\par If $F$ has no $1$-point, then by the Second Fundamental Theorem, we have
\beas &&T(r,F)\\
&\leq& \overline{N}(r,\infty;F)+\overline{N}(r,0;F)+\overline{N}(r,1;F)+S(r,F)\\
&\leq& \overline{N}\left(r,\infty;f^{(k)}\right)+\overline{N}\left(r,0;f^{(k)}\right)+2mT\left(r,f^{(k)}\right)+S\left(r,f^{(k)}\right)\\
&\leq& \frac{1}{n}\left(2m+1+\frac{1}{k+1}\right)T(r,F)+S(r,F),\eeas
which is a contradiction as $n\geq 4m+2$.\par
Thus $\gamma+\delta=1$ and $\gamma\neq 0$. Hence equation (\ref{jjms19}) becomes
\bea\label{jjme20} F=\frac{1}{\gamma G+1-\gamma}.\eea
So, $$\overline{N}\left(r,0;G+\frac{1-\gamma}{\gamma}\right)=\overline{N}(r,\infty;F).$$
If $\gamma\neq1$, then by the Second Fundamental Theorem, we get
\beas && T(r,G)\\
&\leq& \overline{N}(r,\infty;G)+\overline{N}\left(r,0;G\right)+\overline{N}\left(r,0;G+\frac{1-\gamma}{\gamma}\right)+S(r,G)\\
&\leq& \overline{N}\left(r,\infty;g^{(k)}\right)+\overline{N}\left(r,0;g^{(k)}\right)+2mT\left(r,g^{(k)}\right)+\overline{N}\left(r,\infty;f^{(k)}\right)+S\left(r,g^{(k)}\right)\\
&\leq& \frac{1}{n}\left(2m+1+\frac{2}{k+1}\right)T(r,F)+S(r,F),\eeas
which is a contradiction as $n\geq 4m+2$.\\
Thus $\gamma=1$ and hence $FG\equiv 1$, which is impossible by Lemma \ref{jjbb3}.\par
\textbf{Subcase-2.2.} If $A\neq0$ and $C=0$, then $D\neq0$ and (\ref{jjpe1.1}) can be written as
\bea\label{jjms21} F=\lambda G+\mu,\eea
where $\lambda=\frac{A}{D}$ and $\mu=\frac{B}{D}$.\par
It is obvious that  $F$ has atleast one $1$-point, otherwise proceeding as above, we get a contradiction. Thus $\lambda+\mu=1$ with $\lambda\neq0$ and equation (\ref{jjms21}) becomes
$$F=\lambda G+1-\lambda.$$
So, $$\overline{N}\left(r,0;G+\frac{1-\lambda}{\lambda}\right)=\overline{N}(r,0;F).$$
Define $$\xi:=\frac{1}{c}\left(1-\frac{2n}{n-m}+\frac{n}{n-2m}\right).$$
Thus $F+\xi=\left(f^{(k)}-1\right)^{3}Q_{n-3}\left(f^{(k)}\right)$, where $Q_{n-3}(z)$ is an $(n-3)$-degree polynomial and $Q_{n-3}(1)\not= 0$.\par
If $\lambda \neq1$ and $\frac{1-\lambda}{\lambda}\not=\xi$, then by the Second Fundamental Theorem, we get
\beas &&2T(r,G)\\
&\leq& \overline{N}(r,\infty;G)+\overline{N}(r,0;G)+\overline{N}\left(r,0;G+\frac{1-\lambda}{\lambda}\right)+\overline{N}(r,0;G+\xi)+S(r,G)\\
&\leq& \overline{N}\left(r,\infty;g^{(k)}\right)+\overline{N}\left(r,0;g^{(k)}\right)+2mT\left(r,g^{(k)}\right)+\overline{N}\left(r,0;f^{(k)}\right)+2mT\left(r,f^{(k)}\right)\\
&+& \overline{N}\left(r,1;g^{(k)}\right)+(n-3)T\left(r,g^{(k)}\right)+S\left(r,g^{(k)}\right)\\
&\leq& \frac{1}{n}\left(4m+n+\frac{1}{k+1}\right)T(r,G)+S(r,G),\eeas
which is a contradiction as $n\geq 4m+2$.\par
If $\lambda \neq1$ and $\frac{1-\lambda}{\lambda}=\xi$, then $\lambda G=F-\lambda\xi$ and $\lambda\not=-1$ as $c\not=-\frac{1-\frac{2n}{n-m}+\frac{n}{n-2m}}{2}$.\par
Thus by applying the Second Fundamental Theorem and equation (\ref{jjpe1.2}), we obtain
\beas && 2T(r,F)\\
 &\leq& \overline{N}(r,\infty;F)+\overline{N}(r,0;F)+\overline{N}(r,0;F-\lambda\xi)+\overline{N}(r,0;F+\xi)+S(r,F)\\
&\leq& \overline{N}\left(r,\infty;f^{(k)}\right)+\overline{N}\left(r,0;g^{(k)}\right)+2mT\left(r,g^{(k)}\right)+\overline{N}\left(r,0;f^{(k)}\right)+2mT\left(r,f^{(k)}\right)\\
&+& \overline{N}\left(r,1;f^{(k)}\right)+(n-3)T\left(r,f^{(k)}\right)+S\left(r,g^{(k)}\right)\\
&\leq& \frac{1}{n}\left(4m+n+\frac{1}{k+1}\right)T(r,F)+S(r,F),\eeas
which is a contradiction as $n \geq 4m+2$.\par
Thus $\lambda=1$, hence $F\equiv G$; i.e., using Lemma \ref{jjbd2}, we have $f^{(k)}\equiv g^{(k)}$.
\end{proof}
\section{Proofs of the theorems}
\begin{proof} [\textbf{Proof of Theorem \ref{jjthB3} }]
It is clear that $\overline{N}(r,\infty;f^{(k)})\leq\frac{1}{k+1}N(r,\infty;f^{(k)})$ and
$$F'=-\frac{n}{c}\left(f^{(k)}\right)^{n-2m-1}\left(\left(f^{(k)}\right)^{m}-1\right)^{2}f^{(k+1)},$$
$$G'=-\frac{n}{c}\left(g^{(k)}\right)^{n-2m-1}\left(\left(g^{(k)}\right)^{m}-1\right)^{2}g^{(k+1)}.$$
Also by simple calculations, one can get
\bea \label{jjcb2.5} \overline{N}(r,1;F|=1)=\overline{N}(r,1;G|=1)\leq N(r,\infty;H).\eea
\newpage
\textbf{Case-1.} First we assume that $H \not\equiv 0$.\\
Then by simple calculations, the following inequalities are obvious
\bea\label{jjcb2} && N(r,\infty;H)\\
\nonumber &\leq& \ol{N}(r,0;F|\geq2)+\ol{N}(r,0;G|\geq2)+\ol{N}(r,\infty;F)\\
\nonumber&+& \ol{N}(r,\infty;G)+\ol{N}_{*}(r,1;F,G)+\ol{N}_{0}(r,0;F')+\ol{N}_{0}(r,0;G')\\
\nonumber &\leq& \ol{N}\left(r,0;f^{(k)}\right)+\ol{N}\left(r,0;g^{(k)}\right)+\ol{N}\left(r,1;\left(g^{(k)}\right)^{m}\right)\\
\nonumber&+&\ol{N}\left(r,1;\left(f^{(k)}\right)^{m}\right)+\ol{N}\left(r,\infty;f^{(k)}\right)+\ol{N}\left(r,\infty;g^{(k)}\right)\\
\nonumber&+&\ol{N}_{*}(r,1;F,G)+\ol{N}_{0}\left(r,0;f^{(k+1)}\right)+\ol{N}_{0}\left(r,0;g^{(k+1)}\right),\eea
where $\ol{N}_{0}(r,0;F')$ is the reduced counting function of zeros of $F'$ which is not zeros of $F(F-1)$ and $\ol{N}_{0}\left(r,0;f^{(k+1)}\right)$ is the reduced counting function of zeros of $f^{(k+1)}$ which is not zeros of $f^{(k)}\left(\left(f^{(k)}\right)^m-1\right)$ and $(F-1)$.\par
In view of the Second Main Theorem, (\ref{jjcb2}) and (\ref{jjcb2.5}), we get
\bea\label{jjn2}&& (n+m)\left\{T\left(r,f^{(k)}\right)+T\left(r,g^{(k)}\right)\right\}\\
\nonumber &\leq& \overline{N}\left(r,\infty;f^{(k)}\right)+\overline{N}\left(r,0;f^{(k)}\right)+\overline{N}\left(r,\infty;g^{(k)}\right)+\overline{N}\left(r,0;g^{(k)}\right)\\
\nonumber &+& \overline{N}(r,1;F)+\overline{N}(r,1;G)+\ol{N}\left(r,1;\left(f^{(k)}\right)^{m}\right)+\ol{N}\left(r,1;\left(g^{(k)}\right)^{m}\right)\\
\nonumber&-& N_{0}\left(r,0,f^{(k+1)}\right)-N_{0}\left(r,0,g^{(k+1)}\right)+S\left(r,f^{(k)}\right)+S\left(r,g^{(k)}\right)\\
\nonumber &\leq& 2\left\{\overline{N}\left(r,\infty;f^{(k)}\right)+\overline{N}\left(r,\infty;g^{(k)}\right)\right\}+2\left\{\overline{N}\left(r,0;f^{(k)}\right)+\overline{N}\left(r,0;g^{(k)}\right)\right\}\\
\nonumber &+& 2\left\{\ol{N}\left(r,1;\left(g^{(k)}\right)^{m}\right)+\ol{N}\left(r,1;\left(f^{(k)}\right)^{m}\right)\right\}+\overline{N}(r,1;F)+\overline{N}(r,1;G)\\
\nonumber &-&\overline{N}(r,1;F|=1)+\ol{N}_{*}(r,1;F,G)+S\left(r,f^{(k)}\right)+S\left(r,g^{(k)}\right).
\eea
Thus in view of Lemma \ref{jjb4}, (\ref{jjn2}) can be written as
\bea\label{jjn3.51}&& \left(\frac{n}{2}-m-2\right)\left\{T\left(r,f^{(k)}\right)+T\left(r,g^{(k)}\right)\right\}\\
\nonumber &\leq& \frac{2}{k+1}\left\{N\left(r,\infty;f^{(k)}\right)+N\left(r,\infty;g^{(k)}\right)\right\}\\
\nonumber &+&\left(\frac{3}{2}-l\right)\ol{N}_{*}(r,1;F,G)+S\left(r,f^{(k)}\right)+S\left(r,g^{(k)}\right).
\eea
If $l\geq2$, then from (\ref{jjn3.51}) we get a contradiction when $n>2m+4+\frac{4}{k+1}$.\par
Next, if $l=1$, then
\beas && N_{*}(r,1;F,G)=\ol{N}_{L}(r,1;F)+\ol{N}_{L}(r,1;G)\\
&\leq& \frac{1}{2}\left\{N\left(r,0;f^{(k+1)}~|~f^{(k)}\neq0\right)+N\left(r,0;g^{(k+1)}~|~g^{(k)}\neq0\right)\right\}\\
&\leq&\frac{1}{2}\left\{\overline{N}\left(r,\infty;f^{(k)}\right)+\overline{N}\left(r,0;f^{(k)}\right)+\overline{N}\left(r,\infty;g^{(k)}\right)+\overline{N}\left(r,0;g^{(k)}\right)\right\}\\
&+& S\left(r,f^{(k)}\right)+S\left(r,g^{(k)}\right)\\
&\leq& \frac{1}{2}\left(1+\frac{1}{k+1}\right)\left\{T\left(r,f^{(k)}\right)+T\left(r,g^{(k)}\right)\right\}+S\left(r,f^{(k)}\right)+S\left(r,g^{(k)}\right).
\eeas
Thus (\ref{jjn3.51}) becomes
\bea\label{jjn3.5}&& \left(\frac{n}{2}-m-2-\frac{2}{k+1}\right)\left\{T\left(r,f^{(k)}\right)+T\left(r,g^{(k)}\right)\right\}\\
\nonumber &\leq&  \frac{1}{4}\left(1+\frac{1}{k+1}\right)\left\{T\left(r,f^{(k)}\right)+T\left(r,g^{(k)}\right)\right\}+S\left(r,f^{(k)}\right)+S\left(r,g^{(k)}\right),
\eea
which is a contradiction when $n>2m+4.5+\frac{4.5}{k+1}$.\par
If $l=0$, then
\beas && N_{*}(r,1;F,G)=\ol{N}_{L}(r,1;F)+\ol{N}_{L}(r,1;G)\\
&\leq& N\left(r,0;f^{(k+1)}~|~f^{(k)}\neq0\right)+N\left(r,0;g^{(k+1)}~|~g^{(k)}\neq0\right)\\
&\leq& \overline{N}\left(r,\infty;f^{(k)}\right)+\overline{N}\left(r,0;f^{(k)}\right)+\overline{N}\left(r,\infty;g^{(k)}\right)+\overline{N}\left(r,0;g^{(k)}\right)\\
&+& S\left(r,f^{(k)}\right)+S\left(r,g^{(k)}\right)\\
&\leq& \overline{N}\left(r,\infty;f^{(k)}\right)+\overline{N}\left(r,0;f^{(k)}\right)+\overline{N}\left(r,\infty;g^{(k)}\right)+\overline{N}\left(r,0;g^{(k)}\right)\\
&+& S\left(r,f^{(k)}\right)+S\left(r,g^{(k)}\right)\\
&\leq& \left(1+\frac{1}{k+1}\right)\left\{T\left(r,f^{(k)}\right)+T\left(r,g^{(k)}\right)\right\}+S\left(r,f^{(k)}\right)+S\left(r,g^{(k)}\right).
\eeas
Thus (\ref{jjn3.51}) becomes
\bea\label{jjn3.5}&& \left(\frac{n}{2}-m-2-\frac{2}{k+1}\right)\left\{T\left(r,f^{(k)}\right)+T\left(r,g^{(k)}\right)\right\}\\
\nonumber &\leq&  \frac{3}{2}\left(1+\frac{1}{k+1}\right)\left\{T\left(r,f^{(k)}\right)+T\left(r,g^{(k)}\right)\right\}+S\left(r,f^{(k)}\right)+S\left(r,g^{(k)}\right),
\eea
which is a contradiction when $n>2m+7+\frac{7}{k+1}$.\\
\medbreak
\textbf{Case-2.} Next we assume that $H\equiv 0$. Then in view of the Lemma \ref{jjbiku}, we obtained $f^{(k)}\equiv g^{(k)}$ if $n\geq 4m+2$. Hence the proof.
\end{proof}
\begin{proof} [\textbf{Proof of Theorem \ref{jjthB5} }]
If $H\equiv 0$, then by the Lemma \ref{jjbiku}, we obtained that $f^{(k)}\equiv g^{(k)}$ when $n\geq 4m+2$. Thus we consider $H \not\equiv 0$. Then obviously $F\not\equiv G$. As $f^{(k)}$ and $g^{(k)}$ share $(0,q)$ where $q\geq0$, by simple calculations, we have
\bea\label{jjm1} && N(r,\infty;H)\\
\nonumber &\leq& \ol{N}\left(r,\infty;f^{(k)}\right)+\ol{N}\left(r,\infty;g^{(k)}\right)+\ol{N}\left(r,1;\left(g^{(k)}\right)^{m}\right)\\
\nonumber &+&\ol{N}\left(r,1;\left(f^{(k)}\right)^{m}\right)+\ol{N}_{*}\left(r,0;f^{(k)},g^{(k)}\right)+\ol{N}_{*}(r,1;F,G)\\
\nonumber &+&\ol{N}_{0}\left(r,0;f^{(k+1)}\right)+\ol{N}_{0}\left(r,0;g^{(k+1)}\right),\eea
where $\ol{N}_{0}\left(r,0;f^{(k+1)}\right)$ is the reduced counting function of zeros of $f^{(k+1)}$ which is not zeros of $f^{(k)}\left(\left(f^{(k)}\right)^m-1\right)$ and $(F-1)$.\par
Using the Second Fundamental Theorem, (\ref{jjcb2.5}), (\ref{jjm1}) and Lemma \ref{jjb4}, we get
\bea\label{jjm3}&& \left(\frac{n}{2}-m\right)\left\{T\left(r,f^{(k)}\right)+T\left(r,g^{(k)}\right)\right\}\\
\nonumber &\leq& 2\overline{N}\left(r,0;f^{(k)}\right)+\overline{N}_{*}\left(r,0;f^{(k)},g^{(k)}\right)+2\left\{\overline{N}\left(r,\infty;f^{(k)}\right)+\overline{N}\left(r,\infty;g^{(k)}\right)\right\}\\
\nonumber &+&\left(\frac{3}{2}-l\right)\ol{N}_{*}\left(r,1;F,G\right)+S\left(r,f^{(k)}\right)+S\left(r,g^{(k)}\right)\\
\nonumber &\leq& 2\overline{N}\left(r,0;f^{(k)}\right)+\overline{N}\left(r,0;f^{(k)}~|~\geq q+1\right)+2\left\{\overline{N}\left(r,\infty;f^{(k)}\right)+\overline{N}\left(r,\infty;g^{(k)}\right)\right\}\\
\nonumber &+&\left(\frac{3}{2}-l\right)\ol{N}_{*}\left(r,1;F,G\right)+S\left(r,f^{(k)}\right)+S\left(r,g^{(k)}\right).
\eea
Applying the Lemma \ref{jjb7} in (\ref{jjm3}), we obtain
\bea\label{jjm4} && \left(\frac{n}{2}-m-\frac{2}{k+1}\right)\left\{T\left(r,f^{(k)}\right)+T\left(r,g^{(k)}\right)\right\}\\
\nonumber &\leq& 2\overline{N}\left(r,0;f^{(k)}\right)+\overline{N}\left(r,0;f^{(k)}~|~\geq q+1\right)+\left(\frac{3}{2}-l\right)\ol{N}_{*}\left(r,1;F,G\right)\\
\nonumber &+& S\left(r,f^{(k)}\right)+S\left(r,g^{(k)}\right)\\
\nonumber &\leq& \frac{2}{(k+1)(n-2m-1)}\left\{T\left(r,f^{(k)}\right)+T\left(r,g^{(k)}\right)\right\}\\
\nonumber &+& \frac{1}{(k+1)((n-2m)q+n-2m-1)}\left\{T\left(r,f^{(k)}\right)+T\left(r,g^{(k)}\right)\right\}\\
\nonumber &+& \left(\frac{2}{n-2m-1}+\frac{1}{(n-2m)q+n-2m-1}+\frac{3}{2}-l\right)\ol{N}_{*}\left(r,1;F,G\right)\\
\nonumber &+& S\left(r,f^{(k)}\right)+S\left(r,g^{(k)}\right),
\eea
which leads to a contradiction, if $$n>2m+\frac{4}{k+1}+\frac{4}{(k+1)(n-2m-1)}+\frac{2}{(k+1)((n-2m)q+n-2m-1)},$$ and $$l\geq \frac{3}{2}+\frac{2}{n-2m-1}+\frac{1}{(n-2m)q+n-2m-1}.$$
Hence the proof.
\end{proof}
$~~$
\vspace{4 cm}
\\
------------------------------------------------
\\
\textbf{The matter of this chapter has been published in Jordan J. Math. Stat., Vol. 9, No. 2, (2016), pp. 117-139.}
\newpage
\chapter{On the uniqueness of power of a meromorphic function sharing a set with its $k$-th derivative}
\fancyhead[l]{Chapter 8}
\fancyhead[r]{Uniqueness of power of a meromorphic function with its derivative}
\fancyhead[c]{}
\section{Introduction}
\par
The shared value problems relative to a meromorphic function $f$ with its $k$-th derivative have been widely studied in chapters two, three and four. In this chapter, changing that flavor, we divert our investigations to consider uniqueness problem of power of a meromorphic function with its $k$-th derivative under the aegis of set sharing.\par
The inception of this particular field was due to Rubel-Yang (\cite{br10a}) and the afterwards research on Br\"{u}ck conjecture. In continuation to Br\"{u}ck conjecture, in 1998, Gundersen-Yang (\cite{gunyang}) proved that the conjecture is true when $f$ is entire function of finite order and in 2004, Chen-Shon (\cite{chenshon}) proved that the conjecture is also true for entire function of first order $\rho_{1}(f) <\frac{1}{2}$, but upto now, Br\"{u}ck conjecture is still open. Also, \emph{the corresponding conjecture for meromorphic functions fails}, in general.\par
Yang-Zhang (\cite{ar7.1}) first replaced $f$ by $f^{m}$ in Br\"{u}ck conjecture and proved that the conjecture holds for the function $f^m$, and the order restriction on $f$ is not needed if $m$ is relatively large.
\begin{theo 8.A} (\cite{ar7.1}) Let $f$ be a non-constant meromorphic (resp. entire) function and $m>12~(\text{resp.}~7)$ be an integer. If $f^m$ and $(f^m)'$ share $1$ CM, then $f^m = (f^m)'$ and $f$ assumes the form $f(z) = ce^{\frac{z}{m}}$, where $c$ is a non-zero constant.
\end{theo 8.A}
In this direction, in 2009, Zhang (\cite{ar8}) made further improvement as follows:
\begin{theo 8.B} (\cite{ar8})  Let $f$ be a non-constant entire function; $m$, $k$ be positive integers and $a(z)$ $(\not\equiv 0,\infty)$ be a small function of $f$. Suppose $f^{m} -a$ and $(f^{m})^{(k)}-a$ share the value $0$ CM and $m > k + 4$.
Then $f^m\equiv (f^m)^{(k)}$ and $f$ assumes the form $f(z) = ce^{\frac{\lambda}{m}z}$ where $c$ is a non-zero constant and $\lambda^k = 1$.
\end{theo 8.B}
\begin{theo 8.C} (\cite{ar8})  Let $f$ be a non-constant meromorphic function; $m,~k\in \mathbb{N}$ and $a(z)$ $(\not\equiv 0,\infty)$ be a small function of $f$. Suppose $f^{m} -a$ and $(f^{m})^{(k)}-a$ share the value $0$ CM and $(m-k-1)(m-k-4) > 3k + 6$.
Then $f^m\equiv (f^m)^{(k)}$ and $f$ assumes the form $f(z) = ce^{\frac{\lambda}{m}z}$ where $c$ is a non-zero constant and $\lambda^k = 1$.
\end{theo 8.C}
In the same year, Zhang-Yang (\cite{ar9}) further improved {\it Theorem 8.C} by reducing the lower bound of $m$.
\begin{theo 8.D} (\cite{ar9})  Let $f$ be a non-constant meromorphic function; $m$, $k$ be positive integers and $a(z)$ $(\not\equiv 0,\infty)$ be a small function of $f$. Suppose $f^{m} -a$ and $(f^{m})^{(k)}-a$ share the value $0$ IM and $m > 2k+3+\sqrt{(2k+3)(k+3)}$. Then $f^m\equiv (f^m)^{(k)}$ and $f$ assumes the form $f(z) = ce^{\frac{\lambda}{m}z}$ where $c$ is a non-zero constant and $\lambda^k = 1$.
\end{theo 8.D}
Since then a number of improvements and generalizations have been made on the uniqueness of $f^{m}$ and $(f^{m})^{(k)}$, but none of the researchers were being engaged towards changing of the sharing environment in those results. So the following query is natural:
\begin{ques} If $f^{m}$ and $(f^{m})^{(k)}$ share a set $S$ instead of a value $a(\not=0,\infty)$, then can the conclusion of {\it Theorem 8.B} be obtained?
\end{ques}
The following example shows that the minimum cardinality of such sets is at least three.
\begin{exm} Let $S=\{a,b\}$, where $a$ and $b$ are any two distinct complex numbers and $m\geq1$ be any integer. Let $f(z)=(e^{-z}+a+b)^{\frac{1}{m}}$, where we take the principal branch when $m\geq2$. Then $E_{f^m}(S)=E_{(f^m)'}(S)$ but $f^m\not\equiv (f^m)'$.\end{exm}
\section{Main Results}
\par
For a positive integer $n(\geq 3)$, let $P_{Yi}(z)$ (\cite{jj13}) denotes the following polynomial:
\be\label{are5.1}
P_{Yi}(z)=z^{n}+az^{n-1}+b~~~\text{where}~~ab\not=0~\text{and}~\frac{b}{a^{n}}\not=(-1)^{n}\frac{(n-1)^{(n-1)}}{n^{n}}.
\ee
\begin{theo}\label{arthB2}
Let $n(\geq 4)$, $k(\geq1)$ and $m(\geq k+1)$ be three positive integers. Suppose that $S_{Yi}=\{z : P_{Yi}(z)=0 \}$ where $P_{Yi}(z)$ is defined by (\ref{are5.1}). Let $f$ be a non-constant meromorphic function such that $E_{f^{m}}(S_{Yi},l)=E_{(f^{m})^{(k)}}(S_{Yi},l)$. If
\begin{enumerate}
\item [i)] $l\geq2$ and $(n-2)(2n^{2}l-5nl-3n+l+1)>6(n-1)l$, or
\item [ii)] $l=1 $ and $n\geq 5,$ or
\item [iii)] $l=0$ and  $ n\geq 7,$
\end{enumerate}
then  $f^{m} \equiv (f^{m})^{(k)}$ and $f$ assumes the form $f(z)=ce^{\frac{\zeta}{m}z},$ where $c$ is a non-zero constant and $\zeta^{k}=1$.
\end{theo}
\begin{cor}
Let $n(\geq 4)$, $k(\geq1)$ and $m(\geq k+1)$ be three positive integers. Suppose that $S_{Yi}=\{z : P_{Yi}(z)=0 \}$ where $P_{Yi}(z)$ is defined by (\ref{are5.1}). Let $f$ be a non-constant meromorphic function such that $E_{f^{m}}(S_{Yi},3)=E_{(f^{m})^{(k)}}(S_{Yi},3)$. Then  $f^{m} \equiv (f^{m})^{(k)}$ and $f$ assumes the form $f(z)=ce^{\frac{\zeta}{m}z},$ where $c$ is a non-zero constant and $\zeta^{k}=1$.
\end{cor}
\begin{theo}\label{arthB1} Let $n(\geq 4)$, $k(\geq1)$ and $m(\geq k+1)$ be three positive integers. Suppose that $S_{Yi}=\{z : P_{Yi}(z)=0 \}$ where $P_{Yi}(z)$ is defined by (\ref{are5.1}). Let $f$ be a non-constant entire function such that $E_{f^{m}}(S_{Yi},l)=E_{(f^{m})^{(k)}}(S_{Yi},l)$. If
\begin{enumerate}
\item [i)] $l\geq 1$ and $n\geq 4,$ or
\item [ii)] $l=0$ and  $ n\geq 5,$
\end{enumerate}
then  $f^{m} \equiv (f^{m})^{(k)}$ and $f$ assumes the form $f(z)=ce^{\frac{\zeta}{m}z},$ where $c$ is a non-zero constant and $\zeta^{k}=1$.
\end{theo}
\section{Lemmas}
Define $$R(z):=-\frac{1}{b}(z^{n}+az^{n-1}).$$
Throughout of this chapter, we take $F:=R(f^{m})$, $G:=R((f^{m})^{(k)})$ and $H$  is defined by the equation (\ref{CHB}). Also we define $\breve{T}(r)=T(r,f^m)+T(r,(f^m)^{(k)})$ and $\breve{S}(r)=S(r,f).$
\begin{lem}\label{arabc121} If  $f^m=(f^{m})^{(k)}$  and $m\geq k+1$, then  $f$ takes the form $f(z)=ce^{\frac{\zeta}{m}z},$
where $c$ is a non-zero constant and $\zeta^{k}=1$.
\end{lem}
\begin{proof} We \emph{claim} that $0$ and $\infty$ are the Picard exceptional value of $f$.\par
Because otherwise, if $z_{0}$ is a zero of $f$ of order $t$, then it is zero of $f^m$ and $(f^m)^{(k)}$ of order $mt$ and $(mt-k)$ respectively, which is impossible as $k>0$.\par
Similarly, if $z_{0}$ is a pole of $f$ of order $s$, then it is pole of $f^m$ and $(f^m)^{(k)}$ of order $ms$ and $(ms+k)$ respectively, which is impossible as $k>0$.\par
Thus $f$ takes the form of $f(z)=ce^{\frac{\zeta}{m}z},$ where $c$ is a non-zero constant and $\zeta^{k}=1$.
\end{proof}
\begin{lem}\label{arsr2} If $F\equiv G$, then $f^{m}\equiv (f^{m})^{(k)}$ for any $m\geq 1$, $k \geq 1$ and $n \geq 4$.
\end{lem}
\begin{proof}  Define $$h:=\frac{(f^{m})^{(k)}}{f^{m}}.$$
Then, as $f$ is non-constant, we can write
\bea \label{rkd} f^{m}(h^{n}-1)=-a(h^{n-1}-1).\eea
Let $\lambda_{i}$ $(i=1,2,\ldots,n-1)$ be the non-real distinct zeros of $h^{n}-1=0$. If $h$ is non-constant meromorphic function, then by the Second Main Theorem and Lemma \ref{ML}, we get
\beas (n-3)T(r,h)&\leq& \sum\limits_{i=1}^{n-1}\ol{N}(r,\lambda_{i};h)+S(r,h)\\
&\leq& \ol{N}(r,\infty;f)+S(r,f)\\
&\leq& S(r,f),
\eeas
which is a contradiction as $n\geq 4$. Thus $h$ is constant. As $f$ is non-constant, from(\ref{rkd}), we see that the only possibility of $h$ is $1$. Hence the proof.
\end{proof}
\begin{lem}\label{arl.n.1}
If $F$ and $G$ share $(1,0)$, then
\begin{enumerate}
\item [i)] $\overline{N}_{L}(r,1;F)\leq \overline{N}(r,0;f)+\overline{N}(r,\infty;f)+S(r,f),$
\item [ii)] $\overline{N}_{L}(r,1;G)\leq \overline{N}\left(r,0; \left(f^{m}\right)^{k}\right)+\overline{N}(r,\infty;f)+S(r,f).$
\end{enumerate}
\end{lem}
\begin{proof}
As $F$ and $G$ share $1$ IM, we get
\beas \overline{N}_{L}(r,1;F) &\leq&  N(r,1;F)-\ol{N}(r,1,F)\\
&\leq&  N\left(r,\infty;\frac{(f^{m})'}{f^{m}}\right)+S(r,f)\\
&\leq& \overline{N}(r,0;f^{m})+\overline{N}(r,\infty;f^{m})+S(r,f)\\
&\leq& \overline{N}(r,0;f)+\overline{N}(r,\infty;f)+S(r,f).\eeas
Proceeding similarly, we can get the proof of (ii). Hence the proof.
\end{proof}
\begin{lem}\label{arl.n.1m}
If $F$ and $G$ share $(1,l)$ ($l\geq1$), then
$$\overline{N}_{L}(r,1;F)+\overline{N}_{L}(r,1;G) \leq \frac{1}{l}(\overline{N}(r,0;f)+\overline{N}(r,\infty;f))+S(r,f).$$
\end{lem}
\begin{proof}
As $l\geq 1$, we get
\beas \overline{N}_{L}(r,1;F)+\overline{N}_{L}(r,1;G) &\leq& \overline{N}(r,1;F|\geq l+1)\\
&\leq& \frac{1}{l} \left(N(r,1;F)-\ol{N}(r,1,F)\right)\\
&\leq& \frac{1}{l} N\left(r,\infty;\frac{(f^m)'}{f^{m}}\right)+S(r,f)\\
&\leq& \frac{1}{l}\left(\overline{N}(r,0;f)+\overline{N}(r,\infty;f)\right)+S(r,f).\eeas
Hence the proof.
\end{proof}
\begin{lem}\label{arbc1234} Let $F$ and $G$ share $(1,l)$ and $m\geq k+1$. If $H \not\equiv0$, then
\bea\label{arel.2} && N(r,\infty;H)\\
\nonumber &\leq& \overline{N}(r,\infty;f)+\overline{N}\left(r,0;\left(f^m\right)^{(k)}\right)+\ol{N}_{*}(r,1;F,G)+\overline{N}\left(r,-a\frac{n-1}{n};f^m\right)\\
\nonumber &+& \overline{N}\left(r,-a\frac{n-1}{n};\left(f^m\right)^{(k)}\right)+
\overline{N}_{0}\left(r,0;(f^{m})'\right)+\overline{N}_{0}\left(r,0;\left(f^m\right)^{(k+1)}\right),\eea
where $\overline{N}_{0}\left(r,0;\left(f^{m}\right)'\right)$ denotes the counting function of the zeros of $(f^{m})'$ which are not the zeros of $f^{m}\left(f^{m}+a\frac{n-1}{n}\right)$ and $F-1$. Similarly, $\overline{N}_{0}\left(r,0;(f^{m})^{(k+1)}\right)$ is defined.
\end{lem}
\begin{proof} We note that zeros of $f^{m}$ are the zeros of $(f^{m})^{(k)}$ if $m\geq k+1$. Also, zeros of $F$ and $G$ comes from zeros of $f^{m}(f^{m}+a)$ and $(f^{m})^{(k)}\left((f^{m})^{(k)}+a\right)$ respectively,  and
 $$\ol{N}_{L}\big(r,0;f^{m}\big)+\ol{N}_{L}\left(r,0;\left(f^{m}\right)^{(k)}\right)+\ol{N}_{L}\left(r,0;\left(f^{m}\right)^{(k)}\mid f^{m}\not= 0\right)\leq \ol{N}\left(r,0;\left(f^{m}\right)^{(k)}\right).$$
Rest part of the proof is obvious. So we omit the details.
\end{proof}
\begin{lem}\label{arbc121} Let $F$ and $G$ share $(1,l)$, $m\geq k+1$ and $F \not\equiv G$.
\begin{enumerate}
\item [i)] If $l=0$ and $n\geq 5$, then
\beas\label{arbc111} \ol{N}(r,0;f)\leq \ol{N}\left(r,0;\left(f^m\right)^{(k)}\right) &\leq& \frac{3}{n-4}\overline{N}(r,\infty;f)+\breve{S}(r).\eeas
\item [ii)] If $l\geq 1$ and $n>2+\frac{1}{l}$, then
\beas\label{arbc111} \ol{N}(r,0;f)\leq \ol{N}\left(r,0;\left(f^m\right)^{(k)}\right) &\leq& \frac{l+1}{nl-2l-1}\overline{N}(r,\infty;f)+\breve{S}(r).\eeas
\end{enumerate}
\end{lem}
\begin{proof} Define \bea U:=\left(\frac{F'}{(F-1)}-\frac{G'}{(G-1)}\right).\eea
If $U\equiv 0$, then on integration, we have
\beas F-1=B(G-1).\eeas
But, as $F \not\equiv G$, we have $\ol{N}(r,0;f)=S(r,f)$.\par
Next we consider $U\not\equiv 0$. Let $z_{0}$ be a zero of $f$ of order $t$. Then $z_{0}$ is zero of $F$ and $G$ of order $mt(n-1)$ and  $(m t-k)(n-1)$ respectively. Thus $z_{0}$ is a zero of $U$ of order atleast $\nu=(n-2).$ Hence
\bea\label{apuda} \ol{N}(r,0;f) &\leq& \ol{N}\left(r,0;\left(f^m\right)^{(k)}\right)\\
\nonumber&\leq& \frac{1}{\nu}N(r,0;U) \leq \frac{1}{\nu}N(r,\infty;U)+S(r,f)\\
\nonumber&\leq& \frac{1}{\nu}\left\{\ol{N}_{L}(r,1;F)+\ol{N}_{L}(r,1;G)+\ol{N}_{L}(r,\infty;F)+\ol{N}_{L}(r,\infty;G)\right\}+S(r,f)\\
\nonumber&\leq& \frac{1}{\nu}\left\{\ol{N}_{L}(r,1;F)+\ol{N}_{L}(r,1;G)+\ol{N}(r,\infty;f)\right\}+S(r,f).
\eea
If $l=0$, then using the Lemma \ref{arl.n.1}, we can write (\ref{apuda}) as
\beas && \ol{N}(r,0;f)\leq \ol{N}\left(r,0;\left(f^m\right)^{(k)}\right) \\
&\leq& \frac{1}{\nu}\left\{\overline{N}(r,0;f)+\overline{N}\left(r,0;\left(f^{m}\right)^{(k)}\right)+3\overline{N}(r,\infty;f)\right\}+S(r,f)\\
&\leq& \frac{1}{\nu}\left\{2\overline{N}\left(r,0;\left(f^{m}\right)^{(k)}\right)+3\overline{N}(r,\infty;f)\right\}+\breve{S}(r).
\eeas
Therefore
\beas \ol{N}(r,0;f)\leq \ol{N}\left(r,0;\left(f^m\right)^{(k)}\right) &\leq& \frac{3}{n-4}\overline{N}(r,\infty;f)+\breve{S}(r).\eeas
If $l\geq 1$, then using the Lemma \ref{arl.n.1m}, we can write (\ref{apuda}) as
\beas &&\ol{N}(r,0;f)\leq \ol{N}\left(r,0;\left(f^m\right)^{(k)}\right)\\
&\leq& \frac{1}{\nu}\left\{\frac{1}{l}\left(\overline{N}(r,0;f)+\overline{N}(r,\infty;f)\right)+\overline{N}(r,\infty;f)\right\}+S(r,f).\eeas
Thus
\beas \ol{N}(r,0;f)\leq \ol{N}\left(r,0;\left(f^m\right)^{(k)}\right) &\leq& \frac{l+1}{n l-2l-1}\overline{N}(r,\infty;f)+\breve{S}(r).\eeas
Hence the proof.
\end{proof}
\begin{lem}\label{arbc123} Let $F$ and $G$ share $(1,l)$, $m\geq k+1$ and $F \not\equiv G$.
\begin{enumerate}
\item [i)] If $l=0$ and $n\geq 6$, then $\ol{N}(r,\infty;f)\leq \frac{n-4}{2n^{2}-11n+3}\breve{T}(r)+\breve{S}(r).$
\item [ii)] If $l\geq 1$ and $n\geq 4$, then
\beas &&\ol{N}(r,\infty;f)\leq \frac{(n l-2l-1)l}{(2nl-l-1)(n l-2l-1)-(l+1)^{2}}\breve{T}(r)+\breve{S}(r).\eeas
\end{enumerate}
\end{lem}
\begin{proof} Define \bea V:=\left(\frac{F'}{F(F-1)}-\frac{G'}{G(G-1)}\right).\eea
\textbf{Case-1.} If $V\equiv 0$, then by integration, we get
\beas \left(1-\frac{1}{F}\right)=A\left(1-\frac{1}{G}\right).\eeas
As $f^{m}$ and $(f^{m})^{(k)}$ share $(\infty,0)$ and $F \not\equiv G$, so $\overline{N}(r,\infty;f)=S(r,f).$ Hence the proof.\\
\textbf{Case-2.} Next we consider $V\not\equiv 0$.\par
If $z_{0}$ is a pole of $f$ of order $p$, then it is a pole of $(f^{m})^{(k)}$ of order $(pm+k)$. Thus it is a pole of $F$ and $G$ of order $pmn$ and $(pm+k)n$ respectively. Hences $z_{0}$ is a zero of $\left(\frac{F'}{F-1}-\frac{F'}{F}\right)$ order atleast $(pmn-1)$ and zero of $V$ of order atleast $\lambda=2n-1$.
Thus
\bea\label{maa} \lambda \overline{N}(r,\infty;f) &\leq& N(r,0;V) \leq N(r,\infty;V)+S(r,f)\\
\nonumber&\leq& \overline{N}_{*}(r,1;F,G)+\overline{N}_{*}(r,0;F,G)+\ol{N}(r,0;G~|~F\neq 0)+S(r,f)\\
\nonumber&\leq& \overline{N}_{L}(r,1;F)+\overline{N}_{L}(r,1;G)+\overline{N}(r,0;G)+S(r,f)\\
\nonumber&\leq& \overline{N}_{L}(r,1;F)+\overline{N}_{L}(r,1;G)+\overline{N}\left(r,0;\left(f^{m}\right)^{(k)}\right)+\breve{T}(r)+\breve{S}(r).
\eea
If $l=0$, then using Lemmas \ref{arl.n.1} and \ref{arbc121}, we can write (\ref{maa}) as
\beas && \lambda\overline{N}(r,\infty;f) \\
&\leq& \overline{N}(r,0;f)+\overline{N}\left(r,0;\left(f^{m}\right)^{(k)}\right)+2\overline{N}(r,\infty;f)+\overline{N}\left(r,0;\left(f^{m}\right)^{(k)}\right)+\breve{T}(r)+\breve{S}(r)\\
&\leq& 2\overline{N}(r,\infty;f)+3\overline{N}\left(r,0;\left(f^{m}\right)^{(k)}\right)+\breve{T}(r)+\breve{S}(r)\\
&\leq& \left(2+\frac{9}{n-4}\right)\overline{N}(r,\infty;f)+\breve{T}(r)+\breve{S}(r).
\eeas
Therefore
\beas \ol{N}(r,\infty;f)\leq \frac{n-4}{2n^{2}-11n+3}\breve{T}(r)+\breve{S}(r).\eeas
If $l\geq 1$, then using Lemmas \ref{arl.n.1m} and \ref{arbc121}, we can write (\ref{maa}) as
\beas && \lambda \overline{N}(r,\infty;f) \\
&\leq& \frac{1}{l}\left\{\overline{N}(r,0;f)+\overline{N}(r,\infty;f)\right\}+\overline{N}\left(r,0;\left(f^{m}\right)^{(k)}\right)+\breve{T}(r)+ \breve{S}(r)\\
&\leq& \frac{1}{l}\overline{N}(r,\infty;f)+\left(1+\frac{1}{l}\right)\overline{N}\left(r,0;\left(f^{m}\right)^{(k)}\right)+\breve{T}(r)+ \breve{S}(r)\\
&\leq& \left(\frac{1}{l}+\frac{(l+1)^{2}}{l(n l-2l-1)}\right)\overline{N}(r,\infty;f)+\breve{T}(r)+ \breve{S}(r).
\eeas
As $n\geq4 $ and $l\geq1$, so $(2nl-l-1)(nl-2l-1)-(l+1)^{2}=l[l\{n(2n-5)+1\}-(3n-1)]\geq2l$.\\
Therefore
\beas \ol{N}(r,\infty;f)\leq \frac{(nl-2l-1)l}{(2nl-l-1)(nl-2l-1)-(l+1)^{2}}\breve{T}(r)+\breve{S}(r).\eeas
Hence the proof.
\end{proof}
\begin{lem}\label{arsr1} If $H\equiv 0$, then $F\equiv G$ for  $m\geq k+1$ and $n \geq 4$.
\end{lem}
\begin{proof} By the given assumptions, we see that $F$ and $G$ share $(1,\infty)$ and $(\infty,0)$. Also, integrating $H\equiv0$, we have
\bea\label{arpe1.1} F\equiv \frac{AG+B}{CG+D},\eea
where $A,B,C,D$ are constant satisfying $AD-BC\neq0$. Again by Lemma \ref{ML}
\bea\label{arpe1.2} T(r,f^m)=T\left(r,\left(f^m\right)^{(k)}\right)+S(r,f).\eea
\textbf{Case-1.} First we consider $C\neq 0$.\par
Let $z_{0}$ be a pole of $f$ with multiplicity $t$. Then $z_{0}$ is a pole of $F$ with multiplicity $mnt$, but $z_{0}$ is removable singularity or analytic point of $\frac{AG+B}{CG+D}$, which is not possible as $n$ is non-negative and equation (\ref{arpe1.1}) holds. Thus $\overline{N}(r,\infty;f)=S(r,f)$.\par
\textbf{Subcase-1.1.} If $A\neq0$, then equation (\ref{arpe1.1}) can be written as
\bea\label{six} F-\frac{A}{C}=\frac{BC-AD}{C(CG+D)}.\eea
Thus
$$\overline{N}\left(r,\frac{A}{C};F\right)=\overline{N}(r,\infty;G)=S(r,f).$$
Using the Second Fundamental Theorem, we get
\beas && n T(r,f^{m})+O(1)=T(r,F)\\
&\leq& \overline{N}(r,\infty;F)+\overline{N}(r,0;F)+\overline{N}\left(r,\frac{A}{C};F\right)+S(r,F)\\
&\leq& 2\overline{N}(r,\infty;f)+\overline{N}(r,-a;f^m)+\overline{N}(r,0;f^m)+S(r,f)\\
&\leq& \overline{N}(r,-a;f^m)+\overline{N}(r,0;f^m)+S(r,f),
\eeas
which is a contradiction as $n\geq 3$.\par
\textbf{Subcase-1.2.} If $A=0$, then equation (\ref{arpe1.1}) can be written as
\bea\label{kajal} F=\frac{1}{\gamma G+\delta},\eea
where $\gamma=\frac{C}{B}$ and $\delta=\frac{D}{B}$. Obviously $B\not=0$ and $\gamma\not=0$.\par
If $F$ has no 1-point, then Second Fundamental Theorem yields
\beas && n T(r,f^{m})+O(1)=T(r,F)\\
&\leq& \overline{N}(r,\infty;F)+\overline{N}(r,0;F)+\overline{N}(r,1;F)+S(r,F)\\
&\leq& \overline{N}\left(r,\infty;f^m\right)+\overline{N}\left(r,-a;f^m\right)+\overline{N}\left(r,0;f^m\right)+S(r,f)\\
&\leq& \overline{N}\left(r,-a;f^m\right)+\overline{N}\left(r,0;f^m\right)+S(r,f),\eeas
which is a contradiction as $n\geq 3$.
Thus $\gamma+\delta=1$ and $\gamma\neq0$. Also
\bea\label{kobi}\overline{N}\left(r,0;G+\frac{1-\gamma}{\gamma}\right)=\overline{N}(r,\infty;F)=S(r,f).\eea
If $\gamma\neq1$, then Second Fundamental Theorem, equations (\ref{arpe1.2}) and (\ref{kobi}) yields
\beas && n T\left(\left(f^{m}\right)^{(k)}\right)+O(1)=T(r,G)\\
&\leq& \overline{N}(r,\infty;G)+\overline{N}(r,0;G)+\overline{N}\left(r,0;G+\frac{1-\gamma}{\gamma}\right)+S(r,G)\\
&\leq& \overline{N}\left(r,-a;\left(f^m\right)^{(k)}\right)+\overline{N}\left(r,0;\left(f^m\right)^{(k)}\right)+\breve{S}(r),\eeas
which is a contradiction as $n\geq 3$. Therefore $\gamma=1$. Hence $FG\equiv 1$; that is,
\bea \left(f^{m}\right)^{n-1}\left(f^{m}+a\right)\left(\left(f^{m}\right)^{(k)}\right)^{n-1}\left(\left(f^{m}\right)^{(k)}+a\right)\equiv b^{2}.\eea
As $f^{m}$ and $\left(f^{m}\right)^{(k)}$ share $(\infty,0)$, so $\infty$ is the exceptional values of $f^{m}$. Also as zeros of $F$ is neutralized by poles of G, so $0$ and $-a$ are also exceptional values of $f^{m}$. But this is not possible by the Second Fundamental Theorem. So the case $\gamma=1$ can't occur.
\vspace{.2cm}\\
\textbf{Case-2.} Next we consider $C=0$. Then the equation (\ref{arpe1.1}) can be written as
\bea\label{tra}F=\lambda G+\mu,\eea
where $\lambda=\frac{A}{D}$ and $\mu=\frac{B}{D}$. Clearly $\lambda$, $\mu\not=0$. Also $F$ and $G$ share $(1,\infty)$ and $(\infty,\infty)$.\par
If $z_{0}$ be a pole of $f$ of order $t$, then it is a pole of $F$ of order $mtn$ and pole of $G$ of order $(mt+k)n$. But $F$ and $G$ share poles counting multiplicities. Thus $mtn=(mt+k)n$ but $nk\not= 0$ by assumption on $n,k$. Thus
\begin{equation}\label{iamdon} \ol{N}(r,\infty;f)=S(r,f).
\end{equation}
Proceeding similarly as above, we can see that $F$ has atleast one $1$-point. Thus $\lambda+\mu=1$ with $\lambda\neq0$.\par
\textbf{Subcase-2.1.} Let $\lambda\not=1$.\par
We note that if $m\geq k+1$, then $$\overline{N}\left(r,0;f^m\right)\leq\overline{N}\left(r,0;\left(f^m\right)^{(k)}\right).$$ Thus every zero of $f$ is zero of $F$ as well as $G$. Thus  $\ol{N}(r,0;f)=S(r,f)$, otherwise, $\mu=0$, which is impossible.\par
Using the Second Fundamental Theorem, equations (\ref{arpe1.2}), (\ref{iamdon}), we get
\beas && n T\left(r,f^{m}\right)+O(1)=T(r,F)\\
&\leq& \overline{N}(r,\infty;F)+\overline{N}(r,0;F)+\overline{N}(r,1-\lambda;F)+S(r,F)\\
&\leq& \overline{N}\left(r,0;f^m\right)+\overline{N}\left(r,-a;f^m\right)+\overline{N}\left(r,0;\left(f^m\right)^{(k)}\right)+\overline{N}\left(r,-a;\left(f^m\right)^{(k)}\right)+S(r,f)\\
&\leq& 3  T\left(r,f^{m}\right)+S(r,f),\eeas
which is a contradiction as $n\geq 4$.
\vspace{.2 cm}\par
\textbf{Subcase-2.2.} Let $\lambda=1$. Then $$F\equiv G.$$ Hence the proof.
\end{proof}
\section {Proofs of the theorems}
\par
\begin{proof} [\textbf{Proof of Theorem \ref{arthB2} and \ref{arthB1}}] We consider two cases:\par
\textbf{Case-1.} Let $H \not\equiv 0$. Then clearly $F \not\equiv G$. Also by simple calculations, we see that
$$\overline{N}(r,1;F|=1)=\overline{N}(r,1;G|=1)\leq N(r,\infty;H).$$
Now applying Second Fundamental Theorem and  Lemma \ref{arbc1234}, we get
\bea\label{arpe1.1.11}&& (n+1) \left\{T\left(r,f^m\right)+T\left(r,\left(f^m\right)^{(k)}\right)\right\}\\
\nonumber &\leq& \overline{N}(r,\infty;f^{m})+\overline{N}\left(r,\infty;\left(f^{m}\right)^{(k)}\right)+\overline{N}\left(r,0;f^{m}\right)+\overline{N}\left(r,0;\left(f^{m}\right)^{(k)}\right)\\
\nonumber &+& \overline{N}\left(r,-a\frac{n-1}{n};f^m\right)+\overline{N}\left(r,-a\frac{n-1}{n};\left(f^m\right)^{(k)}\right)+\overline{N}(r,1;F)\\
\nonumber &+& \overline{N}(r,1;G)-N_{0}\left(r,0;(f^m)'\right)-N_{0}\left(r,0;\left(f^m\right)^{(k+1)}\right)+\breve{S}(r)\\
\nonumber &\leq& 3\left\{\overline{N}(r,\infty;f)+\overline{N}\left(r,0;\left(f^m\right)^{(k)}\right)\right\}+2\breve{T}(r)+\overline{N}(r,1;F)\\
\nonumber &+& \overline{N}(r,1;G)-\overline{N}(r,1;F|=1)+\overline{N}_{L}(r,1;F)+\overline{N}_{L}(r,1;G)+\breve{S}(r).
\eea
Thus in view of Lemma \ref{jjb4}, (\ref{arpe1.1.11}) can be written as
\bea\label{arpe1.1.1} \left(\frac{n}{2}-1\right)\breve{T}(r) &\leq& 3\left\{\overline{N}(r,\infty;f)+\overline{N}\left(r,0;\left(f^m\right)^{(k)}\right)\right\}\\
\nonumber &+& \left(\frac{3}{2}-l\right)\left\{\overline{N}_{L}(r,1;F)+\overline{N}_{L}(r,1;G)\right\}+\breve{S}(r).
\eea
\textbf{Subcase-1.1.} If $l\geq2$, then using Lemmas \ref{arbc121} and \ref{arbc123} in (\ref{arpe1.1.1}), we get
\bea\label{arpe1.1.1ma} \left(\frac{n}{2}-1\right)\breve{T}(r) &\leq& \frac{3(n-1)l}{(n-2)l-1}\overline{N}(r,\infty;f)+\breve{S}(r)\\
\nonumber &\leq& \frac{3(n-1)l^{2}}{(2nl-l-1)(n l-2l-1)-(l+1)^{2}}\breve{T}(r)+\breve{S}(r)\\
\label{arpe1.1.1mab} &\leq& \frac{3(n-1)l}{2n^{2}l-5nl-3n+l+1}\breve{T}(r)+\breve{S}(r).
\eea
From the inequality (\ref{arpe1.1.1ma}) (resp. \ref{arpe1.1.1mab}), we get a contradiction if $f$ is entire (resp. meromorphic) function and $n\geq 3$ $\left(\text{resp.}~~(n-2)(2n^{2}l-5nl-3n+l+1)>6(n-1)l\right)$.
\vspace{.21 cm}\\
\textbf{Subcase-1.2.} If $l=1$, then using Lemmas \ref{arl.n.1m}, \ref{arbc121} and \ref{arbc123} in (\ref{arpe1.1.1}), we get
\bea\nonumber  \left(\frac{n}{2}-1\right)\breve{T}(r) &\leq& 3\left\{\overline{N}(r,\infty;f)+\overline{N}\left(r,0;\left(f^m\right)^{(k)}\right)\right\}\\
\nonumber &+& \frac{1}{2}\left\{\overline{N}(r,0;f)+\overline{N}(r,\infty;f)\right\}+\breve{S}(r)\\
\label{arpe1.1.1maa} &\leq& \frac{7(n-1)}{2(n-3)}\overline{N}(r,\infty;f)+\breve{S}(r)\\
\label{arpe1.1.1maab} &\leq& \frac{7(n-1)}{2\{(2n-2)(n-3)-4\}}\breve{T}(r)+\breve{S}(r).
\eea
From the inequality (\ref{arpe1.1.1maa}) (resp. (\ref{arpe1.1.1maab})), we get a contradiction if $f$ is entire (resp. meromorphic ) function and $n\geq 3$
(resp. $n\geq 5$).
\vspace{.31 cm}\\
\textbf{Subcase-1.3.} If $l=0$, then using Lemmas \ref{arl.n.1}, \ref{arbc121} and \ref{arbc123} in (\ref{arpe1.1.1}), we get
\bea \nonumber \left(\frac{n}{2}-1\right)\breve{T}(r)&\leq& 3\left\{\overline{N}(r,\infty;f)+\overline{N}\left(r,0;\left(f^m\right)^{(k)}\right)\right\}\\
\nonumber &+& 3\left\{\overline{N}\left(r,0;\left(f^{m}\right)^{(k)}\right)+\overline{N}(r,\infty;f)\right\}+\breve{S}(r)\\
\label{arpe11.1.1maa} &\leq& \frac{6(n-1)}{(n-4)}\overline{N}(r,\infty;f)+\breve{S}(r)\\
\label{arpe11.1.1maab} &\leq& \frac{6(n-1)}{2n^{2}-11n+3}\breve{T}(r)+\breve{S}(r).
\eea
From the inequality (\ref{arpe11.1.1maa}) (resp. (\ref{arpe11.1.1maab})), we get a contradiction if $f$ is entire (resp. meromorphic) function and $n\geq 3$ (resp. $n\geq 7$).
\vspace{ .31 cm}\par
\textbf{Case-2.} Let $H\equiv 0$. Then by Lemmas \ref{arsr1} and \ref{arsr2}, we have $f^m=(f^{m})^{(k)}$. Then in view of Lemma \ref{arabc121}, we see that $f$ takes the form  $$f(z)=ce^{\frac{\zeta}{m}z},$$
where $c$ is a non-zero constant and $\zeta^{k}=1$. Hence the proof.
\end{proof}
$~~$
\vspace{10.5 cm}
\\
------------------------------------------------
\\
\textbf{The matter of this chapter has been accepted for publication in Journal of the Indian Math. Soc., (2017).}
\newpage
\chapter{Uniqueness of the power of a meromorphic functions with its differential polynomial sharing a set}
\fancyhead[l]{Chapter 9}
\fancyhead[r]{Uniqueness of power of a meromorphic functions with its diff. polynomial}
\fancyhead[c]{}
\section{Introduction}
\par
The famous Nevanlinna's five value theorem,  implies that if two non-constant entire functions $f$ and $g$ on the complex plane share \emph{four distinct finite values} (in ignoring multiplicity), then $f\equiv g$. In a special case of this theorem, Rubel-Yang (\cite{br10a}) first observed that \emph{the number four can be replaced by two} when one consider $g = f'$.\par
But the result of Rubel-Yang (\cite{br10a}) is not, in general, true when we consider the sharing of a set of two elements instead of values.
\begin{exm} Let $S=\{a,b\}$, where $a$ and $b$ are any two distinct complex numbers. Let $f(z)=e^{-z}+a+b$, then $E_{f}(S)=E_{f'}(S)$ but $f\not\equiv f'$.\end{exm}
Thus for the uniqueness of meromorphic function with its derivative counterpart, the cardinality of the sharing set should at least be three. In this direction, in 2003, using normal families, Fang and Zalcman made the first breakthrough.
\begin{theo 9.A} (\cite{mm3}) Let $S=\{0,a,b\}$, where $a,b$ are two non-zero distinct complex numbers satisfying $a^{2} \not= b^{2}$, $a\not = 2b$, $a^{2}-ab+b^{2}\not=0$. If for a non-constant entire function $f$,  $E_{f}(S)=E_{f'}(S)$, then $f\equiv f'$.
\end{theo 9.A}
In 2007, Chang-Fang-Zalcman (\cite{mm2.0}) further extended the above result by considering an arbitrary set having three elements in the following manner:
\begin{theo 9.B} (\cite{mm2.0}) Let $f$ be a non-constant entire function and let $S=\{a,b,c\}$ where $a,b~\text{and}~c$ are three distinct complex numbers. If $E_{f}(S)=E_{f'}(S)$, then either
\begin{enumerate}
\item [i)] $f(z)=Ce^{z}$, or
\item [ii)] $f(z)=Ce^{-z}+\frac{2}{3}(a+b+c)$ and $(2a-b-c)(2b-c-a)(2c-a-b)=0$, or
\item [iii)] $f(z)=Ce^{\frac{-1 \pm i\sqrt{3}}{2}z}+\frac{3 \pm i\sqrt{3}}{6}(a+b+c)$ and $a^2+b^2+c^2-ab-bc-ca=0$,\\
where $C$ is a non-zero constant.
\end{enumerate}
\end{theo 9.B}
In the next year, Chang-Zalcman (\cite{mm2}) replaced the entire function by meromorphic function with at most finitely many simple poles in Theorem 9.A and 9.B.
\begin{theo 9.C} (\cite{mm2}) Let $S=\{0,a,b\}$, where $a,b$ are two non-zero distinct complex numbers. If $f$ is a meromorphic function with at most finitely many poles and $E_{f}(S)=E_{f'}(S)$, then $f\equiv f'$.
\end{theo 9.C}
\begin{theo 9.D} (\cite{mm2}) Let $f$ be a non-constant meromorphic function with at most finitely many simple poles and let $S=\{0,a,b\}$, where $a,b$ are two distinct non-zero complex numbers. If $E_{f}(S)=E_{f'}(S)$, then either
\begin{enumerate}
\item [i)] $f(z)=Ce^{z}$, or
\item [ii)] $f(z)=Ce^{-z}+\frac{2}{3}(a+b)$ and either $(a+b)=0$ or $(2a^{2}-5ab+2b^{2})=0$, or
\item [iii)] $f(z)=Ce^{\frac{-1 \pm i\sqrt{3}}{2}z}+\frac{3 \pm i\sqrt{3}}{6}(a+b)$ and $a^2-ab+b^2=0$,\\
where $C$ is a non-zero constant.
\end{enumerate}
\end{theo 9.D}
In 2011, L\"{u} (\cite{mm6.0}) consider an arbitrary set with three elements in Theorem 9.D and obtained the same result with some additional suppositions.
\begin{theo 9.E} (\cite{mm6.0}) Let $f$ be a non-constant transcendental meromorphic function with at most finitely many simple poles and let $S=\{a,b,c\}$, where $a,b,~\text{and}~c$ are three distinct complex numbers. If $E_{f}(S)=E_{f'}(S)$, then the conclusion of Theorem 9.B holds.
\end{theo 9.E}
So we observed from the above results that the researchers were mainly involved to find the uniqueness of an entire or meromorphic function with its first derivative sharing a set at the expanse of allowing several constraints.
But all were practically tacit about the uniqueness of an entire or meromorphic function with its higher order derivatives.
In 2007, Chang-Fang-Zalcman (\cite{mm2.0}) consider the following example to show that in Theorem 9.B, one can not relax the CM sharing to IM sharing of the set $S$. In other words, when multiplicity is disregarded, the uniqueness result ceases to hold.
\begin{exm} If $S=\{-1,0,1\}$ and $f(z)=\sin z$, then $f$ and $f'$ share the set $S$ in ignoring multiplicity but $f\not\equiv f'$.
\end{exm}
So the following question is natural:
\begin{ques} Does there exist any set which when shared by a meromorphic function together with its higher order derivative or even a power of a meromorphic function together with its differential polynomial, lead to wards the uniqueness?
\end{ques}
To seek the possible answer of the above question is the motivation of this chapter. Before going to state the main result of this chapter, We recall the following definition.
\begin{defi}
Let $k(\geq1),l(\geq1)$ be positive integers and $a_{i}$ ($i=0,1,2,\ldots,k-1$) be complex constants. For a non-constant meromorphic function $f$, we define the differential polynomial in $f$ as  $$L=L(f)=a_{0}\left(f^{(k)}\right)^{l}+a_{1}\left(f^{(k-1)}\right)^{l}+\ldots+a_{k-1}\left(f^{'}\right)^{l}.$$
\end{defi}
\section{Main Result}
\par
Suppose for  an integer $n\geq 3$, we shall denote by $P_{Y}(z)$ (\cite{y98}) the following polynomial
\be\label{mme5.1} P_{Y}(z)=az^{n}-n(n-1)z^{2}+2n(n-2)bz-(n-1)(n-2)b^{2},\ee
where $a$ and $b$ are two non-zero complex numbers satisfying $ab^{n-2}\not=2$.\\ We have from (\ref{mme5.1}) that \bea\label{mme5.2}P_{Y}^{'}(z)&=& naz^{n-1}-2n(n-1)z+2n(n-2)b\\&=&\frac{n}{z}[az^{n}-2(n-1)z^{2}+2(n-2)b z].\nonumber\eea
We note that $P_{Y}^{'}(0)\not= 0$ and so from (\ref{mme5.1}) and (\ref{mme5.2}), we get
$$az^{n}-2(n-1)z^{2}+2(n-2)bz=0.$$
Now at each root of $P_{Y}^{'}(z)=0$, we get
\beas && P_{Y}(z)\\&=&az^{n}-n(n-1)z^{2}+2n(n-2)b z-(n-1)(n-2)b^{2}\\&=& 2(n-1)z^{2}-2(n-2)b z-n(n-1)z^{2}+2n(n-2)b z-(n-1)(n-2)b^{2}\\& =& -(n-1)(n-2)(z-b)^{2}.\eeas
So at a root of $P_{Y}^{'}(z)=0$, $P_{Y}(z)$ will be zero if $P_{Y}^{'}(b)=0$. But $P_{Y}^{'}(b)=nb(ab^{n-2}-2)\not =0$. That means a zero of $P_{Y}^{'}(z)$ is not a zero of $P_{Y}(z)$. Thus zeros of $P_{Y}(z)$ are simple.
\begin{theo}\label{mmthB1} Let $m(\geq1),~n(\geq1)$ and $p(\geq0)$ be three positive integers and $f$ be a non-constant meromorphic function.
Suppose that $S_{Y}=\{z : P_{Y}(z)=0 \}$ and $E_{f^{m}}(S_{Y},p)=E_{L(f)}(S_{Y},p)$. If one of the following conditions holds:
\begin{enumerate}
\item [i)] $2 \leq p < \infty$ and $n > 6+6\frac{\mu+1}{\lambda-2\mu},$
\item [ii)] $p=1$ and $n > \frac{13}{2}+7\frac{\mu+1}{\lambda-2\mu},$
\item [iii)] $p=0$ and $n> 6+3\mu+6\frac{(\mu+1)^{2}}{\lambda-2\mu},$
\end{enumerate}
then  $f^{m} \equiv L(f)$, where $\lambda=\min\{m(n-2)-1,(1+k)l(n-2)-1\}$ and $\mu=\min\{\frac{1}{p},1\}$.
\end{theo}
\begin{cor}\label{mmthB12}
There exists a set $S_{Y}$ with eight (resp. seven) elements such that if a non-constant meromorphic (resp. entire) function $f$ and its $k$-th derivative $f^{(k)}$ satisfying $E_{f}(S_{Y},3)=E_{f^{(k)}}(S_{Y},3)$, then $f\equiv f^{(k)}$.
\end{cor}
The following example shows that for a non-constant entire function, the set $S$ in Theorem \ref{mmthB1} can not be replaced by an arbitrary set containing seven distinct elements.
\begin{exm} For example, we take $f=e^{{\omega}^{^{\frac{1}{2k}}}z}$ and $S=\{0,a\omega, a\sqrt{\omega},a, \frac{a}{\sqrt{\omega}}, \frac{a}{\omega}, \frac{a}{\omega \sqrt{\omega}}\}$, where $\omega$ is the non-real cubic root of unity and $a$ is a non-zero complex number. Then it is easy to verify that $f$ and $f^{(k)}$ share $(S,\infty)$, but $f\not \equiv f^{(k)}$.
\end{exm}
\section{Lemmas}
Define
$$R(z):=\frac{az^{n}}{n(n-1)(z-\alpha_{1})(z-\alpha_{2})},$$ where $\alpha_{1}$ and $\alpha_{2}$ are the distinct roots of the equation
\bea n(n-1)z^{2}-2n(n-2)bz+(n-1)(n-2)b^{2}=0.\eea
Throughout this chapter, we take $F=R(f^{m})$, $G=R(L(f))$ and $H$ is defined by the equation (\ref{CHB}).\par Also we define by $T_{Y}(r):=T(r,f^{m})+T\left(r,L(f)\right)$ and $S_{Y}(r):=S(r,f)+S\left(r,L(f)\right)$.
\begin{lem} (\cite{am11})\label{mmlemma} Let $$Q(z)=(n-1)^{2}(z^{n}-1)(z^{n-2}-1)-n(n-2)(z^{n-1}-1)^{2},$$
then $$Q(z)=(z-1)^{4}\prod\limits_{i=1}^{2n-6}(z-\beta_{i})$$
where $\beta_{i} \in \mathbb{C}\setminus\{0,1\} (i=1,2,\ldots,2n-6)$ are distinct.
\end{lem}
\begin{lem}\label{mml.n.1}
If $F$ and $G$ share $(1,p)$, then
\begin{enumerate}
\item [i)] $\overline{N}_{L}(r,1;F)\leq \mu\left(\overline{N}(r,0;f)+\overline{N}(r,\infty;f)\right)+S(r,f)$,
\item [ii)] $\overline{N}_{L}(r,1;G)\leq \mu\left(\overline{N}\left(r,0; L(f)\right)+\overline{N}(r,\infty;f)\right)+S\left(r,L(f)\right)$,
\end{enumerate}
where $\mu=\min\{\frac{1}{p},1\}.$
\end{lem}
\begin{proof} If $p=0$, then
\beas\overline{N}_{L}(r,1;F) &\leq&  N(r,1;F)-\ol{N}(r,1,F)\\
&\leq&  N\left(r,\infty;\frac{(f^{m})'}{f^{m}}\right)+S(r,f)\\
&\leq& \overline{N}\left(r,0;f^{m}\right)+\overline{N}\left(r,\infty;f^{m}\right)+S(r,f)\\
&\leq& \overline{N}(r,0;f)+\overline{N}(r,\infty;f)+S(r,f).
\eeas
If $p\geq1$, then
\beas\overline{N}_{L}(r,1;F) &\leq& \overline{N}(r,1;F|\geq p+1)\\
&\leq& \frac{1}{p} \left(N(r,1;F)-\ol{N}(r,1,F)\right)\\
&\leq& \frac{1}{p} N\left(r,\infty;\frac{(f^{m})'}{(f^{m})}\right)+S(r,f)\\
&\leq& \frac{1}{p}\left(\overline{N}\left(r,0;f^{m}\right)+\overline{N}\left(r,\infty;f^{m}\right)\right)+S(r,f)\\
&\leq& \frac{1}{p}\left(\overline{N}(r,0;f)+\overline{N}(r,\infty;f)\right)+S(r,f).\eeas
Combining the two cases, we get the proof.
\end{proof}
\begin{lem}\label{mml1.1.1} If $F$ and $G$ share $(1,p)$ and $F \not\equiv G$, then
\bea\label{mmel.1}
\overline{N}(r,\infty;f) &\leq& \frac{\mu+1}{\lambda-2\mu}\left(\overline{N}(r,0;f)+\overline{N}\left(r,0;L(f)\right)\right)+S_{Y}(r),\eea
where $\lambda=\min\left\{m(n-2)-1,(1+k)l(n-2)-1\right\}$ and $\mu=\min\left\{\frac{1}{p},1\right\}$.
\end{lem}
\begin{proof} Define $$V:=\left(\frac{F'}{F(F-1)}-\frac{G'}{G(G-1)}\right).$$
\textbf{Case-1.} If $V\equiv 0$, then on integration, we get $$\left(1-\frac{1}{F}\right)=A\left(1-\frac{1}{G}\right).$$
As $f^{m}$ and $L(f)$ share $(\infty,0)$, so if $\overline{N}(r,\infty;f)\neq S(r,f)$, then $A=1$; i.e., $F=G$, which is impossible. Thus $\overline{N}(r,\infty;f)=S(r,f).$ Hence the lemma.\\
\textbf{Case-2.} Next we assume $V\not\equiv 0$.\par
Let $z_{0}$ be a pole of $f$ of order $t$. Then it is a pole of $L(f)$ of order $(t+k)l$, hence pole of $F$ and $G$ of order $tm(n-2)$ and $(t+k)l(n-2)$ respectively.\par
Thus $z_{0}$ is a zero of $\left(\frac{F'}{F-1}-\frac{F'}{F}\right)$ of order atleast $tm(n-2)-1$ and zero of $V$ of order atleast $\lambda$, where $\lambda=\min\left\{m(n-2)-1,(1+k)l(n-2)-1\right\}$. Thus
\beas &&\lambda\overline{N}(r,\infty;f)\\ &\leq& N(r,0;V) \leq N(r,\infty;V)+S_{Y}(r)\\
&\leq& \overline{N}_{L}(r,1;F)+\overline{N}_{L}(r,1;G)+\overline{N}(r,0;f)+\overline{N}\left(r,0;L(f)\right)+S_{Y}(r)\\
&\leq& \mu\left\{\overline{N}(r,0;f)+\overline{N}(r,\infty;f)+\overline{N}\left(r,0; L(f)\right)+\overline{N}(r,\infty;f)\right\}\\
&+& \overline{N}(r,0;f)+\overline{N}\left(r,0;L(f)\right)+S_{Y}(r).
\eeas
Hence
\beas\overline{N}(r,\infty;f) &\leq& \frac{\mu+1}{\lambda-2\mu}\left(\overline{N}(r,0;f)+\overline{N}\left(r,0;L(f)\right)\right)+S_{Y}(r).\eeas
Hence the proof.
\end{proof}
\begin{lem}\label{mmel.2.1.2} Let $H \not\equiv0$. If $F$ and $G$ share $(1,p)$, then
\bea\label{mmel.2} &&N(r,\infty;H)\\ &\leq& \overline{N}(r,\infty;f)+\overline{N}(r,0;f)+\overline{N}\left(r,0;L\left(f\right)\right)+\overline{N}\left(r,b;f^{m}\right)+\overline{N}\left(r,b;L(f)\right)\nonumber\\
\nonumber&+&\overline{N}_{L}(r,1;F)+\overline{N}_{L}(r,1;G)+\overline{N}_{0}\left(r,0;(f^{m})'\right)+\overline{N}_{0}\left(r,0;\left(L(f)\right)'\right),\eea
where $\overline{N}_{0}\left(r,0;(f^{m})'\right)$ denotes the counting function of the zeros of $(f^{m})'$ which are not the zeros of $f(f^{m}-b)$ and $F-1$.  Similarly, $\overline{N}_{0}\left(r,0;\left(L(f)\right)'\right)$ is defined.
\end{lem}
\begin{proof} We see that
$$\overline{N}(r,\infty;F)\leq\overline{N}(r,\infty;f)+\overline{N}(r,\alpha_{1};f^{m})+\overline{N}(r,\alpha_{2};f^{m}).$$
But simple zeros of $f^{m}-\alpha_{i}$ are not poles of $H$ and multiple zeros of  $f^{m}-\alpha_{i}$ are zeros of $(f^{m})'$. Similar explanation for $G$ is also hold. Thus the rest part of the proof is obvious. So we omit the details.
\end{proof}
\section {Proof of the theorem}
\par
\begin{proof} [\textbf{Proof of Theorem \ref{mmthB1} }] We consider two cases:\\
\textbf{Case-1.} Let $H \not\equiv 0$. Then $F \not\equiv G$. By simple calculations, we get
\beas \overline{N}(r,1;F|=1)=\overline{N}(r,1;G|=1)\leq N(r,\infty;H).\eeas
Now applying the Second Fundamental Theorem and Lemma \ref{mmel.2.1.2}, we obtain
\bea\label{mmpe1.1.1} &&(n+1)T(r,f^{m})\\
\nonumber &\leq& \overline{N}(r,\infty;f)+\overline{N}(r,0;f)+\overline{N}(r,b;f^{m})+ \overline{N}(r,1;F)-N_{0}\left(r,0,(f^{m})'\right)+S_{Y}(r)\\
\nonumber &\leq& 2\left\{\overline{N}(r,\infty;f)+\overline{N}(r,0;f)+\overline{N}(r,b;f^{m})\right\}+\overline{N}\left(r,0;L(f)\right)+\overline{N}\left(r,b;L(f)\right)\\
\nonumber &+& \overline{N}(r,1;F|\geq2)+\overline{N}_{L}(r,1;F)+\overline{N}_{L}(r,1;G)+\overline{N}_{0}\left(r,0;\left(L(f)\right)'\right)+S_{Y}(r).
\eea
\textbf{Subcase-1.1.} If $p\geq2$, then
\bea\label{mmpe1.1.2} &&\overline{N}(r,1;F|\geq2)+\overline{N}_{L}(r,1;F)+\overline{N}_{L}(r,1;G)+\overline{N}_{0}\left(r,0;\left(L(f)\right)'\right)\\
\nonumber &\leq& \overline{N}(r,1;G|\geq2)+\overline{N}(r,1;G|\geq3)+\overline{N}_{0}\left(r,0;\left(L(f)\right)'\right)\\
\nonumber &\leq& N\left(r,0;\left(L(f)\right)'~|~L(f)\neq0\right)+S_{Y}(r)\\
\nonumber &\leq& N\left(r,\infty;\frac{\left(L(f)\right)'}{L(f)}\right)+S_{Y}(r)\\
\nonumber &\leq& \overline{N}\left(r,0;L(f)\right)+\overline{N}(r,\infty;f)+S_{Y}(r).\eea
Thus in view of (\ref{mmpe1.1.2}), we can write (\ref{mmpe1.1.1}) as
\bea\label{mmpe1.1.3} &&(n+1)T(r,f^{m})\\
\nonumber &\leq& 3\overline{N}(r,\infty;f)+2\left\{\overline{N}(r,0;f)+\overline{N}\left(r,0;L(f)\right)+\overline{N}(r,b;f^{m})\right\}\\
\nonumber &+& \overline{N}\left(r,b;L(f)\right)+S_{Y}(r).
\eea
Proceeding similarly, we have the expression for $L(f)$ as
\bea\label{mmpe1.1.4} &&(n+1)T\left(r,L(f)\right)\\
\nonumber&\leq& 3\overline{N}(r,\infty;f)+2\left\{\overline{N}(r,0;f)+\overline{N}\left(r,0;L(f)\right)+\overline{N}\left(r,b;L(f)\right)\right\}\\
\nonumber &+& \overline{N}\left(r,b;f^{m}\right)+S_{Y}(r).
\eea
Adding inequalities (\ref{mmpe1.1.3}) and (\ref{mmpe1.1.4}), we have
\bea\label{mmpe1.1.5} &&(n+1)T_{Y}(r)\\
\nonumber &\leq& 6\overline{N}(r,\infty;f)+4\left\{\overline{N}(r,0;f)+\overline{N}\left(r,0;L(f)\right)\right\}\\
\nonumber &+& 3\left\{\overline{N}\left(r,b;f^{m}\right)+\overline{N}\left(r,b;L(f)\right)\right\}+S_{Y}(r).
\eea
Next applying Lemma \ref{mml1.1.1} in the inequality (\ref{mmpe1.1.5}), we have
\bea\label{mmpe1.1.6} (n-6)T_{Y}(r) &\leq& 6\overline{N}(r,\infty;f)+S_{Y}(r)\\
\nonumber &\leq& 6\frac{\mu+1}{\lambda-2\mu}\left(\overline{N}(r,0;f)+\overline{N}\left(r,0;L(f)\right)\right)+S_{Y}(r)\\
\nonumber  &\leq& 6\frac{\mu+1}{\lambda-2\mu} T_{Y}(r)+S_{Y}(r),
\eea
which is a contradiction as $n> 6+6\frac{\mu+1}{\lambda-2\mu}$.
\newpage                                         
\textbf{Subcase-1.2.} If $p=1$, then
\bea\label{mmpe1.1.7} &&\overline{N}(r,1;F|\geq2)+\overline{N}_{L}(r,1;F)+\overline{N}_{L}(r,1;G)+\overline{N}_{0}\left(r,0;(L(f))'\right)\\
\nonumber &\leq& \overline{N}(r,1;G|\geq2)+\overline{N}(r,1;F|\geq2)+\overline{N}_{0}\left(r,0;(L(f))'\right)\\
\nonumber &\leq& N\left(r,0;\left(L(f)\right)'~|~L(f)\neq0\right)+\frac{1}{2}N\left(r,0;(f^{m})'~|~ f^{m}\not=0\right)+S_{Y}(r)\\
\nonumber &\leq& \overline{N}\left(r,0;L(f)\right)+\overline{N}\left(r,\infty;L(f)\right)+\frac{1}{2}\left\{\overline{N}(r,0;f)+\overline{N}(r,\infty;f)\right\}+S_{Y}(r).
\eea
Applying (\ref{mmpe1.1.7}) in (\ref{mmpe1.1.1}), we get
\bea\label{mmpe1.1.8}&& (n+1)T(r,f^{m})\\
\nonumber&\leq& \frac{7}{2}\overline{N}(r,\infty;f)+\frac{5}{2}\overline{N}(r,0;f)+ 2\overline{N}\left(r,0;L(f)\right)+2\overline{N}(r,b;f^{m})\\
\nonumber &+& \overline{N}(r,b;L(f))+S_{Y}(r).
\eea
Proceeding similarly, we have the expression for $L(f)$ as
\bea\label{mmpe1.1.9} &&(n+1)T(r,L(f))\\
\nonumber&\leq& \frac{7}{2}\overline{N}(r,\infty;f)+\frac{5}{2}\overline{N}(r,0;L(f))+2\overline{N}(r,0;f)+2\overline{N}(r,b;L(f))\\
\nonumber &+& \overline{N}(r,b;f^{m})+S_{Y}(r).
\eea
Adding inequalities (\ref{mmpe1.1.8}) and (\ref{mmpe1.1.9}), we get
\bea\label{mmpe1.1.10}&& (n+1)T_{Y}(r)\\
\nonumber &\leq& 7\overline{N}(r,\infty;f)+\frac{9}{2}\left\{\overline{N}(r,0;f)+\overline{N}(r,0;L(f))\right\}\\
\nonumber &+& 3\left\{\overline{N}(r,b;f^{m})+\overline{N}(r,b;L(f))\right\}+S_{Y}(r).
\eea
Thus applying Lemma \ref{mml1.1.1} in (\ref{mmpe1.1.10}), we get
\bea\label{mmpe1.1.11} \left(n-\frac{13}{2}\right)T_{Y}(r) &\leq& 7\overline{N}(r,\infty;f)+S_{Y}(r)\\
\label{mmpe1.1.12} &\leq& 7\frac{\mu+1}{\lambda-2\mu}\bigg(\overline{N}(r,0;f)+\overline{N}(r,0;L(f))\bigg)+S_{Y}(r)\\
\nonumber  &\leq& 7\frac{\mu+1}{\lambda-2\mu} T_{Y}(r)+S_{Y}(r),\eea
which is a contradiction as $n > \frac{13}{2}+7\frac{\mu+1}{\lambda-2\mu}$.\par
\textbf{Subcase-1.3.} If $p=0$, then applying the Second Fundamental Theorem and Lemma \ref{mmel.2.1.2}, we have
\bea\label{mmpe1.1.12} &&(n+1)T_{Y}(r)\\
\nonumber &\leq& \overline{N}(r,\infty;f^{m})+\overline{N}\left(r,\infty;L(f)\right)+\overline{N}(r,0;f^{m})+\overline{N}\left(r,0;L(f)\right)\\
\nonumber&+&\overline{N}\left(r,b;L(f)\right)+\overline{N}(r,b;f^{m})+\overline{N}(r,1;F)+\overline{N}(r,1;G)\\
\nonumber &-& N_{0}\left(r,0;(f^{m})'\right)-N_{0}\left(r,0;\left(L(f)\right)'\right)+S_{Y}(r)\\
\nonumber&\leq& 3\overline{N}(r,\infty;f)+2\overline{N}(r,0;f)+2\overline{N}\left(r,0;L(f)\right)\\
\nonumber &+& 2\overline{N}(r,b;f^{m})+2\overline{N}\left(r,b;L(f)\right)+\overline{N}(r,1;F)+\overline{N}(r,1;G)\\
\nonumber&-&\overline{N}(r,1;F|=1)+\overline{N}_{L}(r,1;F)+\overline{N}_{L}(r,1;G)+S_{Y}(r).\eea
Again,
\beas&&\overline{N}(r,1;F)+\overline{N}(r,1;G)-\overline{N}(r,1;F|=1)\\
 &\leq& \frac{1}{2}\left\{N(r,1;F)+N(r,1;G)+\overline{N}_{L}(r,1;F)+\overline{N}_{L}(r,1;G)\right\}.
\eeas
Thus using Lemmas \ref{mml.n.1} and \ref{mml1.1.1}, we can write (\ref{mmpe1.1.12}) as
\beas &&(n-6)T_{Y}(r) \\
\nonumber &\leq& 6\overline{N}(r,\infty;f)+3\left\{\overline{N}_{L}(r,1;F)+\overline{N}_{L}(r,1;G)\right\}+S_{Y}(r)\\
\nonumber&\leq& 6\overline{N}(r,\infty;f)+3 \mu \left\{\overline{N}(r,0; L(f))+\overline{N}(r,0;f)+2\overline{N}(r,\infty;f)\right\}+S_{Y}(r)\\
\nonumber&\leq& 6(\mu+1)\overline{N}(r,\infty;f)+3 \mu \left\{\overline{N}(r,0; L(f))+\overline{N}(r,0;f)\right\}+S_{Y}(r)\\
\nonumber&\leq& \left(6\frac{(\mu+1)^{2}}{\lambda-2\mu}+3\mu\right)\left\{\overline{N}(r,0;f)+\overline{N}(r,0;L(f))\right\}+S_{Y}(r)\\
\nonumber&\leq& \left(6\frac{(\mu+1)^{2}}{\lambda-2\mu}+3\mu\right)T_{Y}(r)+S_{Y}(r),
\eeas
which is a contradiction as $n> 6+3\mu+6\frac{(\mu+1)^{2}}{\lambda-2\mu}$.
\vspace{.3 cm}\\
\textbf{Case-2.} Next we consider $H\equiv 0$. Then, on integration, we have
\bea\label{mmpe1.1} F=\frac{AG+B}{CG+D},\eea
where $A,B,C,D$ are constant satisfying $AD-BC\neq 0 $. Thus by Lemma \ref{ML}
\bea\label{mmpe1.2} T(r,f^{m})=T\left(r,L(f)\right)+S_{Y}(r).\eea
From(\ref{mmpe1.1}), it is clear that if $C\neq0$, then $\overline{N}(r,\infty;f)=S(r,f)$ and if $C=0$, then $f^{m}$ and $L(f)$ share $(\infty,\infty)$.\\
\textbf{Subcase-2.1.} Let $AC\neq0$. Then (\ref{mmpe1.1}) can be written as
\bea\label{mocha} F-\frac{A}{C}=\frac{BC-AD}{C(CG+D)}.\eea
Thus using the Second Fundamental Theorem and (\ref{mmpe1.2}), (\ref{mocha}), we have
\beas T(r,F) &\leq& \overline{N}(r,\infty;F)+\overline{N}(r,0;F)+\overline{N}\left(r,\frac{A}{C};F\right)+S(r,F)\\
&\leq& \overline{N}(r,\infty;f)+\overline{N}(r,\alpha_{1};f^{m})+\overline{N}(r,\alpha_{2};f^{m})+\overline{N}(r,0;f)\\
&+& \overline{N}\left(r,\infty;L(f)\right)+\overline{N}\left(r,\alpha_{1};L(f)\right)+\overline{N}\left(r,\alpha_{2};L(f)\right)+S_{Y}(r)\\
&\leq& \frac{5}{n}T(r,F)+S(r,F),\eeas
which is a contradiction as $n> 6$.\\
\textbf{Subcase-2.2.} Let $AC=0$. Then  $A=C=0$ never occur.\\
\textbf{Subsubcase-2.2.1.} Let $A=0$ and $C\neq0$. Then $B\neq0$ and (\ref{mmpe1.1}) can be written as
\bea\label{mocha2} F=\frac{1}{\gamma G+\delta},\eea
where $\gamma=\frac{C}{B}$ and $\delta=\frac{D}{B}$. Then $\gamma\not=0$.
\newpage
If $F$ has no $1$-point, then the Second Fundamental Theorem and (\ref{mmpe1.2}) yields
\beas T(r,F) &\leq& \overline{N}(r,\infty;F)+\overline{N}(r,0;F)+\overline{N}(r,1;F)+S(r,F)\nonumber\\
&\leq& \overline{N}(r,\infty;f)+\overline{N}(r,\alpha_{1};f^{m})+\overline{N}(r,\alpha_{2};f^{m})+\overline{N}(r,0;f)+S_{Y}(r)\\
&\leq& \frac{3}{n}T(r,F)+S(r,F),\eeas
which is impossible as $n> 6$.\par
Thus $\gamma+\delta=1$. Also, it is clear from (\ref{mocha2}) that
$$\overline{N}\left(r,0;G+\frac{1-\gamma}{\gamma}\right)=\overline{N}(r,\infty;F).$$
If $\gamma\neq1$, then using the Second Fundamental Theorem and (\ref{mmpe1.2}), we get
\beas T(r,G) &\leq& \overline{N}(r,\infty;G)+\overline{N}(r,0;G)+\overline{N}\left(r,0;G+\frac{1-\gamma}{\gamma}\right)+S(r,G)\nonumber\\
&\leq& \overline{N}\left(r,\infty;L(f)\right)+\overline{N}\left(r,\alpha_{1};L(f)\right)+\overline{N}\left(r,\alpha_{2};L(f)\right)+\overline{N}\left(r,0;L(f)\right)\\
&+& \overline{N}(r,\infty;f)+\overline{N}(r,\alpha_{1};f^{m})+\overline{N}(r,\alpha_{2};f^{m})+S_{Y}(r)\\
&\leq& \frac{5}{n}T(r,F)+S(r,F),\eeas
which is a contradiction as $n> 6$.\par
Thus $\gamma=1$, consequently $FG\equiv 1$. That is,
\bea f^{mn}\left(L(f)\right)^{n}\equiv\frac{n^{2}(n-1)^{2}}{a^{2}}(f^{m}-\alpha_{1})\left(f^{m}-\alpha_{2}\right)\left(L(f)-\alpha_{1}\right)\left(L(f)-\alpha_{2}\right).\eea
Thus, as $n>2$, $f$ has no pole, hence $\overline{N}(r,\infty;f)=S(r,f)$.\par
Let $z_{0}$ be an $\alpha_{1i}$ point of $f$ of order $s$, where $(\alpha_{1i})^{m}=\alpha_{1}$ ($i=1,2,\ldots,m$). Then $z_{0}$ is a zero of $L(f)$ of order $q$ (say) satisfying $nq =s$. Thus
$$\overline{N}(f,\alpha_{1i};f) \leq \frac{1}{n}{N}(f,\alpha_{1i};f).$$
Similarly, if $(\alpha_{2j})^{m}=\alpha_{2}$ ($j=1,2,\ldots,m$), then  we have
$$\overline{N}(f,\alpha_{2j};f) \leq \frac{1}{n}{N}(f,\alpha_{2j};f).$$
Thus by the Second Fundamental Theorem, we get
\beas (2m-1)T(r,f)&\leq& \overline{N}(r,\infty;f)+\sum\limits_{i=1}^{m}\overline{N}(r,\alpha_{1i};f)+\sum\limits_{j=1}^{m}\overline{N}(r,\alpha_{2j};f)+S_{Y}(r)\nonumber\\
&\leq& \frac{2m}{n}T(r,f)+S_{Y}(r),\eeas
which is impossible as $n > 6$.\\
\textbf{Subsubcase-2.2.2.} Let $A\neq0$ and $C=0$. Then $D\neq0$ and (\ref{mmpe1.1}) can be written as
\bea\label{laligu} F=\lambda G+\mu,\eea
where $\lambda=\frac{A}{D}$ and $\mu=\frac{B}{D}$. Then $\lambda\not=0$.
\newpage
If $F$ has no $1$ point, then, similarly as previous, we get a contradiction.\par
Thus $\lambda+\mu=1$ and from (\ref{laligu}), we get
$$\overline{N}\left(r,0;G+\frac{1-\lambda}{\lambda}\right)=\overline{N}(r,0;F).$$
If $\lambda \neq1$, then using the Second Fundamental Theorem and (\ref{mmpe1.2}), we get
\beas T(r,G) &\leq& \overline{N}(r,\infty;G)+\overline{N}(r,0;G)+\overline{N}\left(r,0;G+\frac{1-\lambda}{\lambda}\right)+S(r,G)\nonumber\\
&\leq& \overline{N}\left(r,\infty;L(f)\right)+\overline{N}\left(r,\alpha_{1};L(f)\right)+\overline{N}\left(r,\alpha_{2};L(f)\right)+\overline{N}\left(r,0;L(f)\right)\\
&+& \overline{N}(r,0;f)+S_{Y}(r)\\
&\leq& \frac{5}{n}T(r,G)+S(r,G),\eeas
which is a contradiction as $n > 6$.\par
Thus $\lambda=1$ and hence $F\equiv G$. By substituting $h=\frac{L(f)}{f^{m}}$ in $F\equiv G$, we get
\bea\label{mmpe1.5} && n(n-1)h^{2}f^{2m}\left(h^{n-2}-1\right)-2n(n-2)bhf^{m}\left(h^{n-1}-1\right)\\
&&~~+(n-1)(n-2)b^{2}\left(h^{n}-1\right)=0.\nonumber \eea
If $h$ is a non-constant meromorphic function, then using lemma \ref{mmlemma}, we get
\beas \left\{n(n-1)hf^{m}\left(h^{n-2}-1\right)-n(n-2)b\left(h^{n-1}-1\right)\right\}^{2}=-n(n-2)b^{2}(h-1)^{4}\prod\limits_{i=1}^{2n-6}(h-\beta_{i}).\eeas
Then applying the Second Fundamental Theorem, we get
\beas && (2n-6)T(r,h)\\ &\leq& \overline{N}(r,\infty;h)+\overline{N}(r,0;h)+\sum\limits_{i=1}^{2n-6}\overline{N}(r,0;h-\beta_{i})+S(r,h)\nonumber \\
&\leq& \overline{N}(r,\infty;h)+\overline{N}(r,0;h)+\frac{1}{2}\sum\limits_{i=1}^{2n-6}N(r,0;h-\beta_{i})+S(r,h)\\
&\leq& (n-1) T(r,h)+S(r,h), \eeas
which is a contradiction as $n>6$.\par
Thus $h$ must be a constant. Hence, as $f$ is non-constant and $b\neq0$, we get from equation (\ref{mmpe1.5}) that $$(h^{n-2}-1)=0,~(h^{n-1}-1)=0~\text{and}~(h^{n}-1)=0.$$
Therefore $h=1$; i.e, $f^{m}=L(f).$ Hence the proof.
\end{proof}
$~~$
\vspace{1.2 cm}
\\
------------------------------------------------
\\
\textbf{The matter of this chapter has been published in Math. Morav., Vol. 20, No. 2, (2016), pp. 1-14.}
\newpage
\chapter*{\textbf{List of Publications}}
\addcontentsline{toc}{chapter}
{List of Publications}
\fancyhead[r]{Bibliography}
\fancyhead[l]{}
\fancyhead[c]{}
\begin{enumerate}
\item A. Banerjee and B. Chakraborty, A new type of unique range set with deficient values, \emph{\textbf{Afrika Matematika}}, 26(7-8) (2015), 1561-1572 (\textbf{SCOPUS} indexed).
\item A. Banerjee and B. Chakraborty, Further investigations on a question of Zhang and L\"{u}, \emph{\textbf{Annales Universitatis Paedagogicae Cracoviensis Studia Mathematica}}, 14 (2015), 102-119.
\item A. Banerjee and B. Chakraborty, On the generalizations of Br\"{u}ck conjecture, \emph{\textbf{Communications of Korean Mathematical Society}}, 32(2) (2016), 311-327 (\textbf{SCOPUS} indexed).
\item A. Banerjee and B. Chakraborty, Further results on the uniqueness of meromorphic functions and their derivative counterpart sharing one or two sets, \emph{\textbf{Jordan Journal of Mathematics and Statistics}}, 9(2) (2016), 117-139.
\item A. Banerjee and B. Chakraborty, Uniqueness of the power of a meromorphic functions with its differential polynomial sharing a set, \emph{\textbf{Mathematica Moravica}}, 20(2) (2016), 1-14.
\item A. Banerjee and B. Chakraborty, Some further study on Br\"{u}ck conjecture, \emph{\textbf{An. S\c{t}iin\c{t}. Univ. Al. I. Cuza Ia\c{s}i Mat. (N.S.)}},  62(2, f.2) (2016), 501-511 (\textbf{SCOPUS} indexed).
\item A. Banerjee and B. Chakraborty, On some sufficient conditions of the strong uniqueness polynomials, \emph{\textbf{Advances in Pure and Applied Mathematics}}, 8(1) (2017), 1-13 (\textbf{SCOPUS} indexed).
\item A. Banerjee and B. Chakraborty, On the uniqueness of power of a meromorphic function sharing a set with its $k$-th derivative, \emph{\textbf{ Journal of the Indian Mathematical Society}}, Accepted (2017), (\textbf{SCOPUS} indexed).
\end{enumerate}
\newpage

\end{document}